\documentclass[11pt, leqno]{amsart}
\usepackage[utf8]{inputenc}
\usepackage{geometry}
\usepackage{etoolbox}
\usepackage{lipsum}
\usepackage[english]{babel}
\usepackage{csquotes}
\usepackage{mathrsfs}
\usepackage{mathtools}
\usepackage{xcolor}
\usepackage{yfonts}
\usepackage{thmtools}
\usepackage[makeroom]{cancel}
\usepackage[toc,page]{appendix}
\usepackage{tocvsec2}

\usepackage[breaklinks=true]{hyperref}
\usepackage{cleveref}

\newcommand{\supp}[0]{\mathrm{supp}}
\newcommand{\conv}[0]{\mathrm{conv}}
\newcommand{\height}[0]{\mathrm{ht}}
\newcommand{\charac}[0]{\mathrm{char}}
\newcommand{\wt}[0]{\mathrm{wt}}
\newcommand{\ad}[0]{\mathrm{ad}}

\usepackage{amsmath, amsthm, amssymb}

\theoremstyle{plain}
\newtheorem{theorem}{Theorem}[section]
\newtheorem{lemma}[theorem]{Lemma}
\newtheorem{prop}[theorem]{Proposition}
\newtheorem{cor}[theorem]{Corollary}
\newtheorem{thmx}{Theorem}[section]
\newtheorem{question}[theorem]{Question}

\theoremstyle{definition}
\newtheorem{definition}[theorem]{Definition}
\newtheorem{example}[theorem]{Example}
\newtheorem*{p-psp}{parabolic-PSP}

\newtheorem{rem}[theorem]{Remark}
\newtheorem{observation}[theorem]{Observation}
\newtheorem*{claim}{Claim}
\newtheorem*{abbr}{Abbreviation}
\newtheorem*{convn}{Conventions}

\numberwithin{equation}{section}

\newtheoremstyle{problem}{5pt}{5pt}{}{}{\normalfont}{\textbf{.}}{.5em}{}
\theoremstyle{problem}
\newtheorem{problem}[theorem]{\textbf{Problem}}

\begin{document}

	\title[Weak faces of weights of highest weight modules and root systems]{Weak faces of highest weight modules\\ and root systems}
	
	\author{G. Krishna Teja}
	\begin{abstract}
		Chari and Greenstein [\textit{Adv.\ Math.}\ 2009] introduced certain combinatorial subsets of the roots of a finite-dimensional simple Lie algebra $\mathfrak{g}$, which were important in studying Kirillov--Reshetikhin modules over $U_q(\widehat{\mathfrak{g}})$ and their specializations. Later, Khare [\textit{J.\ Algebra.}\ 2016] studied these subsets for large classes of highest weight $\mathfrak{g}$-modules (in finite type), under the name of weak-$\mathbb{A}$-faces \big(for a subgroup $\mathbb{A}$ of $(\mathbb{R},+)$\big), and more generally, $(\{ 2 \}; \{ 1, 2 \})$-closed subsets. These notions extend and unify the faces of Weyl polytopes as well as the aforementioned combinatorial subsets.
		
		In this paper, we consider these `discrete' notions for an arbitrary Kac--Moody Lie algebra $\mathfrak{g}$, in four distinguished settings: (a) the weights of an arbitrary highest weight $\mathfrak{g}$-module $V$; (b) the convex hull of the weights of $V$; (c) the weights of the adjoint representation; (d) the roots of $\mathfrak{g}$. For (a) \big(respectively, (b)\big) for all highest weight $\mathfrak{g}$-modules $V$, we show that the weak-$\mathbb{A}$-faces and $(\{ 2 \}; \{ 1, 2 \})$-closed subsets agree, and equal the sets of weights on exposed faces (respectively, equal the exposed faces) of the convex hull of weights $\conv_{\mathbb{R}}\wt V$. This completes the partial progress of Khare in finite type, and is novel in infinite type. Our proofs are type-free and self-contained.
		
		For (c) and (d) involving the root system, we similarly achieve complete classifications. For all Kac--Moody $\mathfrak{g}$---interestingly, other than $\mathfrak{sl}_3(\mathbb{C}), \widehat{\mathfrak{sl}_3(\mathbb{C})}$---we show the weak-$\mathbb{A}$-faces and $(\{ 2 \}; \{ 1, 2 \})$-closed subsets agree, and equal Weyl group translates of the sets of weights in certain `standard faces' (which also holds for highest weight modules). This was proved by Chari and her coauthors for root systems in finite type, but is novel for other types.
	\end{abstract}
	\date{\today}
	\maketitle
	\settocdepth{section}
	\tableofcontents
	\section{Introduction}\allowdisplaybreaks
	Throughout this paper, $\mathfrak{g}$ stands for a Kac--Moody Lie algebra over $\mathbb{C}$ with a fixed Cartan subalgebra $\mathfrak{h}$, triangular decomposition $\mathfrak{n}^-\oplus\mathfrak{h}\oplus\mathfrak{n}^+$, root system $\Delta$, the set of simple roots and simple co-roots $\Pi$ and $\Pi^{\vee}$ respectively, the set of vertices/nodes in the Dynkin diagram $\mathcal{I}$, Weyl group $W$, the lattice of integral weights $P$, and the subset of dominant integral weights $P^+$. For a weight module $V$ over $\mathfrak{g}$, let $\wt V$ denote the set of weights of $V$. Let $I_V$ denote the integrability of $V$, i.e., the set of all nodes $i\in\mathcal{I}$ such that the negative simple root vectors $f_i$ act locally nilpotently on $V$. Fix $\lambda\in\mathfrak{h}^*$. We define $J_{\lambda}:=\{i\in\mathcal{I}$ $|$ $\langle\lambda,\alpha_i^{\vee}\rangle\in\mathbb{Z}\cap[0,\infty)\}$. Let $M(\lambda)$ and $L(\lambda)$ stand for the Verma module and the simple highest weight module over $\mathfrak{g}$ with highest weight $\lambda$, respectively. Let $M(\lambda,J)$ stand for the parabolic Verma module with highest weight $\lambda$ and with integrability $J\subseteq J_{\lambda}$. For $J\subseteq\mathcal{I}$, we define $\Pi_J:=\{\alpha_j\in \Pi$ $|$ $j\in J\}$, and denote the parabolic subgroup of $W$ corresponding to $J$ by $W_J$. For a subset $B$ of a real vector space, let $\conv_{\mathbb{R}}B$ denote the convex hull of $B$ over $\mathbb{R}$. In this section, to state various results from the literature related to the objects of our interest, unless otherwise mentioned we assume that $\mathfrak{g}$ is semisimple.
	
	In the past, \textit{root polytopes} ($\conv_{\mathbb{R}}\Delta$ for a finite root system $\Delta$) and \textit{Weyl polytopes} \big($\conv_{\mathbb{R}}(W\lambda)$ for $\lambda\in P^+$ and finite Weyl group $W$\big) were extensively studied in various contexts, and for many applications, by many mathematicians. To list but a few, in the context of compactifications of symmetric spaces through papers by Satake \cite[\S 2.3]{Satake} and Casselman \cite[\S 3]{Casselman}, in the study of algebraic groups by Borel--Tits \cite[\S 12]{Borel}, and also by Vinberg \cite[\S 3]{Vinberg} and Cellini--Marietti \cite{Cellini}. In these works, the faces of Weyl polytopes were classified as follows: 
	\begin{theorem}[Satake \cite{Satake}, Borel--Tits \cite{Borel}, Vinberg \cite{Vinberg}, Casselman \cite{Casselman}]\label{T1.1}
		Given $\lambda\in P^+$, the faces of $\conv_{\mathbb{R}}(W\lambda)$ are of the form $w\big(\conv_{\mathbb{R}}(W_J\lambda)\big)$ with $w\in W$ and $J\subseteq\mathcal{I}$.
	\end{theorem}
	In \cite{KhRi}, Khare and Ridenour studied the convex hulls of the weights of parabolic Verma modules modules over $\mathfrak{g}$. Generalizing the above results for root polytopes and Weyl polytopes, they classified all the faces of $\conv_{\mathbb{R}}(\wt M(\lambda,J))$ as follows:
	\begin{theorem}[Khare--Ridenour \cite{KhRi}]\label{T1.2}
		Let $\lambda\in\mathfrak{h}^*$, $J\subseteq J_{\lambda}$, and let $F\subseteq\conv_{\mathbb{R}}(\wt  M(\lambda,J))$. Then $F$ is a face of $\conv_{\mathbb{R}}(\wt M(\lambda,J))$ if and only if there exist a subset $I\subseteq\mathcal{I}$ and $w\in W_J$, such that $wF=\conv_{\mathbb{R}}\big((\lambda-\mathbb{Z}_{\geq 0}\Pi_I)\cap \wt M(\lambda,J)\big)$.
	\end{theorem}
	\big(Here, $\mathbb{Z}_{\geq 0}\Pi_I$ denotes the set of non-negative integer linear combinations of the simple roots from $\Pi_I$, and $\lambda-\mathbb{Z}_{\geq 0}\Pi_I$ is the Minkowski difference between the sets $\{\lambda\}$ and $\mathbb{Z}_{\geq 0}\Pi_I$.\big)
	\begin{definition}\label{D1.3}
		Let $V$ be a highest weight module over a Kac--Moody algebra $\mathfrak{g}$, with highest weight $\lambda\in\mathfrak{h}^*$. In the spirit of the above theorem, we will call the subsets of the form $(\lambda-\mathbb{Z}_{\geq 0}\Pi_I)\cap\wt V$, $\forall$ $I\subseteq\mathcal{I}$, as \textit{standard faces} of $\wt V$. 
	\end{definition}
	Khare et al. \cite{KhRi} looked more generally at the convex hulls of the weights of parabolic Verma modules uniformly over all subfields of $\mathbb{R}$. Alongside, they formally introduced \textit{weak}-$\mathbb{F}$-\textit{faces} (defined presently) and studied them for the sets of weights of parabolic Verma modules, over arbitrary subfields $\mathbb{F}$ of $\mathbb{R}$. In this, they were motivated to extend to the weights of highest weight modules  the following combinatorial condition/property, which was introduced and studied by Chari and Greenstein \cite{Chari_Ad, ChGr} for finite irreducible root systems over semisimple $\mathfrak{g}$.\\
	\textit{Fix }$\xi\in\mathfrak{h}^*\setminus\{0\}$,\textit{ and define }$\Psi(\xi):=\Big\{\alpha\in\Delta\text{ }\big|\text{ }\xi(\alpha)=\max\limits_{\gamma\in\Delta}\text{ }\xi(\gamma)\Big\}$.\textit{ Assume that }$\Psi(\xi)\subseteq\Delta^+$.
	\begin{align*}\tag{P1}\label{property1}
	&\textit{Suppose }\sum\limits_{\alpha\in\Delta}m_{\alpha}\alpha=\sum\limits_{\beta\in\Psi(\xi)}n_{\beta}\beta\text{ }\textit{ for some } m_{\alpha},n_{\beta}\in\mathbb{Z}_{\geq0}\text{ }\forall\text{ }\alpha\in\Delta,\textit{ }\beta\in\Psi(\xi).\\&\textit{Then }\sum\limits_{\beta\in\Psi(\xi)}n_{\beta}\leq \sum\limits_{\alpha\in\Delta}m_{\alpha},\textit{ with equality if and only if } m_{\alpha}=0\text{ }\forall\text{ }\alpha\in\Delta\setminus\Psi(\xi).
	\end{align*}
	
	In \cite{Chari_Ad, ChGr}, Chari and Greenstein used the above property \eqref{property1} of $\Psi(\xi)$ to construct Koszul algebras of any finite global dimension and also obtain a (graded) character formula for the specialization at $q=1$ of a Kirillov--Reshetikhin module over an untwisted quantum affine algebra. Briefly, what Chari et al. do and achieve is as follows. (For the notation and account in this paragraph, see \cite{Chari_Ad, ChGr}.) In \cite{Chari_Ad}, using $\Psi(\xi)$ they first define a partial order $\leq_{\Psi(\xi)}$ on $P$, and then construct a category $\mathcal{G}_2[\leq_{\Psi(\xi)}\lambda]$, for $\lambda\in P^+$, of finite-dimensional $\mathbb{Z}_{+}$-graded $\mathfrak{g}\ltimes \mathfrak{g}_{\ad}$-modules whose irreducible constituents have their highest weights below $\lambda$ with respect to $\leq_{\Psi(\xi)}$. Then they construct a projective generator in this category and show that its endomorphism algebra is Koszul, with global dimension $\leq |\Psi(\xi)|$ (respectively, $=|\Psi(\xi)|$ for suitable dominant $\lambda$). Property \eqref{property1} of $\Psi(\xi)$ is heavily used in showing each of these results. Next in \cite{ChGr}, using $\Psi(\xi)$ for $\xi\in P^+$, the authors construct a category of $\mathbb{Z}_+$-graded $\mathfrak{g}\ltimes\mathfrak{g}_{\ad}$-modules, in which the irreducible objects have their highest weights coming from a ``ray/ half-closed interval'' corresponding to some element in $P^+\times\mathbb{Z}_+$. These rays/intervals are defined with respect to $\leq_{\Psi(\xi)}$. Making use of property (P1), they show the existence of projective covers for all simple objects in the category, and then find character formulas for these projective covers. Then they show that a Kirillov--Reshetikhin module, upon specializing at $q=1$, becomes isomorphic to the projective cover of some simple object whose highest weight is a positive integer multiple of a fundamental dominant weight, and thereby obtain a character formula for Kirillov--Reshetikhin modules. Later, Chari et al. \cite{ChKh} extended the above results of \cite{Chari_Ad} using the sets of weights falling on the faces of Weyl polytopes.
	
	In \cite{ChDo}, Chari and her co-authors introduced another property of these maximizer subsets $\Psi(\xi)$.
	\begin{align*}
	\tag{P2}\label{property2}
	&\textit{Suppose there are roots }\alpha_1,\alpha_2\in \Psi(\xi)\textit{ and }\beta_1,\beta_2\in\Delta\textit{ such that }\alpha_1+\alpha_2=\beta_1+\beta_2.\\
	&\textit{Then }\beta_1,\beta_2\in \Psi(\xi). \textit{ Furthermore, } \gamma_1+\gamma_2\notin \Delta\text{ }\forall \text{ }\gamma_1,\gamma_2\in \Psi(\xi).\\&\textit{Equivalently,}\quad
	\big[\Psi(\xi)+\Psi(\xi)\big]\cap\big[\Delta+\big(\Delta\setminus\Psi(\xi)\big)\big]=\emptyset\quad\textit{and}\quad\big[\Psi(\xi)+\Psi(\xi)\big]\cap\Delta=\emptyset.
	\end{align*}
	They also showed that if a subset $\Psi\subseteq\Delta$ satisfies either \eqref{property1} or \eqref{property2}, then $\Psi$ is the maximizer subset in $\Delta$ for the linear functional $\left(\sum_{\alpha\in\Psi}\alpha,-\right)$, where $(.,.)$ is the Killing form on $\mathfrak{h}^*$. Using this result, they proved that for simple $\mathfrak{g}$, every parabolic subalgebra contains a unique irreducible \textit{ad-nilpotent} ideal determined by some subset of $\Delta$ having property \eqref{property2}.    
	\begin{rem}
		Property \eqref{property2}, which we presently formalise as ``212-closed subsets $Y \subseteq X$'' of a real vector space, has a natural combinatorial interpretation when $X$ is the set of lattice points in a lattice polytope---i.e., $X$ is the intersection of a lattice in the vector space, with $\conv_{\mathbb{R}} X$, such that the lattice contains the vertices of $\conv_{\mathbb{R}}X$. Suppose $Y$ denotes a subset of `colored' (or in a more contemporary spirit, `infected') lattice points in $X$, with the property that if $y \in Y$ is the average of two other points in $X$, then the `color' or `infection' spreads to both of those points from~$y$. (More precisely, the property is that if two pairs of---not necessarily distinct---points have~the same average, and one pair is colored, then the color spreads from that pair to the other pair.) We would like to understand the extent to which the spread happens. A `continuous' variant of the problem involves working with the entire convex hull itself, rather than the lattice points in $X$.
		
		Two of our main results, Theorems \ref{thmA} and \ref{thmB} (see also Theorem \ref{T2.3}), show that the color/infection spreads precisely to the smallest face of $\conv_{\mathbb{R}} X$ that contains $Y$. In fact, we address not just lattice polytopes, which arise from integrable representations of semisimple Lie algebras, but $\wt V$ and $\conv_{\mathbb{R}} \wt V$ for \textit{all} highest weight modules $V$, and over arbitrary Kac--Moody Lie algebras $\mathfrak{g}$.
	\end{rem}
	Next, we present some recent studies along these directions that further motivated this paper. For the rest of this section, we fix an arbitrary non-trivial additive subgroup $\mathbb{A}$ of $(\mathbb{R},+)$. Recall, $\lambda$ is said to be \textit{simply regular} if $\langle\lambda,\alpha_i^{\vee}\rangle\neq 0$ $\forall$ $i\in\mathcal{I}$.
	
	Recently, Khare \cite{Khare_JA} extensively studied the convex hulls of the weights of highest weight modules over semisimple $\mathfrak{g}$, and solved several problems of a classical nature and achieved the following: i)~$\conv_{\mathbb{R}}\wt V$ is the same for all highest weight modules $V$ having the same integrability; ii)~thereby, the classification result above on faces of $\conv_{\mathbb{R}}\wt V$ extends more generally to all $V$; iii)~for any $\lambda\in\mathfrak{h}^*$, $\wt L(\lambda)$ can be retrieved from $\conv_{\mathbb{R}}\wt L(\lambda)$ by intersecting it with $\lambda-\mathbb{Z}_{\geq 0}\Pi$; iv)~thereby, various formulas were obtained for $\wt L(\lambda)$ $\forall$ $\lambda\in\mathfrak{h}^*$. Alongside, Khare studied using combinatorial approaches, in full generality, the weak-$\mathbb{A}$-faces of $\wt V$; see Definition \ref{D1.4}. He also extended property \eqref{property2} to subsets of $\wt V$, and called such subsets as $(\{2\};\{1,2\})$-closed subsets. 
	\begin{definition}\label{D1.4}
		Let $\mathbb{A}$ be a fixed non-trivial additive subgroup of $(\mathbb{R},+)$, and $Y\subseteq X$ be two subsets of a real vector space. Let $\mathbb{A}_{\geq 0}:=\mathbb{A}\cap [0,\infty)$.\\
		(1) $Y$ is said to be a \textit{weak}-$\mathbb{A}$-\textit{face of} $X$ if and only if 
		\[
		\begin{rcases*}
		\sum\limits_{i=1}^{n}r_iy_i=\sum\limits_{j=1}^{m}t_jx_j\text{ and }\sum\limits_{i=1}^{n}r_i=\sum\limits_{j=1}^{m}t_j >0\qquad
		\text{for }m,n\in\mathbb{N},\\
		y_i\in Y,\text{ }
		x_j\in X,\text{ }r_i,t_j\in \mathbb{A}_{\geq 0}\text{ }\forall\text{ }1\leq i\leq n,1\leq j\leq m
		\end{rcases*}
		\implies\text{ }x_j\in Y\text{ }\forall\text{ }j\text{ such that }t_j\neq 0.
		\]
		By \textit{weak faces} of $X$---as mentioned in the title of the paper for instance---we mean the collection of all weak-$\mathbb{A}$-faces of $X$ for all additive subgroups $\mathbb{A}\subseteq (\mathbb{R},+)$.\smallskip\\ 
		(2) $Y$ is said to be a $(\{2\};\{1,2\})$-\textit{closed subset of} $X$ if and only if
		\[
		(y_1)+(y_2)=(x_1)+(x_2)\qquad \text{ for some }y_1,y_2\in Y\text{ and }x_1,x_2\in X\quad\implies\quad x_1,x_2\in Y.
		\]
		Note that in the above equation $y_1,y_2$ (similarly, $x_1,x_2$) need not be distinct. Also, 1)~sometimes we will simply say that $Y$ is $(\{2\};\{1,2\})$-closed in $X$ to denote ``$Y$ is a $(\{2\};\{1,2\})$-closed subset of $X$'', and 2)~we will refer to the property of $Y$ of being $(\{2\};\{1,2\})$-closed (in $X$) as the $(\{2\};\{1,2\})$-closedness of $Y$ (in $X$). We call $Y$ a proper weak-$\mathbb{A}$-face (similarly, proper $(\{2\};\{1,2\})$-closed subset) of $X$ if $Y\subsetneqq X$ and $Y$ is a weak-$\mathbb{A}$-face (respectively, $(\{2\};\{1,2\})$-closed subset) of $X$. 
	\end{definition}
	\begin{abbr} In view of the extensive usage, for simplicity, from here on we abbreviate the term ``$(\{2\};\{1,2\})$-closed'' as 212-closed.
	\end{abbr}
	More generally, Khare \cite[Theorem C]{Khare_JA} showed for $V$ belonging to a large class of highest weight modules in finite type---which contains $L(\lambda)$ $\forall$ $\lambda\in\mathfrak{h}^*$, as well as all $V$ with simply regular highest weights among others---that the weak-$\mathbb{A}$-faces, 212-closed subsets, and the standard faces of $\wt V$ are all the same up to conjugation by an element of the \textit{integrable Weyl group} $W_{I_V}$ of $V$ (defined in \eqref{defn IV-JY}). This extends the aforementioned result of Chari et al. for root polytopes and Weyl polytopes, to all highest weight modules over semisimple $\mathfrak{g}$ with ``generic'' highest weights. Furthermore, it unifies all of the above notions: weak-$\mathbb{A}$-faces, 212-closed subsets, and standard faces of $\wt V$ (therefore, also classical faces of $\conv_{\mathbb{R}}\wt V$), for all highest weight modules $V$ with ``generic'' highest weights. This leaves open the question of classifying the weak-$\mathbb{A}$-faces and 212-closed subsets of $\wt V$ and $\conv_{\mathbb{R}}\wt V$, for the `more interesting' highest weight modules with highest weights lying on one or more simple root hyperplanes. Our results in this paper cover this case as well (see the next section), and are proved in a uniform, type-free manner. 
	
	Dhillon and Khare \cite{Khare_Ad, Dhillon_arXiv} extended all the results of \cite{Khare_JA} mentioned in points i)--iv) prior to Definition \ref{D1.4}, to all highest weight modules over general Kac--Moody algebras. In these papers, for any highest weight module $V$, they once again classified the faces of $\conv_{\mathbb{R}}\wt V$ to be all the $W_{I_V}$-conjugates of standard faces of $\wt V$. Furthermore, they also studied the dimensions of the faces. Namely, they constructed the \textit{f-polynomial} for $\conv_{\mathbb{R}}\wt V$ whose coefficients give the number of faces of each dimension.        
	
	However, to our knowledge, these combinatorial objects---weak-$\mathbb{A}$-faces and 212-closed subsets of $\wt V$---have not been investigated either beyond the works of Khare \cite{Khare_JA} in finite type, nor even for the adjoint and simple integrable highest weight representations in the infinite setting. This is where our present paper begins and takes over, attracted by the following intriguing problems:
	\begin{itemize}
		\item[a)] Can one extend the results classifying the weak-$\mathbb{A}$-faces of $\wt V$ in \cite{Khare_JA} to infinite type, and to all $\lambda$ in finite type?
		\item[b)] In the semisimple case, \cite{Khare_JA} shows that for $V$ a parabolic Verma module, the weak-$\mathbb{A}$-faces, 212-closed subsets, and standard faces of $\wt V$ all agree modulo $W_{I_V}$-conjugation. Can this result be shown for all highest weight modules $V$? How about for $\mathfrak{g}$ of infinite type? 
	\end{itemize}
	
	In the present paper, we completely solve the above problems for all highest weight modules (and more) over general Kac--Moody algebras. In this, we are also motivated by the various connections and applications these weak-$\mathbb{A}$-faces and 212-closed subsets (of $\wt V$) have to some of the important sets studied in representation theory in the literature; see the survey part in \cite[\S 2]{Khare_JA}. Let $\mathfrak{g}$ be a general Kac--Moody algebra with root system $\Delta$. Briefly, what we achieve in this paper is as follows:
	\begin{itemize}
		\item[(1)] We classify all the 212-closed subsets of: i) $\wt V$ and $\conv_{\mathbb{R}}\wt V$ for every highest weight $\mathfrak{g}$-module $V$ (see Theorems \ref{thmA} and \ref{thmB}); ii) $\wt \mathfrak{g}=\Delta\sqcup\{0\}$; iii) the root system $\Delta$ \big(both ii) and iii) are in Theorem \ref{thmC}\big).
		\item[(2)] We show that the 212-closed subsets of $\wt V$ are the same as the $W_{I_V}$-conjugates of standard faces of $\wt V$---so when $\mathfrak{g}$ is of finite type, the 212-closed subsets of $\Delta\sqcup\{0\}$ are the same~as the $W$-conjugates of standard faces of $\Delta\sqcup\{0\}$. Our proofs for these are very much direct, uniform over all types of $\mathfrak{g}$, and case free. In general, 212-closed subsets of $\Delta$ have nice descriptions (see Theorem \ref{thmC} and Proposition \ref{P2.7}). When $\Delta$ is indecomposable, the 212-closed subsets of $\Delta$ are the same as those of $\Delta\sqcup\{0\}$, except just in two cases---where $\mathfrak{g}=\mathfrak{sl}_3(\mathbb{C})$ or $\widehat{\mathfrak{sl}_3(\mathbb{C})}$.
		\item[(3)] The weak-$\mathbb{A}$-faces and 212-closed subsets for each of $\wt V$, $\conv_{\mathbb{R}}\wt V$ and $\Delta\sqcup\{0\}$ are the same. The analogous result holds for $\Delta$ when none of its indecomposable components equal the root systems of $\mathfrak{sl}_3(\mathbb{C})$ and $\widehat{\mathfrak{sl}_3(\mathbb{C})}$. For $\mathfrak{g}=\mathfrak{sl}_3(\mathbb{C})$ or $\widehat{\mathfrak{sl}_3(\mathbb{C})}$, the weak-$\mathbb{A}$-faces of $\Delta$, 212-closed subsets of $\Delta\sqcup\{0\}$ and weak-$\mathbb{A}$-faces of $\Delta\sqcup\{0\}$ all agree. By point (1) therefore, we classify the weak-$\mathbb{A}$-faces of all the sets in point (1), which is the main goal of this paper.  
		\item[(4)] For $\mathfrak{g}=\mathfrak{sl}_3 (\mathbb{C})$, we will also give the list of all 212-closed subsets for $\Delta$ which do not equal any $W$-conjugate of the standard faces of $\Delta\sqcup\{0\}$.
		\item[(5)] Additionally, we also describe the weak-$\mathbb{A}$-faces and 212-closed subsets of every hull $\conv_{\mathbb{R}}S$ in terms of the analogous notions for $S$ (see Theorem \ref{T2.3}). These results in particular apply to $\conv_{\mathbb{R}}\wt V$ over Kac--Moody algebras, extending the similar results on weak-$\mathbb{R}$-faces in \cite{KhRi} for the convex hulls of the weights of parabolic Verma modules in finite type. 
	\end{itemize}
	\subsection{Organization of the paper} The paper is organized as follows. In Section 2, we state the three main results of this paper: 1)~Theorem \ref{thmA}, classifying, and also showing the equivalence of the notions, 212-closed subsets and weak-$\mathbb{A}$-faces of $\wt V$ for arbitrary highest weight modules $V$ over Kac--Moody $\mathfrak{g}$ and for arbitrary non-trivial $\mathbb{A}\subseteq(\mathbb{R},+)$; 2)~Theorem \ref{thmB}, extending the results in point 1) to convex hulls of $\wt V$; 3)~Theorem \ref{thmC}, extending the results in point 1) to $\Delta$ and also $\Delta\sqcup\{0\}$. All the necessary notations for the paper are given in Subsections 2.1 and 3.1. In Section~3, we quote some definitions and preliminary results. In Section 4, we build necessary tools for the paper: 1)~starting from stating useful remarks on 212-closed subsets of $\wt V$; 2)~finding roots which upon subtracting from the highest weight of $V$ give rise to weights in $\wt V$; 3)~introducing and motivating an important object used in the proof of Theorem \ref{thmA}, namely $P(Y)$ for a 212-closed subset $Y$ of $\wt V$; 4)~finding necessary and sufficient conditions for a 212-closed subset of $\wt V$ to equal a standard face of $\wt V$; 5)~studying the maximal elements of 212-closed subsets of $\wt V$. The entirety of Section 5 is devoted to prove Theorem \ref{thmA}. (This proof is very long and runs over several pages of involved arguments; thus, it is divided into eleven steps, for ease of exposition.) Section 6 contains the proof of Theorem \ref{thmB}. Section 7 is occupied by the (separate) proofs of the four parts of Theorem \ref{thmC}, and it concludes with an answer to the problem, mentioned at the end of Section 2, of finding 212-closed subsets of root systems consisting only of positive roots.        
	\section{Main results}
	In this section, we state the main results of the paper. We begin by developing notation.
	\subsection{Root and weight notation} Let $\mathbb{N}$, $\mathbb{Z}$,  $\mathbb{R}$ and $\mathbb{C}$ stand for the set of natural numbers, integers, real numbers and complex numbers, respectively. We denote the set of non-negative, positive, non-positive and negative real numbers by $\mathbb{R}_{\geq0}, \mathbb{R}_{>0}$, $\mathbb{R}_{\leq 0}$ and $\mathbb{R}_{<0}$, respectively. Similarly, for any $S \subseteq \mathbb{R}$ we define $S_*:=S\cap \mathbb{R}_*$ for $*$ any of $\geq 0,>0,\leq0$ and $<0$. For $n\in\mathbb{N}$, we denote the set $\{1,\ldots,n\}$ by $[n]$. For $\emptyset\neq S\subseteq\mathbb{R}$ and a subset $B$ of an $\mathbb{R}$-vector space, $SB$ stands for $\big\{\sum_{j=1}^{n}r_jb_j\text{ }\big|\text{ }n\in \mathbb{N},\text{ } b_j\in B, r_j\in S\text{ }\forall  \text{ }j\in [n]\big\}$, the set of $S$-linear combinations of elements~from~$B$. We define the $\mathbb{R}$-convex hull of $B$ to be	\[
	\conv_{\mathbb{\mathbb{R}}}B:=\left\{\sum\limits_{j=1}^{n}c_j b_j \text{ }\Bigg|\text{ } n\in\mathbb{N},\text{ }b_j\in B,\text{ } c_j \in \mathbb{R}_{\geq 0}\text{ } \forall\text{ } 1\leq j\leq n \text{ and } \sum\limits_{j=1}^{n}c_j=1\right\}
	\]
	For any two subsets $C$ and $D$ of an abelian group, $C\pm D:=\{c\pm d$ $|$ $c \in C,\text{ }d\in D\}$ the Minkowski sum of $C$ and $\pm D$, respectively. When $C=\{x\}$ is singleton, we denote $C\pm D$ by $x\pm D$~for~simplicity.
	
	Let $\mathfrak{g}=\mathfrak{g}(A)$ denote the Kac--Moody algebra over $\mathbb{C}$ corresponding to a generalized Cartan matrix $A$ with the realisation $(\mathfrak{h},\Pi ,\Pi^{\vee})$, triangular decomposition $\mathfrak{n}^-\oplus\mathfrak{h}\oplus\mathfrak{n}^+$, universal enveloping algebra $U(\mathfrak{g})$, and root system $\Delta$. Whenever we make additional assumptions on $\mathfrak{g}$, such as $\mathfrak{g}$ is indecomposable or semisimple or of finite/affine/indefinite type, we will clearly mention them. Let $\Pi =\{\alpha_i $ $|$ $i\in \mathcal{I} \}$ be the set of simple roots, and $\Pi^{\vee}=\{\alpha_i^{\vee}$ $|$ $i\in\mathcal{I}\}$ be the set of simple co-roots, where $\mathcal{I}$ is a  fixed indexing set for the simple roots for the entire paper. $\mathcal{I}$ also stands for the set of vertices/nodes in the Dynkin diagram for $A$ or $\mathfrak{g}$. Throughout the paper we assume that $\mathcal{I}$ is finite. Let $\Delta^+$ and $\Delta^-$ denote the set of positive and negative roots of $\Delta$, respectively. Let $e_i ,f_i,\alpha_i^{\vee}$, $\forall$ $ i\in\mathcal{I}$ be the Chevalley generators for $\mathfrak{g}$, and $W$ denote the Weyl group of $\mathfrak{g}$ generated by the simple reflections $\{s_i $ $|$ $i\in \mathcal{I}\}$. Let $\mathfrak{g}':=[\mathfrak{g},\mathfrak{g}]$ be the derived subalgebra of $\mathfrak{g}$, which is generated by $e_i,f_i,\alpha_i^{\vee}$, $\forall$ $i\in\mathcal{I}$. Let $\Delta^{re}$ denote the set of real roots of $\Delta$. When $\mathfrak{g}$ is symmetrizable, we fix a standard non-degenerate symmetric invariant bilinear form on $\mathfrak{h}^*$ and denote it by (.,.).
	
	For $\emptyset \neq I\subseteq \mathcal{I}$, we define $\Pi_{I}:=\{\alpha_i$ $|$ $i\in I\}$ and $\Pi_I^{\vee}:=\{\alpha_i^{\vee}$ $|$ $i\in I\}$. We define $\mathfrak{g}_I:=\mathfrak{g}(A_{I\times I})$ to be the Kac--Moody algebra corresponding to the submatrix $A_{I\times I}$ of $A$ with realisation $(\mathfrak{h}_I,\Pi_I,\Pi_I^{\vee})$, where $\mathfrak{h}_I \subseteq \mathfrak{h}$, and the Chevalley generators $e_i ,f_i ,\alpha^{\vee}_i$, $\forall$ $i\in I$. By \cite[Exercise 1.2]{Kac}, $\mathfrak{g}_I$ can be thought of as a subalgebra of $\mathfrak{g}$, and the subroot system $\Delta_{I}:=\Delta\cap\mathbb{Z}\Pi_I$ of $\Delta$ coincides with the root system of $\mathfrak{g}_I$. Let $W_I$ denote the parabolic subgroup of $W$ generated by the simple reflections $\{s_i $ $|$ $i \in I\}$. When $I=\emptyset$, for completeness we define (i)~$\Pi_I,\Pi_I^{\vee}$ and $\Delta_{I}$ to be $\emptyset$; (ii)~$\mathfrak{g}_I$ and $\mathfrak{h}_I$ to be $\{0\}$; (iii)~$W_I$ to be the trivial subgroup $\{1\}$ of $W$. 
	
	Given an $\mathfrak{h}$-module $M$ and
	$\mu\in\mathfrak{h}^*$, we denote the $\mu$-weight space and the set of
	weights of $M$ by
	\[ 
	M_{\mu}=\{m\in M\text{ }|\text{ }h\cdot m= \mu(h)m\text{ }\forall\text{ }h\in \mathfrak{h}\}\quad\text{and}\quad \wt M=\big\{\mu\in\mathfrak{h}^*\text{ }|\text{ }M_{\mu}\neq\{0\}\big\}.
	\]
	We say that $M$ is a weight module over $\mathfrak{h}$ if $M=\bigoplus\limits_{\mu\in\wt M}M_{\mu}$. When each weight space of a weight module $M$ is finite-dimensional, we define $\text{char}M:=\sum_{\mu\in\wt M}\dim(M_{\mu})e^{\mu}$ to be the formal character of $M$. A $\mathfrak{g}$-module $M$ is called a weight module if it is a weight module over $\mathfrak{h}$.
	
	For $\lambda \in \mathfrak{h}^*$, let $M(\lambda)$ and $L(\lambda)$ denote the Verma module over $\mathfrak{g}$ with highest weight $\lambda$ and its unique simple quotient, respectively. By $M(\lambda)\twoheadrightarrow V$ \big($M(\lambda)$ surjecting onto $V$\big) we denote a non-trivial highest weight $\mathfrak{g}$-module $V$ with highest weight $\lambda$.
	
	For $\lambda\in\mathfrak{h}^*$, $M(\lambda)\twoheadrightarrow V$, and $I\subseteq\mathcal{I}$, we define
	\begin{equation}\label{E1.2}
	J_{\lambda}:=\left\{i \in \mathcal{I} \text{ }|\text{ } \langle \lambda ,\alpha^{\vee}_i\rangle \in \mathbb{Z}_{\geq0}\right\},\qquad\qquad
	\wt_I V:=\wt V\cap (\lambda-\mathbb{Z}_{\geq0}\Pi_I).
	\end{equation}
	
	For a vector $x=\sum_{i\in\mathcal{I}}c_i\alpha_i\in\mathbb{C}\Pi$, we define $\supp(x):=\{i\in\mathcal{I}$ $|$ $c_i\neq 0\}$; $\supp(0):=\emptyset$. The following notations are important to note. For $M(\lambda)\twoheadrightarrow{ }V$ and $\emptyset\neq Y\subseteq \wt V$, we define
	\begin{align}\label{defn IV-JY}
	\begin{aligned}
	&I_V:=\left\{i\in \mathcal{I}\text{ }\Big|\text{ }\langle\lambda,\alpha_i^{\vee}\rangle\in\mathbb{Z}_{\geq0}\text{ and }f_i^{\langle\lambda,\alpha_i^{\vee}\rangle+1}V_{\lambda}=\{0\}\right\},\\
	&\mathcal{I}_V:=\{i\in\mathcal{I}\text{ }|\text{ } \lambda-\alpha_i\in\wt V\},\\
	&\mathcal{I}_Y:=\bigcup_{\mu\in Y}\supp(\lambda-\mu),\\
	\end{aligned}\quad\text{ }
	\begin{aligned}
	&W_{I_V}:=\big\langle s_i\text{ }|\text{ }i\in I_V\big\rangle,\\
	&\mathcal{J}_V:=\mathcal{I}\setminus\mathcal{I}_V,\\
	&\mathcal{J}_Y:=\{j\in \mathcal{I}_Y\text{ }|\text{ } \lambda-\alpha_j \notin \wt V\}.
	\end{aligned}
	\end{align}

	Throughout the paper, we use the two symbols $\mathcal{I}_Y$ and $\mathcal{J}_Y$, with the module $V$ always being clear from context. Note that $Y\subseteq \wt_{\mathcal{I}_Y}V$ and $\mathcal{J}_Y\subseteq \mathcal{J}_V$. We call $I_V$ the \textit{integrability} of $V$, or the set of integrable directions in $V$ (or $\wt V$). By the definition of $I_V$, notice that $V$ is $\mathfrak{g}_{I_V}$-integrable.
	
	\subsection{Weak faces of weights, and their convex hulls, of highest weight modules}
	In the remainder of this section, we describe our main results. We begin with the condition \eqref{property1} used by Chari and Greenstein \cite{Chari_Ad, ChGr} (and later by Chari--Khare--Ridenour \cite{ChKh}) in constructing Koszul algebras from truncated current algebras. This condition \eqref{property1} is a ``discrete'' version of the notion of the face of a convex polyhedron. Namely, it is a folklore fact (see e.g. \cite[Theorem 4.3]{KhRi}) that if one replaces $\mathbb{Z}_{\geq 0}$ by $\mathbb{R}_{\geq 0}$ in \eqref{property1}, applied to the subsets $Y=\Psi(\xi) \subseteq X$ where $X$ is a convex polyhedron on which $\xi$ has a positive maximum value, then $Y$ is necessarily a face of $X$. With this in mind, Khare--Ridenour (and Khare) \cite{KhRi, Khare_JA} defined various notions of ``discrete faces''---weak-$\mathbb{A}$-faces \big(for arbitrary non-trivial additive subgroups $\mathbb{A}$ of $(\mathbb{R},+)$\big) and 212-closed subsets. See Definition \ref{D1.4}. At a basic level, these notions are related as follows:
	\begin{observation}\label{O2.3}
		It is clear from the definitions that for any pair of subsets $Y \subseteq X$ of a real vector space, and all non-trivial $\mathbb{A}
		\subseteq (\mathbb{R},+)$, each of the following statements implies the next:
		\begin{enumerate}
			\item There exists a linear functional $\psi\in X^*$ such that $\psi(x)\leq \psi(y)$ for all $x\in X$ and $y\in Y$, i.e., $Y$ maximizes $\psi$ in $X$.
			\item $Y$ is a weak-$\mathbb{R}$-face of $X$.
			\item $Y$ is a weak-$\mathbb{A}$-face of $X$.
			\item $Y$ is a 212-closed subset of $X$.
		\end{enumerate}
	\end{observation}
	In this paper, we focus only on certain distinguished sets $X$ in Lie theory---namely, the set $X = \wt V$ (and $X=\conv_{\mathbb{R}}\wt V$) of weights of (respectively, the convex hull of weights of) an arbitrary highest weight module $V$ over a Kac--Moody Lie algebra $\mathfrak{g}$; as well as for $X$ the root system $\Delta$ of $\mathfrak{g}$ \big(and $X = \Delta \sqcup \{ 0 \}$, the weights of the adjoint representation\big). The main results of the paper classify the above notions for these subsets $X$---in a sense, the ``discrete faces'' of $X$ (in parallel to the classification of faces of $\conv_{\mathbb{R}} \wt V$ by Satake, Borel--Tits, and others listed above). In particular, we will show that the above notions all agree for these distinguished sets $X$. 
	
	The first main result of this paper classifies all of the above notions for all highest weight $\mathfrak{g}$-modules. To state it, we use the notation $J_{\lambda}$, $\wt_I V$, $I_V$, $\mathcal{I}_V,\text{ }\mathcal{J}_V,\text{ }\mathcal{I}_Y\text{ and }\mathcal{J}_Y$ from \eqref{E1.2} and \eqref{defn IV-JY}.\allowdisplaybreaks 
	\begin{thmx}\label{thmA}
		Let $\mathfrak{g}$ be a general Kac--Moody algebra, $\lambda\in\mathfrak{h}^*$, $M(\lambda)\twoheadrightarrow V$, and $Y$ be a non-empty 212-closed subset of $\wt V$. Then there exists $\omega\in W_{I_V}$ such that $\omega Y=\wt_{\mathcal{I}_{\omega Y}}V$ (the standard face of $\wt V$ corresponding to $\mathcal{I}_{\omega Y}\subseteq\mathcal{I}$); in particular, when $\lambda\in Y$ such an element $\omega$ exists in $W_{\mathcal{J}_Y}\subseteq W_{\mathcal{J}_V}$. Therefore, it follows that all of the following notions are equivalent:
		\begin{itemize}
			\item $W_{I_V}$-conjugates of the standard faces of $\wt V$.
			\item Weak $\mathbb{A}$-faces of $\wt V$ over an arbitrary non-trivial additive subgroup $\mathbb{A}\subseteq (\mathbb{R},+)$.
			\item 212-closed subsets of $\wt V$.
		\end{itemize}
	\end{thmx}
	\begin{observation}\label{note1}
		If $Y$ is 212-closed in $\wt V$, then $wY$ is 212-closed in $\wt V$ for any $w\in W_{I_V}$. This can be easily seen as follows. Suppose $(z_1)+(z_2)=(x_1)+(x_2)$ for some $z_1,z_2\in wY$ and $x_1,x_2\in \wt V$. Apply $w^{-1}$ on both sides of the equation to get $(w^{-1}z_1)+(w^{-1}z_2)=(w^{-1}x_1)+(w^{-1}x_2)$. Notice, $w^{-1}z_1,w^{-1}z_2\in Y$ and $w^{-1}x_1,w^{-1}x_2\in \wt V$ by the $W_{I_V}$-invariance of $\wt V$. Now, the 212-closedness of $Y$ implies that $w^{-1}x_1,w^{-1}x_2\in Y$, and therefore $x_1,x_2\in wY$ as required. A similar argument shows that if $Y$ is a weak-$\mathbb{A}$-face of $\wt V$, then so is every $wY$ for $w\in W_{I_V}$.     
	\end{observation}
	Recall, the equivalence of the three notions in Theorem \ref{thmA} was shown when $\mathfrak{g}$ is semisimple and $\lambda-\Pi\subset\wt V$ by Khare; see \cite[Theorem 3.4]{Khare_JA}. This particular assumption that $\lambda-\Pi\subset\wt V$ paved the way for our understanding of the 212-closed subsets of $\wt V$, and also motivated our results in Section \ref{S4} on 212-closed subsets of the weights of parabolic Verma and simple highest weight modules over Kac--Moody $\mathfrak{g}$.
	
	Our second main result, Theorem \ref{thmB}, is the analogue of Theorem \ref{thmA} (which is for $X=\wt V$) in the ``continuous setting'', i.e., for $X=\conv_{\mathbb{R}}\wt V$. Theorem \ref{thmB}, similar to Theorem \ref{thmA}, shows that the above notions all agree, and recovers precisely the (exposed) faces of the convex hull: 
	\begin{thmx} \label{thmB}
		In the notation of Theorem \ref{thmA}, if $X = \conv_{\mathbb{R}} \wt V$, then the following notions are equivalent:
		\begin{itemize}
			\item[(1)] exposed faces of $X$, i.e., maximizer subsets of $X$ with respect to a linear functional.
			\item[(2)] weak-$\mathbb{R}$-faces of $X$.
			\item[(3)] 212-closed subsets of $X$.
			\item[(4)] subsets of the form $\conv_{\mathbb{R}}\big(w[(\lambda-\mathbb{Z}_{\geq 0}\Pi_J)\cap\wt V]\big)$ for $w \in W_{I_V}$ and $J\subseteq \mathcal{I}$. 
		\end{itemize}
	\end{thmx}
	Thus, Theorem \ref{thmB} also classifies the weak-$\mathbb{A}$-faces of $\conv_{\mathbb{R}}\wt V$ for all $\{0\}\subsetneqq \mathbb{A}\subseteq (\mathbb{R},+)$, by Observation \ref{O2.3}. This result was previously known in finite type for parabolic Verma modules, see \cite[Theorem 4.3]{KhRi}, and hence for certain classes of highest weight modules by \cite{Khare_Ad}, \cite{Dhillon_arXiv}~and~\cite[Theorem~C]{Khare_JA}. Theorem \ref{thmB} completes this classification in finite type, and extends it to arbitrary Kac--Moody $\mathfrak{g}$. The key result used in the proof of Theorem \ref{thmB} is Theorem \ref{T2.3} (below), which studies 212-closed subsets and weak faces of arbitrary convex subsets of vector spaces over any subfield of $\mathbb{R}$.
	
	We now state Theorem \ref{T2.3} in its most general form. Let $\mathbb{F}$ be a subfield of $\mathbb{R}$ and $S$ be an arbitrary subset of an $\mathbb{F}$-vector space. We assert that the 212-closed subsets, and therefore the weak faces, of the $\mathbb{F}$-\textit{convex hull} of $S$ \big($\conv_{\mathbb{F}}S$\big) are given by those of $S$. We call a subset $X$ of an $\mathbb{F}$-vector space to be $\mathbb{F}$-\textit{convex} (or simply \textit{convex} for $\mathbb{F}=\mathbb{R}$) if $rx+(1-r)y\in X$ $\forall$ $x,y\in X$, $r\in [0,1]\cap\mathbb{F}$.
	\begin{theorem}\label{T2.3}
		Let $\mathbb{F}$ be an arbitrary subfield of $\mathbb{R}$, and $X$ an $\mathbb{F}$-convex subset of an $\mathbb{F}$-vector space. Fix two non-empty subsets $S$ and $D$ of the vector space. Let $Y$ be a non-empty 212-closed subset \big(for instance, a weak-$\mathbb{A}$-face over a non-trivial subgroup $\mathbb{A}\subseteq(\mathbb{F},+)$\big) of $X$. \begin{itemize}
			\item[(a)] Suppose $X=\conv_{\mathbb{F}}S$. Then $Y=\conv_{\mathbb{F}}(Y\cap S)$. Note, $Y\cap S$ is 212-closed in $S$.
			\item[(b)] More generally, suppose $X=\conv_{\mathbb{F}}S+\mathbb{F}_{\geq 0}D$, the Minkowski sum of the $\mathbb{F}$-convex hull of the points in $S$ and the cone generated by the directions given by $D$ over $\mathbb{F}_{\geq 0}$. Then $Y$ has the following Minkowski sum decomposition: \[Y=\conv_{\mathbb{F}}(Y\cap S)+\mathbb{F}_{\geq 0}\big\{d\in D\text{ }|\text{ }(d+Y)\cap Y\neq\emptyset\big\}.\] 
		\end{itemize}
	\end{theorem}
	Theorem \ref{T2.3} might be interesting in its own right; moreover, it can be easily seen that the implications (2) $\implies$ (4) and (3) $\implies$ (4) of Theorem \ref{thmB} immediately follow from Theorem \ref{T2.3} in view of Theorem \ref{thmA}.
	\subsection{Weak faces of root systems}
	Our third main result involves classifying the weak-$\mathbb{A}$-faces and 212-closed subsets of the root system $\Delta$, and of $\Delta \sqcup \{ 0 \}$, for arbitrary Kac--Moody $\mathfrak{g}$. In this case, we show that all of these four notions agree for every indecomposable Kac--Moody $\mathfrak{g}$, with precisely two exceptions: if $X = \Delta$, and $\mathfrak{g} = \mathfrak{sl}_3(\mathbb{C})\text{ or } \widehat{\mathfrak{sl}_3(\mathbb{C})}$. Before writing the complete classification in all cases, we mention some notation required in the finite and affine cases.
	
	For a finite type root system $\Delta$, we denote the set of short roots, long roots, and the highest root of $\Delta$ respectively by $\Delta_s$, $\Delta_l$, and $\theta$; with the understanding that $\Delta_s=\Delta_l$ when $\Delta$ is simply laced. For the rest of this paragraph, assume that $\Delta$ is of affine type. Let $\ell\in \mathbb{N}$, $\mathcal{I}=\{0,\ldots,\ell\}$, and let $\delta$ be the smallest positive imaginary root of $\Delta$. Recall: (1)~the subroot system $\mathring{\Delta}$ generated by $\alpha_i$, $\forall$ $i\in \{1,\ldots\ell\}$ is of finite type from tables Aff 1--Aff 3 and Subsection 6.3 of Kac's book~\cite{Kac}; (2)~the roots in $\Delta$ are explicitly describable in terms of the roots in $\mathring{\Delta}$ by the well known result \cite[Proposition 6.3]{Kac}. Let the Weyl group of $\mathring{\Delta}$ be denoted by $\mathring{W}$. For $\Delta$ of type $X_{\ell}^{(r)}$, where $r\in\{1,2,3\}$---see \cite[Tables Aff 1--Aff 3]{Kac}---and for $Y\subseteq \Delta\sqcup\{0\}$, we define
	\begin{equation}\label{E2.4}
	Y_s:=\begin{cases}
	(Y\cap\mathring{\Delta}_s)+\mathbb{Z}\delta &\text{if } Y\cap \mathring{\Delta}_s\neq\emptyset,\\
	\emptyset &\text{if }Y\cap\mathring{\Delta}_s=\emptyset,
	\end{cases}\qquad\quad  Y_l:=\begin{cases}
	(Y\cap\mathring{\Delta}_l)+r\mathbb{Z}\delta &\text{if } Y\cap \mathring{\Delta}_l\neq\emptyset,\\
	\emptyset &\text{if }Y\cap\mathring{\Delta}_l=\emptyset.
	\end{cases}
	\end{equation}   
\begin{rem}
    Let $\Delta$ be the root system of a Kac--Moody algebra. Note that $\Delta\sqcup\{0\}$ is 212-closed in itself, but not in $\Delta$ clearly. Similarly, observe that $\Delta$ is 212-closed in itself but not in $\Delta\sqcup\{0\}$, as the following equation by the definition of a 212-closed subset implies $0\in\Delta$ which is absurd.
    \[(\xi)+(-\xi)=2(0).\]
    In view of these, throughout the paper, we only study the equivalence of the proper 212-closed subsets (and the proper weak-$\mathbb{A}$-faces) of both $\Delta$ and $\Delta\sqcup\{0\}$. 
\end{rem}	
	We now present `most of' the classification of the weak-$\mathbb{A}$-faces and 212-closed sets, of both $\Delta$ and $\Delta \sqcup \{ 0 \}$, in two results for ease of exposition.
	\begin{theorem}
		Let $\mathfrak{g}$ be an indecomposable Kac--Moody algebra with root system $\Delta$.
		\begin{itemize}
			\item[(a)] Assume that $\mathfrak{g}$ is of finite type and $\mathfrak{g}\neq\mathfrak{sl}_3(\mathbb{C})$. Then the proper 212-closed subsets of $\Delta$ are the same as those of $\Delta\sqcup\{0\}$. Therefore, the proper 212-closed subsets of $\Delta$ are of the form $w\big[(\theta-\mathbb{Z}_{\geq 0}\Pi_I)\cap \Delta\big]$ for $w\in W$ and $I\subsetneqq\mathcal{I}$.
			\item[(b)] When $\mathfrak{g}$ is of affine type, the following are equivalent: 
			\begin{itemize}
				\item[1.] $Y$ is a proper 212-closed subset of $\Delta$ (respectively of $\Delta\sqcup\{0\}$). 	          \item[2.] $Y=Y_s\cup Y_l$, and $Y\cap\mathring{\Delta}$ \big(respectively $Y\cap( \mathring{\Delta}\sqcup\{0\})$\big) is a proper 212-closed subset of $\mathring{\Delta}$ \big(respectively of $\mathring{\Delta}\sqcup\{0\}$\big). 
			\end{itemize}
			\item[(c)] When $\mathfrak{g}$ is of indefinite type, $\Delta$ (respectively $\Delta\sqcup\{0\}$) is the only 212-closed subset of $\Delta$ (respectively of $\Delta\sqcup\{0\}$).
		\end{itemize}
	\end{theorem}
	\begin{cor}\label{C2.5}
		Let $\mathfrak{g}$ be an indecomposable Kac--Moody algebra. 
		\begin{itemize}
			\item[(1)] The 212-closed subsets and weak-$\mathbb{A}$-faces of $\Delta\sqcup\{0\}$ are the same. Further, the proper 212-closed subsets (respectively proper weak-$\mathbb{A}$-faces) of $\Delta\sqcup\{0\}$ are also 212-closed subsets (respectively proper weak-$\mathbb{A}$-faces) of $\Delta$. 
			\item[(2)] For $\mathfrak{g}\neq \mathfrak{sl}_3(\mathbb{C})$ or $\widehat{\mathfrak{sl}_3(\mathbb{C})}$, the 212-closed subsets and weak-$\mathbb{A}$-faces of $\Delta$ are the same.   
		\end{itemize}
	\end{cor}
	With these results to provide an `initial impression', we now present our final main result, which subsumes both of them, and completes the above classification.
	\begin{thmx}\label{thmC}
		Let $\mathfrak{g}$ be an indecomposable Kac--Moody algebra with root system $\Delta$. Fix a non-zero additive subgroup $\mathbb{A}$ of $(\mathbb{R},+)$. 
		\begin{itemize}
			\item[(\ref{thmC}1)] Assume that $\mathfrak{g}$ is of finite type and $\mathfrak{g}\neq \mathfrak{sl}_3(\mathbb{C})$. Then the proper 212-closed subsets of $\Delta$, the proper 212-closed subsets of $\Delta\sqcup\{0\}$, the proper weak-$\mathbb{A}$-faces of $\Delta$ and the proper weak-$\mathbb{A}$-faces of $\Delta\sqcup\{0\}$ are all the same, and moreover they are of the form $w\big[(\theta-\mathbb{Z}_{\geq 0}\Pi_I)\cap \Delta\big]$ for $w\in W$ and $I\subsetneqq\mathcal{I}$.
			\item[(\ref{thmC}2)] Assume that $\mathfrak{g}=\mathfrak{sl}_3(\mathbb{C})$, and let $\mathcal{I}=\{1,2\}$. Then the proper 212-closed subsets of $\Delta\sqcup\{0\}$, the proper weak-$\mathbb{A}$-faces of $\Delta$ and the proper weak-$\mathbb{A}$-faces of $\Delta\sqcup\{0\}$ are all the same, and they are of the form $w\big[(\theta-\mathbb{Z}_{\geq 0}\Pi_I)\cap \Delta\big]$ for $w\in W$ and $I\subsetneqq\mathcal{I}$. Furthermore, every such subset is 212-closed in $\Delta$, and all the $W$-conjugates of the following three sets comprise the rest of the proper 212-closed subsets of $\Delta$.    
			\[\Pi=\{\alpha_1,\alpha_2\},\qquad\qquad\Delta^+=\{\alpha_1,\alpha_2,\alpha_2+\alpha_1\},\qquad\qquad \{-\alpha_1,-\alpha_2,\alpha_2+\alpha_1\}.\]
			\item[(\ref{thmC}3)] Assume that $\mathfrak{g}$ is of affine type. Then the following two statements a) and b) are equivalent: 
			\begin{itemize}
				\item[a)] $Y$ is a proper 212-closed subset of $\Delta$ (respectively of $\Delta\sqcup\{0\}$). 	          \item[b)] There exists $Z\subset\mathring{\Delta}$ such that $Z$ is a proper 212-closed subset of $\mathring{\Delta}$ \big(respectively of $\mathring{\Delta}\sqcup\{0\}$\big) and $Y=Z_s\cup Z_l$. 
			\end{itemize}
			When $\mathfrak{g}\neq \widehat{\mathfrak{sl}_3(\mathbb{C})}$, in view of part (\ref{thmC}1), the proper 212-closed subsets of $\Delta$, the proper 212-closed subsets of $\Delta\sqcup\{0\}$, the proper weak-$\mathbb{A}$-faces of $\Delta$ and the proper weak-$\mathbb{A}$-faces of $\Delta\sqcup\{0\}$ are all the same. When $\mathfrak{g}=\widehat{\mathfrak{sl}_3(\mathbb{C})}$---very similar to part (\ref{thmC}2)---the proper 212-closed subsets of $\Delta\sqcup\{0\}$, the proper weak-$\mathbb{A}$-faces of $\Delta$ and the proper weak-$\mathbb{A}$-faces of $\Delta\sqcup\{0\}$ are all the same, and additionally all of these subsets are also 212-closed in $\Delta$.
			\item[(\ref{thmC}4)] Assume that $\mathfrak{g}$ is of indefinite type. Then $\Delta$ (respectively $\Delta\sqcup\{0\}$) is the only 212-closed subset of $\Delta$ (respectively of $\Delta\sqcup\{0\}$). Similarly, $\Delta$ (respectively $\Delta\sqcup\{0\}$) is the only weak-$\mathbb{A}$-face of $\Delta$ (respectively of $\Delta\sqcup\{0\}$).
		\end{itemize}
	\end{thmx}
	\begin{observation}\label{R5.1}\hfill
	\begin{itemize}
	\item[(1)] Let $\mathfrak{g}$ and $\mathbb{A}$ be as in Theorem \ref{thmC}, and $\emptyset \neq Y$ be a proper weak-$\mathbb{A}$-face of $\Delta$. Then:
	\begin{equation}\label{E'2.4}
	\sum\limits_{i=1}^nr_iy_i=\sum\limits_{j=1}^mt_jx_j\quad\text{for some }r_i,t_j\in\mathbb{A}_{\geq 0},\text{ }y_i\in Y,\text{ }x_j\in \Delta\implies \sum\limits_{i=1}^nr_i\leq \sum\limits_{j=1}^m t_j.
	\end{equation}\allowdisplaybreaks
	 \big(Recall, this implication is similar to the one in Property \eqref{property1}.\big) This can be easily verified as follows. Suppose on the contrary that  $\sum\limits_{i=1}^nr_i> \sum\limits_{j=1}^m t_j$. Since $Y$ is also a proper weak-$\mathbb{A}$-face of $\Delta\sqcup\{0\}$ by Theorem \ref{thmC}, the following implication contradicts $Y\subset \Delta$.\smallskip
	 \[
	 \sum\limits_{i=1}^nr_iy_i= \left(\sum\limits_{i=1}^nr_i- \sum\limits_{j=1}^m t_j\right)(0)+ \sum\limits_{j=1}^m t_jx_j\implies 0\in Y.
	 \]
		\item[(2)] When $\Delta$ is of affine type, by part (\ref{thmC}3) of Theorem \ref{thmC}, the 212-closed subsets of $\Delta$ (or of $\Delta\sqcup\{0\}$) are in bijective correspondence with those of $\mathring{\Delta}$ (respectively those of $\mathring{\Delta}\sqcup\{0\}$). 
		\end{itemize}
	\end{observation}
	Observe that so far we have only provided the classification results when $\Delta$ is indecomposable. The following proposition extends these to all Kac--Moody $\mathfrak{g}$. 
	\begin{prop}\label{P2.7}
		Let $\Delta$ be not necessarily indecomposable, and let $\Delta_1\sqcup\cdots\sqcup\Delta_n$ be its decomposition into indecomposable subroot systems ($n=1$ when $\Delta$ is indecomposable). Fix $\emptyset\neq Y\subseteq\Delta$.
		\begin{itemize}
			\item[(a)] If $Y$ is 212-closed in $\Delta$, then $Y\cap\Delta_i$ is 212-closed in $\Delta_i$ for each $1\leq i\leq n$.
			\item[(b)] More generally, $Y$ is proper 212-closed in $\Delta$ if and only if $Y\cap\Delta_i$ is proper 212-closed in $\Delta_i$ for each $1\leq i\leq n$.
		\end{itemize}
		The analogous statements about the weak-$\mathbb{A}$-faces of $\Delta$, 212-closed subsets of $\Delta\sqcup\{0\}$, and weak-$\mathbb{A}$-faces of $\Delta\sqcup\{0\}$ are also true. 
	\end{prop}
	\begin{rem}
		Let $\Delta$ and $\Delta_i$, $\forall$ $i \in [n]$, and $Y$ be as in Proposition \ref{P2.7}.	a) The converse to~part~(1) of the proposition fails to hold, e.g. for $\Delta = \Delta_1 \sqcup \Delta_2$ and $Y = \Delta_1 \sqcup Y_2$, where $Y_2$ is some proper 212-closed subset of $\Delta_2$. b) If $Y$ is proper 212-closed in $\Delta$, in view of Theorem \ref{thmC} part (\ref{thmC}4), necessarily $Y\cap \Delta_i=\emptyset$ whenever $\Delta_i$ is of indefinite type. 
	\end{rem}
	\begin{proof}[\textnormal{\textbf{Proof of Proposition \ref{P2.7}}}] We prove the assertions in the proposition for 212-closed subsets and for weak-$\mathbb{A}$-faces of $\Delta$ separately in two steps below. The same results about the 212-closed subsets and weak-$\mathbb{A}$-faces of $\Delta\sqcup\{0\}$ hold by the same proofs in Steps 1 and 2, with $\Delta\sqcup\{0\}$ and $\Delta_i\sqcup\{0\}$ in place of $\Delta$ and $\Delta_i$, respectively, for all $i$.  \smallskip\\
		\textbf{Step 1:} Assume that $Y$ is non-empty and 212-closed in $\Delta$. Now $\Delta=\Delta_1\sqcup\cdots\sqcup\Delta_n$ implies that $Y=\big(Y\cap\Delta_1\big)\sqcup\cdots\sqcup\big(Y\cap \Delta_n\big)$. By the definition of 212-closed subsets, it follows that $(Y\cap \Delta_i)$ is 212-closed in $\Delta_i$ for each $i\in [n]$, which proves part (a). Next, regarding the proper 212-closed subsets of $\Delta$, suppose $\Delta_j\subset Y$ for some $j\in [n]$. Then by the 212-closedness of $Y$, the equations
		\[(\alpha)+(-\alpha)=(\beta)+(-\beta),\qquad \text{for all }\text{ }\beta\in\Delta\text{ and }\alpha\in\Delta_j\]
		imply $Y=\Delta$. This proves the forward implication in part (b). For the reverse implication, let $Y_i\subset\Delta_i$ and $Y_i$ be proper 212-closed in $\Delta_i$ for each $i\in [n]$. We show that $Y_1\sqcup\cdots\sqcup Y_n$ is proper 212-closed in $\Delta$, which implies the result. Notice, $Y_1\sqcup\cdots \sqcup Y_n\subsetneqq\Delta$. Assume that 
		\[(\alpha_1)+(\alpha_2)=(\beta_1)+(\beta_2)\qquad \text{for some }\alpha_1,\alpha_2\in Y_1\cup\cdots\cup Y_n\text{ and }\beta_1,\beta_2\in\Delta. \]
		 Showing $\beta_1,\beta_2\in Y$ proves $Y_1\sqcup\cdots\sqcup Y_n$ is 212-closed in $\Delta$, as $\alpha_1,\alpha_2,\beta_1,\beta_2$ are arbitrary. For this, recall that $\mathbb{R}\Delta_k\cap\mathbb{R}\Delta_l=\{0\}$ for any $1\leq k\neq l\leq n$, as $\Delta_1,\ldots,\Delta_n$ are the irreducible components of $\Delta$. This leads to two cases: i)~all the four roots $\alpha_1,\alpha_2,\beta_1\text{ and }\beta_2$ belong to $\Delta_k$ for some $k\in [n]$, or ii)~the pair $(\alpha_1,\alpha_2)=(\beta_1,\beta_2)$ or $(\beta_2,\beta_1)$. If ii) holds, then $\beta_1,\beta_2\in Y_1\sqcup\cdots\sqcup Y_n$ trivially. Else if i) holds, then $\beta_1,\beta_2\in Y_1\sqcup\cdots\sqcup Y_n$ by the 212-closedness of $Y_k$ in $\Delta_k$.\smallskip\\
		\textbf{Step 2:} Let $\{0\}\subsetneqq \mathbb{A}$ be a subgroup of $(\mathbb{R},+)$. The proofs of the analogous assertion in part~(a)~and the forward implication in part (b) for weak-$\mathbb{A}$-faces of $\Delta$ are similar to those of for the 212-closed~subsets of $\Delta$ in Step 1. For each $i\in [n]$, let $Y_i$ be a proper weak-$\mathbb{A}$-face of $\Delta_i$. We show that $Y_1\sqcup\cdots\sqcup Y_n$ is a weak-$\mathbb{A}$-face of $\Delta$, which would finish the proof of part (b) for the weak-$\mathbb{A}$-faces of $\Delta$. For simplicity, we will assume that $n=2$; when $n>2$, the result can be proved in a similar manner. Let
	 	\begin{align}\label{E2.5}
		\begin{aligned}[c]
		\sum_{i=1}^p r_i y_i=\sum_{j=1}^q t_j x_j\quad\text{and}\quad\sum_{i=1}^p r_i =\sum_{j=1}^q t_j>0 
		\end{aligned}\quad\quad
		\begin{aligned}[c]
		&\text{for some }y_i\in Y_1\sqcup Y_2,\text{ }x_j\in\Delta,\\
		&r_i,t_j\in\mathbb{A}_{\geq 0},\text{ }
		 p,q\in\mathbb{N}.
		\end{aligned}
		\end{align}
		We are required to show $x_j\in Y_1\sqcup Y_2$ $\forall$ $j$ such that $t_j>0$. Observe that $\mathbb{R}\Delta=\mathbb{R}\Delta_1\oplus\mathbb{R}\Delta_2$ leads to two cases: i) for some $t\in \{1,2\}$, $y_i,x_j$ belong to $\Delta_t$ for all 
      $i\in [p]$, $j\in [q]$, or ii) we will have four non-empty subsets $P_1,P_2\subset [p]\text{ and }Q_1,Q_2\subset [q]$ as follows: 
      \begin{equation} 
      P_k:=\{i\in [p]\text{ }|\text{ }y_i\in Y_k\},\text{ } Q_k:=\{j\in [q]\text{ }|\text{ }x_j\in \Delta_k\},\text{ }\text{so that }\sum\limits_{i\in P_k}r_iy_i=\sum\limits_{j\in Q_k}t_jx_j\text{ }\forall \text{ }k\in [2].\end{equation}
    If i) holds, then we are done as $Y_t$ is a weak-$\mathbb{A}$-face of $\Delta_t$. Else if ii) holds, then as each $Y_k$ is a proper weak-$\mathbb{A}$-face of $\Delta_k$, \eqref{E'2.4} results in $\sum\limits_{i\in P_k}r_i\leq \sum\limits_{j\in Q_k}t_j$ for each $k\in [2]$. Now, the definitions~of $P_k$ and $Q_k$, and the equality $\sum\limits_{i\in [p]}r_i=\sum\limits_{j\in [q]}t_j$ together force $\sum\limits_{i\in P_k}r_i=\sum\limits_{j\in Q_k}t_j$ $\forall$ $k\in [2]$. Finally, again as each $Y_k$ is a weak-$\mathbb{A}$-face of $\Delta_k$, $x_j\in Y_1\sqcup Y_2$ $\forall$ $j\in Q_k$ such that $t_j>0$, completing~the~proof. 
    \end{proof}
	In view of Proposition \ref{P2.7} and Theorem \ref{thmC}, we have therefore fully solved the problems of classifying the $(\{2\};\{1,2\})$-closed subsets and weak-$\mathbb{A}$-faces of $\Delta$ and also $\Delta\sqcup\{0\}$ for arbitrary (not necessarily indecomposable) Kac--Moody root systems $\Delta$.
	
	The following remark extends the above main results of this paper to related algebras $\tilde{\mathfrak{g}}$.
	\begin{rem}[Results over related Kac--Moody algebras] Given a generalized Cartan matrix $A$, let $\bar{\mathfrak{g}}(A)$ denote the Lie algebra generated by $e_i, f_i, \mathfrak{h}$ modulo only the Serre relations, and let the Kac--Moody Lie algebra $\mathfrak{g}=\mathfrak{g}(A)$ be the further quotient of $\overline{\mathfrak{g}}(A)$ by the largest ideal intersecting $\mathfrak{h}$ trivially. The above results---and results in the later sections---are stated and proved over the Kac--Moody algebra $\mathfrak{g}(A)$, but they hold equally well over the Kac--Moody algebra $\overline{\mathfrak{g}}(A)$ and hence uniformly over any `intermediate' algebra $\overline{\mathfrak{g}}(A)\twoheadrightarrow\tilde{\mathfrak{g}}\twoheadrightarrow\mathfrak{g}=\mathfrak{g}(A)$. This is because as was clarified in \cite{Dhillon_arXiv}, the results there hold over all such Lie algebras $\tilde{\mathfrak{g}}$; similarly, the root system and the weights of highest weight modules (of a fixed highest weight $\lambda$) remain unchanged over all $\tilde{\mathfrak{g}}$. \end{rem}
	As a final touch, in Section 7, we address the following problem of determining the 212-closed subsets of $\Delta\sqcup\{0\}$ consisting only of positive roots in the semisimple case.
	\begin{problem}\label{prob1} Let $\mathfrak{g}$ be semisimple with root system $\Delta$. Let $Y$ be a proper non-empty 212-closed subset of $\Delta\sqcup\{0\}$. Determine the set $\{w\in W$ $|$ $w Y\subset \Delta^+\}$.
	\end{problem}
	This problem needs attention because Chari and Greenstein studied properties \eqref{property1} and \eqref{property2} (see the Introduction) for subsets of $\Delta\sqcup\{0\}$ consisting only of positive roots, when $\Delta$ is finite. Also, notice that the above problem is stated only for semisimple $\mathfrak{g}$. This is because when $\mathfrak{g}$ is not semisimple, by the explicit descriptions for the 212-closed subsets of $\Delta\sqcup\{0\}$---see for affine and indefinite types from Theorem \ref{thmC}---they will contain negative roots.
	\section{Preliminaries}
	The goal of this section is to state some preliminary results which are needed in the paper. Before we proceed, we give (more) notation, in addition to the notation given in Subsection 2.1, which is needed here and also later. Throughout this section, unless otherwise stated, we will assume that $\mathfrak{g}$ is a Kac--Moody algebra with a fixed Cartan subalgebra $\mathfrak{h}$, root system $\Delta$, and simple roots $\Pi$.  
	\subsection{Notation} 	
	For $h\in\mathfrak{h}$ and $\mu\in\mathfrak{h}^*$, we define $\langle \mu ,h \rangle$ to be the evaluation of $\mu$ at $h$, which we also denote by $\mu(h)$. We define $P$ and $P^+$ to be the sets of integral and dominant integral weights, respectively. For a real root $\alpha$, we define $\mathfrak{sl}_{\alpha}:=\mathfrak{g}_{-\alpha}\oplus \mathbb{C}\alpha^{\vee}\oplus\mathfrak{g}_{\alpha}\simeq \mathfrak{sl}_2(\mathbb{C})$. We occasionally use $s_{\alpha}$ for the reflection in $W$ about the hyperplane perpendicular to $\alpha$. For any non-zero element $x\in U(\mathfrak{g})$, we treat $x^0$ as $1\in U(\mathfrak{g})$.

	Fix $\emptyset\neq I\subseteq\mathcal{I}$. Let $\mathfrak{l}_I= \mathfrak{g}_{I}+\mathfrak{h}$ and $\mathfrak{p}_I=\mathfrak{g}_I +\mathfrak{h}+\mathfrak{n}^+$ be the standard Levi and parabolic Lie subalgebras of $\mathfrak{g}$ corresponding to $I$, respectively. Suppose $V$ is a highest weight $\mathfrak{p}_I$ or $\mathfrak{l}_I$-module with highest weight $\lambda\in\mathfrak{h}^*$. Then $V$ becomes a highest weight $\mathfrak{g}_I$-module with highest weight $\lambda\big|_{\mathfrak{h}_I}$ the restriction of $\lambda$ to $\mathfrak{h}_I$. But for simplicity, throughout the paper we also denote the highest weight of $V$ for the $\mathfrak{g}_I$-action by $\lambda$.
	
	Let $\prec$ be the usual partial order on $\mathfrak{h}^*$, under which $x\prec y \in \mathfrak{h}^* \iff y -x \in \mathbb{Z}_{\geq 0}\Pi$. Fix $\emptyset\neq I\subseteq\mathcal{I}$, $n\in\mathbb{Z}$, and a vector $x=\sum\limits_{i\in \mathcal{I}}c_i \alpha_i$ for some $c_i \in \mathbb{C}$. We define 
	\begin{align}\label{I-Notation}
	\begin{aligned}
	&\supp(x):=\{i \in \mathcal{I} \text{ }|\text{ } c_i \neq 0\},\\
	&\supp_I(x):=\{i\in I \text{ }|\text{ } c_i \neq 0\},\\
	&\height(x):=\sum_{i\in \mathcal{I}}c_i,\quad \height_{I}(x):=\sum_{i\in I}c_i,
	\end{aligned}
	\qquad\qquad
	\begin{aligned}
	&\Delta_{I,n}:=\{\beta \in \Delta \text{ }|\text{ } \height_I(\beta)=n\},\text{ and}\\
	&\text{for convenience, }\text{ }\supp(0)=\Delta_{\emptyset,n }:=\emptyset\\
	&\text{and }\mathbb{Z}\Delta_{\emptyset,n}=\mathbb{R}\Delta_{\emptyset,n}:=\{0\}.\\
	&\text{Note that }\Delta_{I,0}=\Delta_{I^c}=\Delta\cap\mathbb{Z}\Pi_{I^c}.
	\end{aligned}
	\end{align}
	
	For $\alpha\in\Delta,\text{ } k\in \mathbb{Z}_{\geq0}, \text{ and } \mu\in\mathfrak{h}^*$, we define \begin{equation}\label{INTVL} [\mu-k\alpha,\mu]:=\{\mu-j\alpha \text{ }|\text{ } j=0,\ldots,k\},\quad\text{ } (\mu-k\alpha,\mu):=[\mu-k\alpha,\mu]\setminus\{\mu-k\alpha,\mu\},\quad\text{ }  (\mu,\mu)=\emptyset.\end{equation}
	The notation $[\mu-k\alpha,\mu]$ should not be confused with the Lie bracket. Recall, we used (.,.) to denote the invariant form on $\mathfrak{g}$ when $\mathfrak{g}$ is symmetrizable, while in the above equation, we used it to denote open intervals. Nevertheless, in the course of the paper, it will be clear from the context as to which one (.,.) refers to, i.e., the invariant form or an open interval. 
	
	Throughout the paper, for $\lambda_1,\ldots,\lambda_n\in\mathfrak{h}^*$, $n\geq 2$, and $U\subset\mathfrak{h}^*$, the notation
	\begin{equation}\label{E3.3}
	\lambda_1\prec\cdots\prec\lambda_i\prec\cdots\prec\lambda_n\in\text{ }U\text{ }\forall i\text{ }\text{ }
	\text{denotes}: (1) \text{ }\lambda_{i-1}\prec\lambda_{i}\text{ }\forall\text{ }i>1,\text{ and }(2)\text{ }\lambda_i\in U\text{ }\forall\text{ }i\in [n].
	\end{equation}
	\subsection{Preliminaries} We begin by listing some useful results for our paper about root systems and highest weight modules over Kac--Moody algebras, which we will use without further mention. 
	\begin{itemize}
		\item[(R1)]\label{(R1)} Let $\Delta$ be the root system of $\mathfrak{g}$, and fix $\emptyset\neq I\subseteq\mathcal{I}$. From \cite{Teja}, recall the \textit{parabolic partial sum property (parabolic-PSP)} of $\Delta$: every root $\beta$ with $\height_{I}(\beta)>0$ can be written as an ordered sum of roots from $\Delta_{I,1}$ such that each partial sum (of the ordered sum) is also a root. In particular, $\mathbb{Z}_{\geq 0}\left(\Delta^+\setminus\Delta^+_{I^c}\right)=\mathbb{Z}_{\geq 0}\Delta_{I,1}$; see \eqref{I-Notation} for the notation. When $I=\mathcal{I}$, $\Delta_{I,1}$ equals $\Pi$, and the parabolic-PSP in this case is the same as the usual \textit{partial sum property}. 
		\item[(R2)]\label{(R2)} Let $\mathfrak{g}$ be a Kac--Moody algebra, $\lambda\in\mathfrak{h}^*$, $M(\lambda)\twoheadrightarrow V$, and $0\neq v_{\lambda}\in V_{\lambda}$ be a highest weight vector. Then for $\mu\in\wt V$ the weight space $V_{\mu}$ is spanned by the weight vectors of the form $f_{i_1}\cdots f_{i_n}v_{\lambda}$ such that $f_{i_j}\in\mathfrak{g}_{-\alpha_{i_j}}$, $i_j\in\mathcal{I}$ and $\alpha_{i_j}\in\Pi$ for each $j\in [n]$, and $\sum\limits_{j=1}^{n}\alpha_{i_j}=\lambda-\mu$.
	\end{itemize}
	\begin{lemma}\label{L3.1}
		Let $\mathfrak{g}$ be a Kac--Moody algebra, and $V$ be a highest weight $\mathfrak{g}$-module. Fix $\mu\in\wt V$. Suppose $\langle\mu,\alpha_i^{\vee}\rangle\geq 0$ for some $\alpha_i\in\Pi$, $i\in\mathcal{I}$. Then $\big[\mu-\lceil\langle\mu,\alpha_i^{\vee}\rangle\rceil\alpha_i,\mu\big]\subseteq\wt V$; where $\lceil \cdot\rceil$ denotes the ceiling function, and for the interval notation see \eqref{INTVL}. In particular, $s_i\mu=\mu-\langle\mu,\alpha_i^{\vee}\rangle\alpha_i$ belongs to $\wt V$ when $\langle\mu,\alpha_i^{\vee}\rangle\in\mathbb{Z}_{\geq 0}$.
	\end{lemma}
	\begin{proof}[\textnormal{\textbf{Proof}}]
		Let $\mu$ and $i$ be as in the statement. Fix a non-zero vector $x\in V_{\mu}$. Recall, $e_i$ acts locally nilpotently on $V$, in particular, $e_i$ acts nilpotently on $x$. So, there exists $n\in\mathbb{Z}_{\geq 0}$ such that $e_i^n x\neq 0$ and $e_i^{n+1}x=0$. It follows that $U(\mathfrak{g}_{-\alpha_i})(e_i^{n}x)$ is a highest weight module over $\mathfrak{sl}_{\alpha_i}:=\mathfrak{g}_{-\alpha_i}\oplus \mathbb{C}\alpha_i^{\vee}\oplus\mathfrak{g}_{\alpha_i}$, with highest weight $\mu+n\alpha_i$, or more precisely $\big(\mu+n\alpha_i\big)\big|_{\mathbb{C}\alpha_i^{\vee}}$. Since $\langle\mu,\alpha_i^{\vee}\rangle\geq 0$, note that if $\langle\mu+n\alpha_i,\alpha_i^{\vee}\rangle=\langle\mu,\alpha_i^{\vee}\rangle+2n\in\mathbb{Z}_{\geq 0}$, then $\langle\mu,\alpha_i^{\vee}\rangle\in\mathbb{Z}_{\geq 0}$, which implies $\lceil\langle\mu,\alpha_i^{\vee}\rangle\rceil=\langle\mu,\alpha_i^{\vee}\rangle$. By $\mathfrak{sl}_2$-theory, observe that $f_i^{k}(e_i^nx)\neq 0$ for: (1)~any $k\in\mathbb{Z}_{\geq 0}$, if $\langle\mu,\alpha_i^{\vee}\rangle+2n\notin \mathbb{Z}_{\geq 0}$; (2)~any $0\leq k\leq \langle\mu,\alpha_i^{\vee}\rangle+2n$, if $\langle\mu,\alpha_i^{\vee}\rangle+2n\in\mathbb{Z}_{\geq 0}$. In both of these cases, in particular,
		\[f_i^k(e_i^nx)\neq 0\quad \text{for all }n\leq k\leq n+\lceil \langle\mu,\alpha_i^{\vee}\rangle\rceil \implies \big[\mu-\lceil\langle\mu,\alpha_i^{\vee}\rangle\rceil\alpha_i,\mu\big]\subseteq \wt V.\]
		This completes the proof of the lemma.
	\end{proof}
	
	In \cite{Teja}, it has been emphasized that the study of/ working with the sets $\Delta_{I,1}$ for $\emptyset\neq I\subseteq\mathcal{I}$ has many applications to representation theory, for instance to the weights of arbitrary highest weight $\mathfrak{g}$-modules. In the present paper, we deal with these sets at many places in the special cases $I=\mathcal{I}_V\text{ and }\mathcal{I}_Y\setminus\mathcal{J}_Y$, for non-empty 212-closed subsets $Y$ of $\wt V$.
	
	Next, we define parabolic Verma modules over $\mathfrak{g}$. Parabolic Verma modules, also known as generalized Verma modules, were introduced and studied by Lepowsky in a series of papers; see~\cite{L_JA} and the references therein. For a detailed list of their properties, see \cite[Chapter 9]{Hump_BGG}. Recall the definition of $J_{\lambda}$ from \eqref{E1.2}.\allowdisplaybreaks
	\begin{definition}\label{D2.1}
		\begin{itemize}
			\item[(1)] Let $\lambda\in\mathfrak{h}^*$, $\emptyset\neq J\subseteq J_{\lambda}$, and $M(\lambda)$ denote the Verma module over $\mathfrak{g}$ with highest weight $\lambda$. Throughout the paper, let $0\neq m_{\lambda}\in M(\lambda)_{\lambda}$ denote a highest weight vector of $M(\lambda)$. Let $L^{\max}_J(\lambda)$ denote the largest integrable highest weight module over $\mathfrak{g}_J$---or equivalently over $\mathfrak{p}_J$, via the natural action of $\mathfrak{h}$ and trivial action of $\bigoplus\limits_{\alpha\in\Delta^+\setminus\Delta^+_J} \mathfrak{n}^+_{\alpha}$---with highest weight $\lambda$. Let $L_J(\lambda)$ denote the simple highest weight module over $\mathfrak{g}_J$ (or equivalently $\mathfrak{p}_J$) with highest weight $\lambda$. Note that $L_J(\lambda)$ is integrable over $\mathfrak{g}_J$. Recall that when $\mathfrak{g}_J$ is symmetrizable, by \cite[Corollary 10.4]{Kac} $L^{\max}_J(\lambda)=L_J(\lambda)$.
			\item[(2)] For $J\subseteq J_{\lambda}$, we denote the
			parabolic Verma module corresponding to $J$ (and $\lambda$) by
			\[
			M(\lambda,J):=U(\mathfrak{g})\otimes_{U(\mathfrak{p}_J)}L^{\max}_J(\lambda)
			\simeq M(\lambda)\Bigg/\Bigg(\sum_{j\in
				J}U(\mathfrak{g})f_{j}^{\langle\lambda,\alpha_j^{\vee}\rangle+1}m_{\lambda}\Bigg).\]
			Note, $L^{\max}_J(\lambda)$ is the
			parabolic Verma module over $\mathfrak{g}_J$
			corresponding to $J\subseteq J_{\lambda}$ (and $\lambda$).
			\item[(3)] By definition and by the paragraph below, notice that $M(\lambda,J)$ is $\mathfrak{g}_J$-integrable, and therefore $\wt M(\lambda,J)$ is $W_J$-invariant.
		\end{itemize} 
	\end{definition}
	
	For $\lambda\in\mathfrak{h}^*$, $J_{\lambda}$ has the property: if $j\in J_{\lambda}$, then for any $M(\lambda)\twoheadrightarrow{ }V$ and $\mu\in\wt V$, $\langle\mu,\alpha_j^{\vee}\rangle\in\mathbb{Z}$. We now look more closely at the integrability of a $\mathfrak{g}$-module, which plays a very important role in the paper. Given $J\subseteq \mathcal{I}$, recall that a weight module $M$ of $\mathfrak{g}$ is said to be $\mathfrak{g}_J$-integrable if $e_j$ and $f_j$ act locally nilpotently on $M$ $\forall$ $j\in J$.\\ 
	Let $M(\lambda)\twoheadrightarrow V$ and $J\subseteq J_{\lambda}$. Then $V$ is $\mathfrak{g}_J$-integrable if $f_j$ acts locally nilpotently on $V$ $\forall$ $j\in J$. 
	\begin{lemma}\label{L2.2}
		\begin{itemize}
			\item[(1)] Let $j\in \mathcal{I}$. Then $f_j$ acts locally nilpotently on $V$ if and only if $f_j$ acts nilpotently on the highest weight space $V_{\lambda}$.
			\item[(2)] $V$ is $\mathfrak{g}_J$-integrable if and only if $\charac V$ is $W_J$-invariant.
		\end{itemize}
	\end{lemma}
	Recall, the integrability of $V$ is denoted by $I_V$ and was defined in \eqref{defn IV-JY}, and $W_{I_V}$ is called the \textit{integrable Weyl group} of $V$. So, $V$ is always $\mathfrak{g}_{I_V}$-integrable. 
	
	The results in the following section about 212-closed subsets of the weights of parabolic Verma modules hold true for a large class of highest weight modules---namely, the class of highest weight modules whose sets of weights coincide with those of parabolic Verma modules. By Theorems 2.8 and 2.9 of Dhillon--Khare \cite{Dhillon_arXiv}, this class contains: i)~all the simple highest weight modules and parabolic Verma modules; ii)~in particular, every highest weight module $V$ with highest weight $\lambda$ and integrability $I_V$ such that the Dynkin subdiagram on $J_{\lambda}\setminus I_V$ is ``complete''. 
	
	While working with $\wt M(\lambda,J)$ \big(and $\wt L(\lambda)$\big), we will extensively make use of the $\mathfrak{g}_{J}$-integrability of $M(\lambda,J)$ \big(and respectively the $\mathfrak{g}_{I_{L(\lambda)}}$-integrability of $L(\lambda)$\big) as well as the following formulas for their weights. See e.g. Theorem 2.8 and Proposition 3.7 of \cite{Dhillon_arXiv} for the proofs and consequences of these formulas. In the notation of our paper, note that $J_{\lambda}=I_{L(\lambda)}$.
	\begin{align}\label{E2.1}
	\begin{aligned}[t]
	\text{Minkowski decomposition:}
	\end{aligned}\qquad
	\begin{aligned}[t]
	\wt M(\lambda,J)&=\wt L_J(\lambda)-\mathbb{Z}_{\geq0}(\Delta^+\setminus\Delta^+_J)\\&=\wt L_J(\lambda)-\mathbb{Z}_{\geq0}\Delta_{J^c,1}\quad\text{(by parabolic-PSP)}.
	\end{aligned}
	\end{align}
	\begin{equation}\label{E2.2}
	\text{Integrable slice decomposition:}\qquad\wt M(\lambda,J)= \bigsqcup\limits_{\xi\in\mathbb{Z}_{\geq0}\Pi_{J^c}}\wt L_J(\lambda-\xi).
	\end{equation}
	\begin{equation}\label{E2.3}
	\wt L(\lambda)\text{ }=\text{ }\wt M(\lambda,I_{L(\lambda)})\text{ }=\text{ }\wt L_{J_{\lambda}}(\lambda)-\mathbb{Z}_{\geq0}(\Delta^+\setminus\Delta^+_{J_{\lambda}})\text{ }=\text{ }\wt L_{J_{\lambda}}(\lambda)-\mathbb{Z}_{\geq0}\Delta_{J_{\lambda}^c,1},
	\end{equation}
	where the final equality was shown in previous work \cite{Teja}, via the parabolic-PSP. The following are some conventions which will be followed in the whole course of the paper.
	\begin{convn} Unless otherwise indicated, $Y$ will always denote a (conjectural) non-empty weak-$\mathbb{A}$-face or a non-empty 212-closed subset of $X$, and $X$ will equal $\wt V$ or $\conv_{\mathbb{R}}\wt V$ or $\wt \mathfrak{g}=\Delta\sqcup\{0\}$ or $\Delta$. At many places in the paper, to show $Y$ (similarly, some subset) is a weak-$\mathbb{A}$-face or a 212-closed subset of $X$, we write equations similar to the ones in Definition \ref{D1.4}, such as:
		\[
		\sum_{i=1}^nr_iy_i=\sum_{j=1}^mt_jx_j,\qquad\qquad (y_1)+(y_2)=(x_1)+(x_2).
		\]
		In such equations, we will always write elements from $Y$ (respectively from $X$) on the left side (respectively on the right side) of the equations. We use parentheses around each term, as in the second equation above, while dealing with 212-closed subsets. This is to distinguish the two terms on each side of the equations, so as to avoid any confusion that might arise; for instance, in the second equation above in the particular case when $y_1+y_2$ (similarly, $x_1+x_2$) is already an element of $Y$ (respectively, of $X$). 
	\end{convn}
	\section{Useful results about 212-closed subsets of $\wt V$}\label{S4}
	In this section, we assume $\mathfrak{g}$ to be a Kac--Moody algebra with root system $\Delta$ and the set of nodes in its Dynkin diagram $\mathcal{I}$, $M(\lambda)\twoheadrightarrow V$ to be an arbitrary (non-zero) highest weight $\mathfrak{g}$-module, and $Y$ to be a non-empty 212-closed subset of $\wt V$. Our goals in this section are: 1)~to get hands-on experience of working with 212-closed subsets of $\wt V$ by observing various consequences of the defining condition for them; 2)~develop necessary tools for proving the main results of this paper, notably, to introduce and study the object $P(Y)$ defined in \eqref{EP(Y)}; 3)~discuss the factors that lead to the close relationship between the 212-closed subsets and the standard faces of $\wt V$. We will use without further mention, the fact that if $j\in J_{\lambda}$, then for any $M(\lambda)\twoheadrightarrow{ }V$ and $\mu\in\wt V$, $\langle\mu,\alpha_j^{\vee}\rangle\in\mathbb{Z}$.
	\subsection{Useful remarks about 212-closed subsets of $\wt V$}
	We begin by noting the following important remarks, Remarks \ref{R3.1}--\ref{R3.4}, about 212-closed subsets $Y$ of $\wt V$. These are deduced from the defining conditions/equations for 212-closed subsets. These remarks, and the equations in them, will be referred to in most of the proofs of the paper. We recommend that the reader thoroughly recall the definitions of $J_{\lambda}$, $\wt_I V$, $\mathcal{I}_V,\text{ }\mathcal{J}_V,\text{ }\mathcal{I}_Y\text{ and }\mathcal{J}_Y$ from \eqref{E1.2} and \eqref{defn IV-JY}.
	\begin{rem}\label{R3.1}
		Let $\alpha\in\Delta$, and $k>1$ be an integer. Fix $\mu\in\wt V$, and set $\mu'=\mu-k\alpha$. Assume that the interval $[\mu',\mu]\subset\wt V$. Suppose either a)~both $\mu,\mu'\in Y$, or b)~$(\mu',\mu)\cap Y\neq \emptyset$. Then $[\mu',\mu]\subset Y$, since for each $j\in (0,k)$,
		\begin{eqnarray}
		(\mu)+(\mu')&=&(\mu-j\alpha)+\big(\mu-(k-j)\alpha\big)\label{E''6.2},\\
		2(\mu-j\alpha)&=&\big(\mu-(j-1)\alpha\big)+\big(\mu-(j+1)\alpha\big).\label{E''6.3}
		\end{eqnarray} 
	\end{rem}
	\begin{rem}\label{R3.2}
		Let $\alpha\in\Delta$, and fix three (not necessarily distinct) weights $\mu,\mu',\mu''\in Y$.\\
		a) If $\mu'-\alpha,\mu+\alpha\in\wt V$, then they lie in $Y$. Similarly, if $\mu-\alpha,\mu''+\alpha\in\wt V$, then they lie in $Y$. These two lines hold true by the following equations.
		\begin{eqnarray}
		(\mu)+(\mu')&=&(\mu+\alpha)+(\mu'-\alpha),\label{E''6.4}\\
		(\mu)+(\mu'')&=&(\mu-\alpha)+(\mu''+\alpha).\label{E''6.5}
		\end{eqnarray}
		b) If $\mu'-\alpha,\mu''+\alpha\in\wt V$, then $(y+\mathbb{Z}\alpha)\cap\wt V\subseteq Y$ $\forall$ $y\in Y$, by equations (\ref{E''6.4}) and (\ref{E''6.5}) above (with $y$ in place of $\mu$).
	\end{rem}
	In Section 5, the above two remarks (which are for $\wt V$) will be extended to the `continuous' setting of this paper, i.e., for general convex sets, for e.g. $\conv_{\mathbb{R}}\wt V$.
	
	Next, for $\emptyset\neq Z\subseteq\wt V$, we discuss an important property of $\mathcal{J}_Z$ (see \eqref{defn IV-JY} for the notation). 
	\begin{rem}\label{R3.3} Let $\lambda\in\mathfrak{h}^*$, $0\neq m_{\lambda}$ be a highest weight vector of $M(\lambda)$, and $M(\lambda)\twoheadrightarrow{ }V$. For any $i\in\mathcal{I}$, it is an easy check that $e_i(f_i m_{\lambda})=\langle\lambda,\alpha_i^{\vee}\rangle m_{\lambda}$. This in particular proves that if $j\in \mathcal{J}_V$, then $\langle\lambda,\alpha_j^{\vee}\rangle$ must be 0. The vanishing of $V_{\lambda-\alpha_j}$ implies that $V$ is $\mathfrak{g}_{\{j\}}$-integrable. Therefore, $\mathcal{J}_V\subseteq I_V$, the integrability of $V$. Recall, by \cite[Theorem 1.1]{Dhillon_arXiv} both $\wt V$ and $\charac V$ are $W_{I_V}$-invariant. So, both $\wt V$ and $\charac V$ are in particular $W_{\mathcal{J}_V}$-invariant. Now, for any $Z\subseteq\wt V$, as $\mathcal{J}_Z\subseteq \mathcal{J}_V$, the previous sentence implies that $\wt V$, and more generally $\wt_{\mathcal{I}_Z}V$, is $W_{\mathcal{J}_Z}$-invariant.
	\end{rem}
	The following remark is about the connected components in the Dynkin subdiagram on $\mathcal{I}_Z$.
	\begin{rem}\label{R3.4}
		Assume that $Z\subseteq \wt V$ and $\mathcal{I}_Z\neq\emptyset$. Then every connected component of $\mathcal{I}_Z$ contains a node from $\mathcal{I}_Z\setminus \mathcal{J}_Z$. This can be verified as follows. Suppose $\emptyset\neq K\subseteq \mathcal{J}_Z$, and $K$ corresponds to a connected component in $\mathcal{I}_Z$. Pick a weight $\mu\in Z$ such that $K\cap\supp(\lambda-\mu)\neq~\emptyset$, which exists as $K\subseteq \mathcal{I}_Z$. Now, observe that $U\big(\mathfrak{n}^-_{\mathcal{I}_Z}\big)=U\big(\mathfrak{n}^-_{\mathcal{I}_Z\setminus K}\big)\otimes U\big(\mathfrak{n}^-_{K}\big)$, $K\subseteq \mathcal{J}_Z$, and $\mathfrak{n}^-_{\mathcal{J}_Z} V_{\lambda}=~\{0\}$ together force that every vector in $V_{\mu}$ must be 0. But this contradicts $\mu$ is a weight of $V$. 
	\end{rem}
	The above remarks more or less achieve our first goal of this section. Now, we proceed to our second goal, of developing the necessary tools for proving the main results. 
	\subsection{Roots which give rise to weights in $\wt V$} 
	The goal of this subsection is to answer the following question, which is quite fundamental in nature.
	\begin{question}\label{Q1}
		Let $\lambda\in\mathfrak{h}^*$, $M(\lambda)\twoheadrightarrow{ }V$, and $\gamma$ be a (positive) root. When is $\lambda-\gamma$ a weight of $V$?
	\end{question}
	The answer to this interesting question is crucial for proving many results of this paper, and is given in the following result.
	\begin{prop}\label{L3.6}
		Let $\mathfrak{g}$ be a Kac--Moody algebra, $\lambda\in\mathfrak{h}^*$, and $M(\lambda)\twoheadrightarrow V$. Then \[(\lambda-\Delta)\cap \wt V=\lambda-(\Delta^+\setminus\Delta_{\mathcal{J}_V}^+).\]
		In particular when $\lambda$ is simply regular (i.e., $\langle\lambda,\alpha_i^{\vee}\rangle\neq 0$ $\forall$ $i\in\mathcal{I}$), $\lambda-\Delta^+\subset\wt V$. More generally, for any $\emptyset\neq Y\subseteq \wt V$, $\lambda-\Delta_{\mathcal{I}_Y\setminus \mathcal{J}_Y,1}\subset \wt V$.
	\end{prop}
	Before we prove this proposition, we note the following result, which describes $\wt V$ in terms of $\mathbb{Z}_{\geq0}\Delta_{\mathcal{I}_V,1}$ for arbitrary $M(\lambda)\twoheadrightarrow{ }V$ (see \eqref{I-Notation} for the notation). 
	\begin{lemma}\label{L3.5}
		Let $\mathfrak{g}$ and $V$ be as in Proposition \ref{L3.6}. Let $0\neq v_{\lambda}$ be a highest weight vector in $V_{\lambda}$. For any $\mu\in\wt V$, the corresponding weight space $V_{\mu}$ is spanned by the weight vectors of the form: 
		\[
		\left(\prod\limits_{i=1}^{n}f_{\gamma_i}^{c_i}\right)v_{\lambda}\quad\text{such that }\text{ }c_i\in\mathbb{Z}_{\geq 0},\text{ }\gamma_i\in\Delta_{\mathcal{I}_V,1},\text{ }\text{ } \lambda-\sum\limits_{i=1}^{n}\gamma_i=\mu,\text{ }\text{ and }\text{ }0\neq f_{\gamma_i}\in\mathfrak{g}_{-\gamma_i}\text{ }\forall\text{ }1\leq i\leq n.\]
		Hence, $\wt V\subseteq \lambda-\mathbb{Z}_{\geq 0}\Delta_{\mathcal{I}_V,1}$.
	\end{lemma}
	\begin{proof}[\textnormal{\textbf{Proof}}]
		By \cite[Corollary 3.2]{Teja} (with $I=\mathcal{I}_V$ in it), $U(\mathfrak{n}^-)$ is spanned by the monomials of the form \begin{align*} 
		\begin{aligned}
		\left(\prod\limits_{i=1}^nf_{\gamma_i}^{c_i}\right)\left(\prod\limits_{ j=1}^mf_{\beta_j}^{d_j}\right)
		\end{aligned}\qquad
		\begin{aligned}
		&\text{where }m,n\in\mathbb{N},\text{ }\gamma_i\in\Delta_{\mathcal{I}_V,1},\text{ } \beta_j\in\Delta_{\mathcal{J}_V}^+,\text{ }f_{\gamma_i}\in\mathfrak{g}_{-\gamma_i} ,\\
		&f_{\beta_j}\in\mathfrak{g}_{-\beta_j},\text{ }
		(c_i,d_j)\in\left(\mathbb{Z}_{\geq 0}\right)^2\text{ }\forall\text{ }(i,j)\in [n]\times [m]. 
		\end{aligned}
		\end{align*}
		So, $V$ is spanned by the weight vectors of the form $\left(\prod_{i=1}^n f_{\gamma_i}^{c_i}\right)\left(\prod_{j=1}^m f_{\beta_j}^{d_j}\right)v_{\lambda}$. In particular, for $\mu\in\wt V$, $V_{\mu}$ is spanned by the weight vectors as in the previous line with $\sum\limits_{i=1}^nc_i\gamma_i+\sum\limits_{j=1}^md_j\beta_j=\lambda-\mu$. Observe that for $\left(\prod_{i=1}^n f_{\gamma_i}^{c_i}\right)\left(\prod_{j=1}^m f_{\beta_j}^{d_j}\right)v_{\lambda}$ to be non-zero, necessarily $d_j=0$ for all $j\in [m]$, as $\mathfrak{n}^-_{\mathcal{J}_V}v_{\lambda}=\{0\}$. This implies the result. 
	\end{proof}
	\begin{proof}[\textnormal{\textbf{Proof of Proposition \ref{L3.6}}}]
		Firstly, if $\mathcal{J}_V= \mathcal{I}$, then by the definition of $\mathcal{J}_V$, $\lambda$ and $V$ must be $0\in\mathfrak{h}^*$ and $L(0)$, respectively, and so the result is true in this case. So, we assume throughout the proof that $\mathcal{J}_V\subsetneqq\mathcal{I}$. The forward inclusion that $(\lambda-\Delta)\cap\wt V\subseteq \lambda-(\Delta^+\setminus\Delta_{\mathcal{J}_V}^+)$ is obvious, since $\mathfrak{n}^-_{\mathcal{J}_V}V_{\lambda}=\{0\}$. For the reverse inclusion, we prove:
		\[ 
		\beta\in\Delta^+\setminus\Delta^+_{\mathcal{J}_V}\implies\lambda-\beta\in\wt V\qquad\text{via induction on}\quad\height(\beta)\geq 1.\]
		In the base step $\height(\beta)=1$, $\beta$ must belong to $\Pi_{\mathcal{I}_{V}}$. So, $\lambda-\beta\in\wt V$ by the definition of $\mathcal{I}_V$.\\
		Induction step: Assume that $\height(\beta)>1$. Recall, $J_{\lambda}:=\{j\in \mathcal{I}$ $|$ $\langle\lambda,\alpha_{j}^{\vee}\rangle\in\mathbb{Z}_{\geq0} \}$. The Minkowski difference formula for $\wt L(\lambda)$ (see \eqref{E2.3}) and $M(\lambda)\twoheadrightarrow V\twoheadrightarrow L(\lambda)$ together imply
		\begin{equation}\label{E3.5}
		\lambda-\mathbb{Z}_{\geq0}(\Delta^+\setminus\Delta_{J_{\lambda}}^+)\subseteq\wt L(\lambda)\subseteq\wt V.
		\end{equation}
		Therefore, the result is true when $\mathcal{J}_V= J_{\lambda}$, or when $\beta\in \Delta^+\setminus\Delta^+_{J_{\lambda}}$.
		
		So, we assume for the rest of the proof that $\mathcal{J}_V\subsetneqq J_{\lambda}$ and $\beta\in\Delta_{J_{\lambda}}^+\setminus\Delta^+_{\mathcal{J}_V}$. By~\cite[Lemma~3.4]{Teja}, pick $\alpha\in\Pi_{J_{\lambda}}$ such that $\beta-\alpha\in\Delta^+_{J_{\lambda}}\setminus\Delta_{\mathcal{J}_V}^+$. By the induction hypothesis, $\lambda-\beta+\alpha\in\wt V$. If $\langle\lambda-\beta+\alpha,\alpha^{\vee}\rangle>0$, then by Lemma \ref{L3.1} (with $\lambda-\beta+\alpha$, $\alpha$ in place of $\mu$, $\alpha_i$ in it, respectively),
		\[
		\lambda-\beta\in\big[s_{\alpha}(\lambda-\beta+\alpha),\lambda-\beta+\alpha\big]\subset\wt V,
		\]
		as required. Else, we have \[
		0\geq\langle\lambda-\beta+\alpha,\alpha^{\vee}\rangle\implies \langle\beta,\alpha^{\vee}\rangle\geq2+\langle\lambda,\alpha^{\vee}\rangle.
		\]
		Let $m:=\langle\beta,\alpha
		^{\vee}\rangle$. Note by the definition of $J_{\lambda}$ that $\langle\lambda,\alpha^{\vee}\rangle\geq 0$. This implies $m\geq2$, and so $\height(s_{\alpha}\beta)<\height(\beta)$. Now, suppose $\height_{J_{\lambda}\setminus\mathcal{J}_V}(s_{\alpha}\beta)>0$, or equivalently $s_{\alpha}\beta\in\Delta^+_{J_{\lambda}}\setminus\Delta^+_{\mathcal{J}_V}$. Then by the induction hypothesis, $\lambda-s_{\alpha}\beta\in\wt V$. By Lemma \ref{L3.1},
		\[
		\langle\lambda-s_{\alpha}\beta,\alpha^{\vee}\rangle\geq \langle\beta,\alpha^{\vee}\rangle >0\implies \lambda-\beta\in\big[s_{\alpha}(\lambda-s_{\alpha}\beta), \lambda-s_{\alpha}\beta\big]\subset\wt V.
		\]
		So, we assume for the rest of the proof that $\height_{J_{\lambda}\setminus\mathcal{J}_V}(s_{\alpha}\beta)=0$, or equivalently $s_{\alpha}\beta\in\Delta^+_{\mathcal{J}_V}$. Notice, this assumption forces $\alpha\in \Pi_{J_{\lambda}\setminus\mathcal{J}_V}$, as $\height_{J_{\lambda}\setminus\mathcal{J}_V}(\beta)>0$. Observe that $s_{\alpha}\beta+\alpha$ is a root in $\Delta_{J_{\lambda}}^+\setminus\Delta^+_{\mathcal{J}_V}$. Moreover, $\height(\alpha+s_{\alpha}\beta)<\height(\beta)$, as $\beta=s_{\alpha}\beta+m\alpha$ and $m\geq 2$. So, by the induction hypothesis, $\lambda-s_{\alpha}\beta-\alpha\in\wt V$. Now, $\langle\lambda,\alpha^{\vee}\rangle\geq 0$ leads to two cases below.\medskip\\
		(1) $\langle\lambda,\alpha^{\vee}\rangle\geq 1$: In this case, by Lemma \ref{L3.1},
		\begin{align*}
		\langle\lambda-s_{\alpha}\beta-\alpha,\alpha^{\vee}\rangle=\langle\lambda,\alpha^{\vee}\rangle+&m-2\geq m-1>0,\quad\text{implies}\\ \lambda-\beta=\lambda-s_{\alpha}\beta-\alpha-(m-1)\alpha\in& \big[s_{\alpha}(\lambda-s_{\alpha}\beta-\alpha), \lambda-s_{\alpha}\beta-\alpha\big]\subset\wt V.
		\end{align*}
		(2) $\langle\lambda,\alpha^{\vee}\rangle=0$: $\alpha\in\Pi_{J_{\lambda}\setminus \mathcal{J}_V}$ implies $\alpha\in \Pi_{\mathcal{I}_V}$, and so $\lambda-\alpha\in \wt V$ by the definition of $\mathcal{I}_V$. Recall, $V_{\lambda-\alpha}$ is spanned by the (non-zero) maximal vector $f_{\alpha}v_{\lambda}$, where $f_{\alpha}$ and $v_{\lambda}$ span $\mathfrak{g}_{-\alpha}$ and $V_{\lambda}$, respectively. Observe that $U(\mathfrak{g})(f_{\alpha}v_{\lambda})$ is a highest weight $\mathfrak{g}$-module with highest weight $\lambda-\alpha$. So, $\wt L(\lambda-\alpha)\subset \wt \big(U(\mathfrak{g})(f_{\alpha}v_{\lambda})\big)\subset \wt V$. As $\supp(\alpha)\not\subset I_{L(\lambda-\alpha)}$, observe by the Minkowski difference formula for $\wt L(\lambda-\alpha)$ \big(see \eqref{E2.3}\big) that
		\[\lambda-\alpha-\gamma\in \wt L(\lambda-\alpha)\qquad\text{for all }\text{ }\gamma\in\Delta^+\text{ with }\height_{\supp(\alpha)}(\gamma)>0.\]
		In particular, therefore $(\lambda-\alpha)-\big(s_{\alpha}\beta+(m-1)\alpha\big)=\lambda-\beta\in\wt L(\lambda-\alpha)\subset\wt V$, since \[s_{\alpha}\beta+(m-1)\alpha\in\Delta^+\quad\text{ and }\quad\height_{\supp(\alpha)}\big(s_{\alpha}\beta+(m-1)\alpha\big)=m-1>0.
		\]
		The last assertion that $\lambda-\Delta_{\mathcal{I}_Y\setminus \mathcal{J}_Y,1}\subset \wt V$ for $\emptyset\neq Y\subseteq \wt V$ follows by the above proof, and the definitions of $\mathcal{I}_Y$ and $\mathcal{J}_Y$. Hence, the proof of the proposition is complete. 
	\end{proof}	
 The following corollary is an immediate application of Proposition \ref{L3.6} and Lemma \ref{L3.5}. It extends the parabolic-PSP for $\wt V$---see Subsection 5.2 of \cite{Teja}---to every 212-closed subset $Y$ of $\wt V$ which contains $\lambda$, in the special case $I=\mathcal{I}_Y\setminus\mathcal{J}_Y$. Recall the chain notation in \eqref{E3.3}. 
	\begin{cor}\label{C3.7}
		Let $\mathfrak{g}$ and $V$ be as in Proposition \ref{L3.6}. Let $Y$ be 212-closed in $\wt V$ and $\lambda\in Y$.
		\begin{itemize}
			\item[(a)] Given $\mu\precneqq\lambda\in Y$, there exist weights $\mu_j\in\wt V$, $0\leq j\leq n=\height_{\mathcal{I}_Y\setminus\mathcal{J}_Y}(\lambda-\mu)$, such that
			\[
			\mu=\mu_0\prec\cdots\prec\mu_j\prec\cdots\prec\mu_n=\lambda\in Y,\quad
			\mu_j-\mu_{j-1}\in\Delta_{\mathcal{I}_Y\setminus \mathcal{J}_Y,1},\quad \lambda-(\mu_j-\mu_{j-1})\in Y\text{ }\forall\text{ }1\leq j\leq n.\]
			\item[(b)] Given $i\in \mathcal{I}_Y$, there exists a root $\gamma_i\in\Delta_{\mathcal{I}_Y\setminus\mathcal{J}_Y,1}$ such that $\height_{\{i\}}(\gamma_i)>0$ and $\lambda-\gamma_i\in Y$. In the previous line, if $i\in \mathcal{I}_Y\setminus\mathcal{J}_Y$, then $\height_{\{i\}}(\gamma_i)=1$.
		\end{itemize}
	\end{cor}
	\begin{proof}[\textnormal{\textbf{Proof}}]
		Let $\mu$ be as in part (a). We prove part (a) by induction on $\height(\lambda-\mu)\geq 1$.\\
		Base step: $\height(\lambda-\mu)=1$ implies $\lambda-\mu\in\Pi_{\mathcal{I}_Y\setminus\mathcal{J}_Y}$, and so the result is trivial.\\
		Induction step: Assume that $\height(\lambda-\mu)>1$. Pick a root $\gamma\in\Delta_{\mathcal{I}_Y\setminus\mathcal{J}_Y,1}$ such that $\mu+\gamma\in \wt V$. Such a root $\gamma$ exists: (i) by the existence of a non-zero weight vector of the form in the statement of Lemma \ref{L3.5}, when $\height_{\mathcal{I}_Y\setminus\mathcal{J}_Y}(\lambda-\mu)>1$; or (ii) as $\lambda-\mu\in\Delta_{\mathcal{I}_Y\setminus\mathcal{J}_Y,1}$ by Lemma \ref{L3.5} when $\height_{\mathcal{I}_Y\setminus\mathcal{J}_Y}(\lambda-\mu)=1$. By Proposition \ref{L3.6}, $\lambda-\gamma\in\wt V$. By the 212-closedness of $Y$, the following equation implies both $\mu+\gamma,\lambda-\gamma\in Y$. 	\begin{equation}\label{E3.6}
		(\mu)+(\lambda)=(\mu+\gamma)+(\lambda-\gamma),
		\end{equation}
		Finally, the induction hypothesis applied to $\mu+\gamma$ completes the proof of part (a).
		
		The proof of part (b) follows upon observing that for each $i\in\mathcal{I}_Y$, there exists a weight $\mu_i\in Y$ such that $i\in\supp(\lambda-\mu_i)$ (by the definition of $\mathcal{I}_Y$), and then applying part (a) to $\mu_i$.
	\end{proof}
	The rest of this section is devoted to understanding the 212-closed subsets of $\wt V$ by looking at their maximal elements. We say that a weight $\mu\in Y$ is a \textit{maximal element} of $Y$ if there does not exist a weight $\mu'\in Y$ such that $\mu\precneqq \mu'$. Our study in the rest of this section is divided into two subsections for the sake of avoiding any confusion. In Subsection \ref{S4.3}, we deal with the situation when $V$ is arbitrary and $\lambda\in Y$ (in which case, $\lambda$ is the unique maximal element in $Y$), and in Subsection \ref{S4.4} when $\wt V=\wt M(\lambda,J)$ for some $J\subseteq J_{\lambda}$ and $Y$ is arbitrary.
	\subsection{On $P(Y)$ for $Y$ 212-closed in $\wt V$, and $Y$ equalling a standard face of $\wt V$}\label{S4.3}
	Throughout this subsection, we assume unless otherwise stated that $M(\lambda)\twoheadrightarrow{ }V$ is an arbitrary highest weight $\mathfrak{g}$-module, and $Y$ is an arbitrary 212-closed subset of $\wt V$ and $Y$ properly contains $\lambda$. 
	
	Our motivation to study the sets $Y$ that contain the highest weight comes from Theorem 3.4 in~\cite{Khare_JA}. In fact, in Steps 1--4 of the proof of Theorem \ref{thmA}, we show for any highest weight module $V$ and any $Y$ 212-closed in $\wt V$ that $Y$ contains a vertex---i.e., there exists an element $w\in W_{I_V}$ such that $w\lambda\in Y$. For instance, this result has been shown by Chari for integrable highest weight modules over semisimple $\mathfrak{g}$, see \cite[Lemma 6.8]{Khare_JA}. The assumption $\lambda\in Y$ has interesting consequences, and this subsection presents the same. Broadly, the two major achievements in this subsection are: (1)~we introduce and motivate an object called $P(Y)$ for $Y\subseteq \wt V$---see equation \eqref{EP(Y)} below for its definition---which is the key object for our line of attack in the proof of Theorem~\ref{thmA}; (2)~we give various necessary and sufficient conditions for $Y$ which contains $\lambda$ to equal a standard face of $\wt V$.
	
	To begin with, observe for any standard face $Z\subseteq\wt V$ that $\lambda-\Pi_{\mathcal{I}_Z\setminus\mathcal{J}_Z}\subset Z$. Given $Y$, we show more generally that this phenomenon holds true inside a suitable $W_{\mathcal{J}_Y}$-translate of $Y$:
	\begin{lemma}\label{L3.9}
		Let $\mathfrak{g}$ be a Kac--Moody algebra, $\lambda\in\mathfrak{h}^*$, and $M(\lambda)\twoheadrightarrow V$. Suppose $Y$ is 212-closed in $\wt V$ and $Y$ properly contains $\lambda$. Then there exists $w\in W_{\mathcal{J}_Y}$ such that $\lambda-\Pi_{\mathcal{I}_Y\setminus \mathcal{J}_Y}\subset wY$.
	\end{lemma}
	Before we give the proof of this lemma, we first show it for an example below. This simple example motivated this lemma and several other results in this paper.
	\begin{example}\label{E3.8}
		Let $\mathfrak{g}=
		\mathfrak{sl}_{3}(\mathbb{C})$, and $\mathcal{I}=\{1,2\}$. Let $V=L(\omega_1)$, where $\omega_1$ is the fundamental dominant weight corresponding to $1\in\mathcal{I}$; i.e., $\omega_1$ satisfies $\langle\omega_1,\alpha_{j}^{\vee}\rangle=\delta_{1,j}$ $\forall$ $j\in\{
		1,2\}$. Check that $\wt V=\{\omega_1, \omega_1-\alpha_1, \omega_1-\alpha_1-\alpha_2\}$. Let $Y=\{\omega_1,\omega_1-\alpha_1-\alpha_2\}$. Observe: 
		\begin{align*}
		\begin{aligned}[t]
		(i)&\text{ } Y\text{ is }212\text{-closed in }\wt V,\\ (ii)&\text{ }1\in \mathcal{I}_Y\setminus\mathcal{J}_Y,\text{ but }\omega_1-\alpha_1\notin Y,\end{aligned}\qquad\qquad
		\begin{aligned}[t]
		(iii)&\text{ }2\in \mathcal{J}_Y\text{ and  }s_2Y=\{\omega_1,\omega_1-\alpha_1\},\\
		(iv)&\text{ }\omega_1-\Pi_{\mathcal{I}_Y\setminus \mathcal{J}_Y}\subset s_2Y.
		\end{aligned}
		\end{align*}
	\end{example}
	\begin{proof}[\textnormal{\textbf{Proof of Lemma \ref{L3.9}}}]
		Let $Y$ be as in the statement. For each $t\in\mathcal{I}_Y\setminus\mathcal{J}_Y$, we fix $\mu_t\in Y$ with least $\height(\lambda-\mu_t)$ such that $\lambda-\mu_t\in\Delta_{\mathcal{I}_Y\setminus\mathcal{J}_Y,1}$ and $\supp_{\mathcal{I}_Y\setminus\mathcal{J}_Y}(\lambda-\mu_t)=\{t\}$; see \eqref{I-Notation} for the notation. Such weights $\mu_t$ exist by Corollary \ref{C3.7}(b). (We assumed $\lambda\in Y$ to apply this corollary.)  \[\text{We define}\quad d(Y):=\left(\sum\limits_{t\in\mathcal{I}_Y\setminus\mathcal{J}_Y}\height(\lambda-\mu_t)\right)-\big|\mathcal{I}_Y\setminus\mathcal{J}_Y\big|.\]
		Given $Y$, note that the choice of $\mu_t$ may be not be unique, yet $d(Y)$ is well defined. We prove the lemma by induction on $d(Y)\geq 0$. In the base step $d(Y)=0$, the lemma is trivial.\\
		Induction step: Assume that $d(Y)>0$. So, there exists a node $i\in \mathcal{I}_Y\setminus\mathcal{J}_Y$ such that the corresponding weight $\mu_i$ satisfies $\height(\lambda-\mu_i)>1$; or equivalently, $\height_{\mathcal{J}_Y}(\lambda-\mu_i)>0$. Fix such an $i$ and $\mu_i$. By \cite[Lemma 3.4]{Teja} applied to $\lambda-\mu_i\in\Delta_{\mathcal{I}_Y\setminus\mathcal{J}_Y,1}$, pick $j\in\mathcal{J}_Y$ such that $\lambda-\mu_i-\alpha_j\in\Delta_{\mathcal{I}_Y\setminus\mathcal{J}_Y,1}$. By Proposition \ref{L3.6}, $\mu_i+\alpha_j\in\wt V$. By the minimality of $\height(\lambda-\mu_i)$ (corresponding to $i$), observe that $\mu_i+\alpha_j\notin Y$. Now, suppose $\langle\mu',\alpha_j^{\vee}\rangle>0$ for some $\mu'\in Y$. Then $\mu'-\alpha_j\in [\mu'-\langle\mu',\alpha_j^{\vee}\rangle\alpha_j,\mu']\subset\wt V$ by Lemma \ref{L3.1}. Further, equation (\ref{E''6.4}) (with $\mu_i$ and $\alpha_j$ replacing $\mu$ and $\alpha$ in it, respectively) implies $\mu_i+\alpha_j\in Y$, which contradicts $\mu_i+\alpha_j\notin Y$. Similarly, if $\langle\mu_i,\alpha_j^{\vee}\rangle\geq0$, then $\langle\mu_i+\alpha_j,\alpha_j^{\vee}\rangle\geq2$ implies $\mu_i-\alpha_j\in\wt V$ by Lemma \ref{L3.1}. Further, by the 212-closedness of $Y$, the equation
		\begin{equation}\label{E'4.7} 2(\mu_i)=(\mu_i+\alpha_j)+(\mu_i-\alpha_j),\end{equation}
		implies $\mu_i+\alpha_j\in Y$, which contradicts $\mu_i+\alpha_j\notin Y$. Therefore, $\langle y,\alpha_j^{\vee}\rangle\leq0$ $\forall$ $y\in Y$ and $\langle\mu_i,\alpha_j^{\vee}\rangle<0$. Recall by Remark \ref{R3.3} that $\wt V$ is invariant under $s_j$, as $j\in \mathcal{J}_Y$. In particular, $s_jY\subset \wt V$. Now, observe: (i)~$s_jY$ is 212-closed (see Observation \ref{note1}); (ii)~$\lambda\in s_jY$ as $s_j\lambda=\lambda$; (iii)~$\mathcal{I}_Y\setminus\mathcal{J}_Y=\mathcal{I}_{s_jY}\setminus\mathcal{J}_{s_jY}$; (iv)~$s_j\mu_t\in s_jY$ and $\lambda-s_j\mu_t\in\Delta_{\mathcal{I}_Y\setminus\mathcal{J}_Y,1}$; (v)~$\height(\lambda-s_j\mu_t)\leq \height(\lambda-\mu_t)$ $\forall$~$t\in\mathcal{I}_Y\setminus\mathcal{J}_Y$, and $\height(\lambda-s_j\mu_i)<\height(\lambda-\mu_i)$. In view of these, for each $t\in\mathcal{I}_Y\setminus\mathcal{J}_Y$, if we fix $\mu'_t\in s_jY$ (similar to $\mu_t\in Y$) with least $\height(\lambda-\mu'_t)$ such that $\lambda-\mu'_t\in\Delta_{\mathcal{I}_Y\setminus\mathcal{J}_Y,1}$ and $\supp_{\mathcal{I}_Y\setminus\mathcal{J}_Y}(\lambda-\mu'_t)=\{t\}$. Then $\height(\lambda-\mu'_t)\leq \height(\lambda-\mu_t)$ $\forall$ $t\in\mathcal{I}_Y\setminus\mathcal{J}_Y$, and $\height(\lambda-\mu'_i)< \height(\lambda-\mu_i)$. This implies $d(s_jY)<d(Y)$. So, the induction hypothesis applied to $s_jY$ completes the proof.
	\end{proof}	  
	Lemma \ref{L3.9} enables us to obtain some necessary and sufficient conditions for $Y$ to be a standard face of $\wt V$, see Proposition \ref{P3.11} below. In order to state and prove this result, we first define
	\begin{align}\label{EP(Y)}
	\begin{aligned} P(Y):=\Bigg\{\lambda-\eta\text{ }\in Y\quad\Bigg|
	\end{aligned}\quad
	\begin{aligned}
	&\eta\in\Delta_{\mathcal{I}_Y\setminus\mathcal{J}_Y,1},\text{ }\eta\text{ is a sum of distinct simple roots, and}\\
	&\lambda-\eta'\in Y\text{ }\text{whenever}\text{ }\eta'\in\Delta_{\mathcal{I}_Y\setminus\mathcal{J}_Y,1}\text{ and }\eta'\preceq \eta
	\end{aligned}
	\begin{aligned}\Bigg\}.
	\end{aligned}
	\end{align}
	$P(Y)$ reveals a lot of information about the 212-closed subset $Y$, and this will be evident for instance from Lemma \ref{L3.10} and Proposition \ref{P3.11} below. 
	As mentioned before, working with $P(Y)$ is the key tool in the proof of Theorem \ref{thmA} in the next section. Therefore, a good understanding of $P(Y)$ is necessary, and the next few lines try to provide the same.
	To begin with, note that $P(Y)$ can possibly be empty. Observe that $P(Y)\subseteq P\left(\wt_{\mathcal{I}_Y}V\right)$. For simplicity, we describe $P(\wt_{\mathcal{I}_Y}V)$, and this should suffice to gain some familiarity about $P(Y)$. Recall from Remark \ref{R3.4} that every connected component of $\mathcal{I}_Y$ contains a node from $\mathcal{I}_Y\setminus \mathcal{J}_Y$. Let $K$ be a non-empty subset of $\mathcal{I}$. Recall, whenever the Dynkin subdiagram on $K$ is connected, the sum $\sum_{t\in K}\alpha_t$ is a root. Next, whenever $K\cap \big(\mathcal{I}_Y\setminus\mathcal{J}_Y\big)\neq\emptyset$ and the Dynkin subdiagram on $K$ is connected, Proposition \ref{L3.6} implies $\lambda-\sum_{t\in K}\alpha_t\in\wt V$. Now, observe that 
	\begin{align*}
	\begin{aligned} P\big(\wt_{\mathcal{I}_Y}V\big)=\Bigg\{\lambda-\sum\limits_{i\in I}\alpha_i\text{ }\Bigg|
	\end{aligned}\text{ }
	\begin{aligned}
	&I\subseteq \mathcal{I}_Y,\text{ }|I\cap(\mathcal{I}_Y\setminus\mathcal{J}_Y)|=1,\text{ and the}\\
	&\text{Dynkin subdiagram on }
	I\text{ is connected}
	\end{aligned}
	\begin{aligned}\Bigg\},
	\end{aligned}
	\text{ }
	\begin{aligned}
	\text{ and }\lambda-\Pi_{\mathcal{I}_Y\setminus\mathcal{J}_Y}\subseteq P(\wt_{\mathcal{I}_Y}V).
	\end{aligned}
	\end{align*}
	Throughout the paper, in particular in the definition of $P(\wt_{\mathcal{I}_Y}V)$ above, we treat a Dynkin diagram with single vertex to be connected. Now $\big|P(\wt_{\mathcal{I}_Y}V)\big|\leq \big|\mathcal{I}_Y\setminus\mathcal{J}_Y\big|2^{|\mathcal{J}_Y|}$, so $P(\wt_{\mathcal{I}_Y}V)$, whence $P(Y)$, is finite. The following lemma shows an important consequence of $P(Y)\neq\emptyset$.
	\begin{lemma}\label{L3.10}
		Let $\mathfrak{g}$, $V$ and $Y$ be as in Lemma \ref{L3.9}. Fix $\emptyset\neq K\subseteq \mathcal{I}$, and suppose $\lambda-\sum\limits_{t\in K}\alpha_t\in P(Y)$. Then $\wt_{K}V\subseteq Y$. More generally, $\wt_{\mathcal{I}_{P(Y)}}V\subseteq Y$.
	\end{lemma}
	\begin{proof}[\textnormal{\textbf{Proof}}]
		Let $K$ be as in the statement. By the definition of $P(Y)$, there exists a unique $i\in K$ such that $i\in\mathcal{I}_Y\setminus\mathcal{J}_Y$. By part (D1) of Theorem D in \cite{Teja}, there exists a sequence of roots $\beta_r$, $1\leq r\leq |K|$, such that \big(see \eqref{E3.3} for the chain notation\big)
		\[ \alpha_i=\beta_1\prec\cdots\prec\beta_r\prec\cdots\prec\beta_{|K|}=\sum\limits_{t\in K}\alpha_t\text{ }\in \Delta_{\mathcal{I}_Y\setminus\mathcal{J}_Y,1},\text{ }\text{ and when }|K|>1,\text{ } \beta_r-\beta_{r-1}\in\Pi_{\mathcal{J}_Y}\text{ } \forall\text{ }r>1.\]
		Now, Proposition \ref{L3.6} and the definition of $P(Y)$ together yield
		\begin{equation}\label{E''3.6} \lambda-\beta_{|K|}\prec\cdots\prec\lambda-\beta_r\prec\cdots\prec\lambda-\beta_1\text{ }\in Y.\end{equation}
		Recall, $\lambda \in Y$. With these observations, we now show via induction on $\height(\lambda-\mu)\geq1$ that $\mu\in\wt_K V$ implies $\mu\in Y$, which proves the first part of the lemma. The proof for $\wt_{\mathcal{I}_{P(Y)}}V\subseteq Y$ is similar.\\
		Base step: $\height(\lambda-\mu)=1$ forces $\mu=\lambda-\alpha_i=\lambda-\beta_1$ \big(see the chain in \eqref{E''3.6}\big), and so $\mu\in Y$.\\
		Induction step: Assume that $\height(\lambda-\mu)>1$. Fix $j\in K$ such that $\mu+\alpha_j\in\wt_KV$. Set $\beta_0=0$ for convenience. Observe that there exists $1\leq s\leq |K|$ such that $\beta_s-\beta_{s-1}=\alpha_j$. By the induction hypothesis, $\mu+\alpha_j\in Y$. Now, the 212-closedness of $Y$, the equation
		\[ (\mu+\alpha_j)+(\lambda-\beta_{s})=(\mu)+(\lambda-\beta_{s-1}),\]
		and also the chain in \eqref{E''3.6} together imply that $\mu\in Y$, completing the proof of the lemma.
	\end{proof}
	Along the same lines of Lemma \ref{L3.10}, we prove another lemma needed in the proof of Theorem~\ref{thmA}.\allowdisplaybreaks
	\begin{lemma}\label{L4.11}
		Let $\mathfrak{g}$, $V$ and $Y$ be as in Lemma \ref{L3.9}. Let $\beta\in\Delta_{\mathcal{I}_Y\setminus\mathcal{J}_Y,1}$ be a sum of distinct simple roots. Assume that $\lambda-\beta\in P(Y)$. Suppose $j\in\mathcal{J}_Y$ is such that $j\notin \supp(\beta)$ and $\beta+\alpha_j$ is a root \big(which happens if and only if the Dynkin subdiagram on $\supp(\beta)\sqcup\{j\}$ is connected\big). Then: \[\lambda-\beta-\alpha_j\in Y\quad \implies \quad \lambda-\beta-\alpha_j\in P(Y).
		\]
	\end{lemma}
	\begin{proof}[\textnormal{\textbf{Proof}}]
		Let $Y$, $\beta$ and $j$ be as in the statement. Assume that $\lambda-\beta-\alpha_j\in Y$. By the assumptions on $\beta$ and $j$, observe that $\beta+\alpha_j$ satisfies the first two conditions in the definition of $P(Y)$; namely $\beta+\alpha_j\in\Delta_{\mathcal{I}_Y\setminus\mathcal{J}_Y,1}$ and $\beta+\alpha_j$ is a sum of distinct simple roots. So to complete the proof of the lemma, we show:
		\[ \gamma\in \Delta_{\mathcal{I}_Y\setminus\mathcal{J}_Y,1}\text{ }\text{ and }\text{ }\gamma\precneqq \beta+\alpha_j\quad\implies\quad \lambda-\gamma\in Y,\qquad\text{via induction on }\height(\gamma)\geq 1.\]
		Base step: $\height(\gamma)=1$ implies $\gamma\in \Pi_{\mathcal{I}_Y\setminus\mathcal{J}_Y}$. This implies $\gamma\prec \beta$ since $\gamma\prec \beta+\alpha_j$. Now as $\lambda-\beta\in P(Y)$, the definition of $P(Y)$ implies $\lambda-\gamma\in Y$.\\
		Induction step: Assume that $\height(\gamma)>1$. If $\height_{\{j\}}(\gamma)=0$, then $\gamma\preceq \beta$, and we will be done as in the base step. So, we assume for the rest of the proof that $\height_{\{j\}}(\gamma)>0$. As $\lambda-\beta\in P(Y)$, by the definition of $P(Y)$ and part (D1) of Theorem D in \cite{Teja}, there exists a chain of roots $\beta_r\in \Delta_{\mathcal{I}_Y\setminus\mathcal{J}_Y,1}$, $1\leq r\leq n=\height(\beta)+1$ such that
		\[
		\beta_1\in\Pi_{\mathcal{I}_Y\setminus\mathcal{J}_Y},\text{ }
		\beta_{n-1}=\beta\text{ and }\beta_n=\beta+\alpha_j,\quad\lambda-\beta_r\in Y\text{ }\forall  r,\quad	    \beta_{r-1}\prec\beta_{r}\text{ and }\beta_r-\beta_{r-1}\in \Pi_{\mathcal{J}_Y}\text{ }\forall\text{ } r>1.
		\]
		By \cite[Lemma 3.4]{Teja}, we fix a node $i\in\mathcal{J}_Y$ such that $\gamma-\alpha_i\in\Delta_{\mathcal{I}_Y\setminus\mathcal{J}_Y,1}$. By the induction hypothesis, $\lambda-\gamma+\alpha_i\in Y$. Note, $i\in\supp(\beta+\alpha_j)=\supp(\beta)\sqcup\{j\}$. So, there exists $s\in [n]$ such that $\beta_s-\beta_{s-1}=\alpha_i$. By Proposition \ref{L3.6}, $\lambda-\gamma\in \wt V$. Now, the equation
		\[
		(\lambda-\gamma+\alpha_i)+(\lambda-\beta_s)=(\lambda-\gamma)+(\lambda-\beta_{s-1})\]
		by the 212-closedness of $Y$ implies $\lambda-\gamma\in Y$. Hence, the proof of the lemma is complete.\end{proof}
	We will conclude this subsection by giving some necessary and sufficient conditions for a 212-closed subset $Y$ to equal a standard face of $\wt V$, in the following proposition. The implication (3) $\implies$ (1) in it will be useful in the proof of Theorem \ref{thmA} in the next section. 
	\begin{prop}\label{P3.11}
		Let $\mathfrak{g}$, $V$ and $Y$ be as in Lemma \ref{L3.9}. Then the following are equivalent:
		\begin{enumerate} 
			
			\item[(1)] $Y=\wt_{\mathcal{I}_Y}V$.
			\item[(2)] $\lambda-\left(\Delta_{\mathcal{I}_Y\setminus \mathcal{J}_Y,1}\cap\Delta_{\mathcal{I}_Y}\right)\subset Y$.
			\item[(3)] $P(Y)=P(\wt_{\mathcal{I}_Y}V).$ 
			\item[(4)] $Y$ is $W_{\mathcal{J}_Y}$-invariant.
		\end{enumerate} 
	\end{prop}
	\begin{proof}[\textnormal{\textbf{Proof}}]
		Let $Y$ be as in the statement. $(1)\implies (2)$ follows by the definitions of $\mathcal{I}_Y$ and $\wt_{\mathcal{I}_Y}V$ \big(see \eqref{E1.2} and \eqref{defn IV-JY}\big), since $\lambda-\Delta_{\mathcal{I}_Y\setminus \mathcal{J}_Y,1}\subset\wt V$ by Proposition \ref{L3.6}. Next, $(2)\implies(3)$ is obvious as $P(Y)\subseteq P(\wt_{\mathcal{I}_Y}V)\subseteq\lambda-\big(\Delta_{\mathcal{I}_Y\setminus\mathcal{J}_Y,1}\cap\Delta_{\mathcal{I}_Y}\big)$. (3) $\implies$ (1) follows by Lemma \ref{L3.10} via the observation $\mathcal{I}_{P(Y)}=\mathcal{I}_Y$ when $P(Y)=P(\wt_{\mathcal{I}_Y}V)$. Recall, $(1)\implies (4)$ was shown in Remark~\ref{R3.3}.
		
		Finally, we show the non-trivial implication $(4)\implies(1)$. Assume that $Y$ is $W_{\mathcal{J}_Y}$-invariant. By Lemma \ref{L3.9}, $\lambda-\Pi_{\mathcal{I}_Y\setminus\mathcal{J}_Y}\subset Y$. We prove by induction on $\height(\lambda-\mu)\geq 1$ that $\mu \in \wt_{\mathcal{I}_Y}V$ implies $\mu \in Y$. In the base step $\height(\lambda-\mu)=1$, the result follows by $\lambda-\Pi_{\mathcal{I}_Y\setminus\mathcal{J}_Y}\subset Y$.\\
		Induction step:  Assume that $\height(\lambda-\mu)>1$. Fix $i\in \mathcal{I}_Y$ such that $\mu+\alpha_{i}\in \wt_{\mathcal{I}_Y}V$. Recall, $\lambda-\alpha_i\in Y$. By the induction hypothesis, $\mu+\alpha_{i}\in Y$. When $i\in \mathcal{I}_Y\setminus \mathcal{J}_Y$, the equation
		\[
		(\mu+\alpha_i)+(\lambda-\alpha_i)=(\mu)+(\lambda)
		\]
		by the 212-closedness of $Y$ implies $\mu\in Y$. So, we assume for the rest of the proof that $i\in \mathcal{J}_Y$. Note by the $W_{\mathcal{J}_Y}$-invariance of $Y$ that $s_{i}(\mu+\alpha_{i})\in Y$. If $\langle\mu+\alpha_{i},\alpha_{i}^{\vee}\rangle=1$, then $s_{i}(\mu+\alpha_{i})=\mu$, and we are done. Else if $\langle\mu+\alpha_{i},\alpha_{i}^{\vee}\rangle>1$, then Remark \ref{R3.1} applied to $s_{i}(\mu+\alpha_{i})\precneqq \mu+\alpha_{i}\in Y$ (in place of $\mu'\precneqq \mu$ in it) yields $\mu\in Y$. So, we assume now that $\langle\mu+\alpha_{i},\alpha_{i}^{\vee}\rangle\leq 0$. This implies $\langle\mu,\alpha_{i}^{\vee}\rangle\leq-2$, and so $s_{i}\mu\succneqq\mu$. As $s_i\mu\in\wt V$ \big(by the $W_{\mathcal{J}_Y}$-invariance of $\wt V$\big), the induction hypothesis implies $s_{i}\mu \in Y$. Finally, $\mu\in Y$ by the $W_{\mathcal{J}_Y}$-invariance of $Y$, completing the proof of the lemma.
	\end{proof}
	\subsection{On maximal elements of 212-closed subsets of $\wt M(\lambda,J)$}\label{S4.4}
	Throughout this subsection, unless otherwise mentioned, we assume $\wt V=\wt M(\lambda,J)$ for some $J\subseteq J_{\lambda}$, and $Y\subseteq \wt V$ is an arbitrary 212-closed subset. The goal of this subsection is to study the \textit{maximal elements} of $Y$; see Proposition \ref{P3.13}. Recall, $\mu\in Y$ is a maximal element of $Y$ if there does not exist a weight $y\in Y$ such that $\mu\precneqq y$. By our analysis, we now show that the maximal elements of any 212-closed subset of $\wt M(\lambda,J)$ are contained in the set of vertices $W_J\lambda$ of $\wt M(\lambda,J)$ \big(or of $\conv_{\mathbb{R}}(\wt M(\lambda,J))$\big). For this, we first note the following observation, which is somewhat similar to Corollary \ref{C3.7}(a), but much stronger than it in view of the nice Minkowski difference formula \eqref{E2.1} for $\wt M(\lambda,J)$. 
	\begin{observation}\label{O3.12}
		Let $Y\subseteq \wt M(\lambda,J)$ be 212-closed. For some $\gamma_1,\ldots,\gamma_n\in\Delta_{J^c,1}$, suppose $\lambda-\sum_{i=1}^n\gamma_i\in Y$. Then $\lambda-\mathbb{Z}_{\geq0}\{\gamma_i\}_{i=1}^{n}\subseteq Y$. To verify this, consider the system of~equations:
		\begin{equation}\label{E3.9}
		2\left(\lambda-\sum\limits_{j=1}^i\gamma_j\right)=\left(\lambda-\sum\limits_{j=1}^{i-1}\gamma_j\right)+\left(\lambda-2\gamma_i-\sum\limits_{j=1}^{i-1}\gamma_j\right) \qquad
		\text{for each }i\in[n].
		\end{equation}
		When $i=1$, treat $\sum_{j=1}^{i-1}\gamma_j$ as 0. For each $i$, notice by the Minkowski difference formula for $\wt M(\lambda,J)$ that both terms on the right hand side of the corresponding $i^{th}$ equation in the system in \eqref{E3.9} are weights of $M(\lambda,J)$. By the 212-closedness of $Y$, the system of equations in \eqref{E3.9} results in $\lambda-\sum_{j=1}^{i}\gamma_j\in Y$ for all $i\in [n],\text{ and also }\lambda\in Y$. Now, observe for each $i\in [n]$,
		\[ 
		\left(\lambda-\sum_{j=1}^{i}\gamma_j\right)+(\lambda)=\left(\lambda-\sum_{j=1}^{i-1}\gamma_j\right)+(\lambda-\gamma_i)\quad\implies\quad \lambda-\gamma_i\in Y.
		\]
		From here, for any sequence $(c_i)_{i=1}^n\in \mathbb{Z}_{\geq 0}$, via an induction argument on $\sum_{i=1}^nc_i$, it can be easily shown that $\lambda-\sum_{i=1}^nc_i\gamma_i\in Y$, which implies the result. 
	\end{observation} 
	Next, we generalize the aforementioned result of Chari (see \cite[Lemma 6.8]{Khare_JA}), which says for $\lambda\in P^+$ in the semisimple case that every 212-closed subset of $\wt L(\lambda)$ intersects the orbit $W\lambda$. \begin{prop}\label{P3.13}
		Let $\mathfrak{g}$ be a Kac--Moody algebra, $\lambda\in\mathfrak{h}^*$, and $M(\lambda)\twoheadrightarrow V$ be such that $\wt V=\wt M(\lambda,J)$ for some $J\subseteq J_{\lambda}$.  	   	\begin{itemize}
			\item[(a)] Let $\mu\in\wt V$. Then $\mu\in W_J\lambda$ if and only if for every real root $\beta\in\Delta$, at most one of $\mu\pm\beta$ is a weight. In other words,
			\[\mu\in W_J\lambda \quad\iff\quad \{\mu-\beta,\text{ }\mu+\beta\}\not\subset \wt V\text{ }\text{ for all }\text{real roots }\beta.\]
			\item[(b)] Let $Y\subseteq\wt V$ be 212-closed. If $\mu$ is a maximal element of $Y$, then $\mu\in Y\cap\big( W_J\lambda\big)$. In particular, $Y\cap \big(W_J\lambda\big)\neq \emptyset$.
		\end{itemize}
	\end{prop} 
	Note by equation \eqref{E2.3} that this applies to $V=L(\lambda)$ for arbitrary $\lambda\in\mathfrak{h}^*$ and Kac--Moody $\mathfrak{g}$, and hence specializes to the aforementioned result from \cite{Khare_JA}.	\begin{proof}[\textnormal{\textbf{Proof}}]
		We first prove part (a), by the method of contradiction. Recall, $\wt V=\wt M(\lambda,J)$ is $W_J$-invariant. Fix an element $w\in W_J$ and a root $\beta$. Assume that both $w\lambda\pm\beta\in\wt V$. Without loss of generality, assume that $w^{-1}\beta$ is a positive root. By the $W_J$-invariance of $\wt V$, $w^{-1}\big(w\lambda+\beta\big)=\lambda+w^{-1}\beta\in \wt V$. This is absurd as $\lambda\precneqq\lambda +w^{-1}\beta$. Hence, $\{w\lambda-\beta,w\lambda+\beta\}\not\subset\wt V$ $\forall$ $w\in W_J$,~$\beta\in\Delta$. 
		
		For the reverse implication in part (a), assume that $\mu\in\wt V\text{ is such that }\{\mu-\beta,\mu+\beta\}\not\subset\wt V$ for any real root $\beta$. We prove that $\mu\in W_J\lambda$ by induction on $\height(\lambda-\mu)\geq 0$. In the base step, $\mu=\lambda$, and so the result is trivial. \\
		Induction step: Assume that $\height(\lambda-\mu)>0$. Fix $j\in\mathcal{I}$ such that $\mu+\alpha_j\in\wt V$. The condition satisfied by $\mu$ forces $\mu-\alpha_j\notin \wt V$. Observe by the Minkowski difference formula for $\wt V$ in \eqref{E2.1} that $j\in J$. If $\langle\mu+\alpha_j,\alpha_j^{\vee}\rangle\geq2$, then $\mu-\alpha_j\in[s_j(\mu+\alpha_j),\mu+\alpha_j]\subset\wt V$ by Lemma \ref{L3.1}, which cannot happen. So, $\langle\mu+\alpha_j,\alpha_j^{\vee}\rangle\leq1$, and this implies $\langle\mu,\alpha_j^{\vee}\rangle\leq-1$. Since $\mu-\alpha_j\notin\wt V$, $\mathfrak{g}_{-\alpha_j}V_{\mu}=\{0\}$; in other words, every non-zero vector in $V_{\mu}$ must be a minimal vector for the action of $\mathfrak{sl}_{\alpha_j}$. Recall that $e_t$, $\forall$ $t\in\mathcal{I}$ acts locally nilpotently on $V$. In particular, observe that $e_j$ acts nilpotently on $U(\mathfrak{g}_{\alpha_j})V_{\mu}$. So, $U(\mathfrak{g}_{\alpha_j})V_{\mu}$ is an integrable $\mathfrak{sl}_{\alpha_j}$-module. This implies $[\mu,s_j\mu]\subseteq\wt \big(U(\mathfrak{g}_{\alpha_j})V_{\mu}\big)\subseteq\wt V$. Now, observe: (i)~$s_j\mu\succneqq\mu$, and so $\height(\lambda-s_j\mu)<\height(\lambda-\mu)$; (ii)~$s_j\mu$ (similar to $\mu$) satisfies the condition: $\{s_j\mu-\eta,s_j\mu+\eta\}\not\subset\wt V$ for any real root $\eta$. Hence, the induction hypothesis applied to $s_j\mu$ finishes the proof of the reverse implication in part (a). 
		
		To show part (b), let $\mu$ be a maximal element of $Y$. For any root $\gamma$, if $\mu\pm\gamma\in\wt V$, then by the 212-closedness of $Y$, equation \eqref{E'4.7} (with $\mu$, $\gamma$ replacing $\mu_i$, $\alpha_j$, respectively) forces $\mu+\gamma\in Y$. But this contradicts the maximality of $\mu$. Therefore, $\{\mu-\beta,\mu+\beta\}\not\subset\wt V$ for any root $\beta$ (in particular for real $\beta$). So, by part (a), $\mu\in W_J\lambda$. Hence, the maximal elements of $Y$ lie in $Y\cap\big(W_J\lambda\big)$.  	\end{proof}
	\begin{cor}\label{C3.14}
		\begin{itemize}
			\item[(1)] If $\emptyset\neq Y$ is 212-closed in $\wt M(\lambda,J)$, then $\exists$ $\omega\in W_J$ such that $\lambda\in \omega Y$. 
			\item[(2)] Note, $\wt M(0,I^c)=-\mathbb{Z}_{\geq 0}\Delta_{I,1}$ $\forall$ $I\subseteq \mathcal{I}$. Thus, Proposition \ref{P3.13} applied to $\mathbb{Z}_{\geq 0}\Delta_{I,1}$ shows that given $0\neq \mu\in\mathbb{Z}_{\geq 0}\Delta_{I,1}$, there always exists a real root $\beta$ such that $\mu\pm\beta\in\mathbb{Z}_{\geq 0}\Delta_{I,1}$. 
		\end{itemize}
	\end{cor}
	The reverse implication in part (a) of the above proposition leads to the following question.
	\begin{question}\label{Q2}
		Let $\mathfrak{g}$ be a Kac--Moody algebra, $\lambda\in\mathfrak{h}^*$, and $M(\lambda)\twoheadrightarrow{ }V$. Consider the set $\big\{\mu\in \wt V\text{ }\big|\text{ } \{\mu-\beta,\mu+\beta\}\not\subset\wt V\text{ for any real root }\beta\big\}$. Is there a characterization for this set? 
	\end{question}
	This question might be interesting to explore in the pursuit of better understanding $\wt V$ for arbitrary $M(\lambda)\twoheadrightarrow{ }V$, and also given that this question has a nice answer when $\wt V=\wt M(\lambda,J)$.
	\section{Proof of Theorem \ref{thmA}: Weak faces of $\wt V$}
	Throughout this section, we assume that $\mathfrak{g}$ is a general Kac--Moody algebra, $\lambda\in\mathfrak{h}^*$, $M(\lambda)\twoheadrightarrow{ }V$ is arbitrary, and $Y$ is an arbitrary non-empty 212-closed subset of $\wt V$. The goal of this section is to prove Theorem \ref{thmA}, which at its core says that every 212-closed subset of $\wt V$ can be conjugated by elements of $W_{I_V}$ to some standard face of $\wt V$; which implies the equivalence of all the three notions in the statement of the theorem. Recall the notation $J_{\lambda},\mathcal{I}_V, \mathcal{J}_V, \mathcal{I}_Y \text{ and }\mathcal{J}_Y$ from \eqref{defn IV-JY}. We will use without further mention the fact: $j\in J_{\lambda}$ implies $\langle\mu,\alpha_j^{\vee}\rangle\in\mathbb{Z}$ $\forall$ $\mu\in\wt V$.
	
	Recall, Corollary \ref{C3.14} shows when $\wt V=\wt M(\lambda,J)$ that $Y$ can be conjugated by elements in $W_J$ to a 212-closed subset (of $\wt V$) which contains $\lambda$. In the following theorem, we generalize this phenomenon for arbitrary $\wt V$, which is crucial in the proof of Theorem \ref{thmA}.
	\begin{theorem}\label{T4.1}
		Let $\mathfrak{g}$ be a Kac--Moody algebra, $\lambda\in\mathfrak{h}^*$, and $M(\lambda)\twoheadrightarrow{ }V$. Suppose $\emptyset\neq Y\subseteq \wt V$ is 212-closed and $\lambda\notin Y$. Then there exists a sequence of simple reflections $s_1,\ldots, s_n\in W_{J_{\lambda}}$ such~that \[\left(\prod\limits_{t=r}^n s_t\right)Y\text{ } \text{ is }212\text{-closed in }\wt V\text{ for each }1\leq r\leq n,\quad\text{and}\quad \lambda\in\left(\prod\limits_{t=1}^ns_t\right)Y. 
		\]
	\end{theorem}
	Notice, there is no role of $W_{I_V}$ in the assertions of this theorem. However, as an application of this theorem, we give a further refinement of this theorem in terms of $W_{I_V}$, so as to serve our purpose of proving Theorem \ref{thmA}. More precisely, in Steps 1--4 of the proof of Theorem \ref{thmA} we show that there exists $w\in W_{I_V}$ such that $\lambda\in wY$. Observe that the assertions in Theorem \ref{T4.1} are much stronger than just $\lambda$ belonging to some $W_{J_{\lambda}}$-conjugate of $Y$.  
	\begin{proof}[\textnormal{\textbf{Proof of Theorem \ref{T4.1}}}]
		Let $Y$ be as in the statement. We prove the theorem by induction on \[d(Y):=\min\{\height(\lambda-y)\text{ }|\text{ }y\in Y\}\geq 1.\]
		Base step: Assume that $d(Y)=1$. Then for some $i\in\mathcal{I}$, $\lambda-\alpha_i\in Y$. Suppose either $\langle\lambda,\alpha_i^{\vee}\rangle\notin\mathbb{Z}_{\geq 0}$, or $V$ is not $\mathfrak{g}_{\{i\}}$-integrable. Then $\lambda-\mathbb{Z}_{\geq 0}\alpha_i\subseteq\wt V$. Further, the 212-closedness of $Y$ and 		\begin{equation}\label{E''5.1}
		2(\lambda-\alpha_i)=(\lambda)+(\lambda-2\alpha_i)\end{equation}
		together imply $\lambda\in Y$, which contradicts $\lambda\notin Y$. So, we assume for the rest of this step that $\langle\lambda,\alpha_i^{\vee}\rangle\in\mathbb{Z}_{\geq 0}$ and $V$ is $\mathfrak{g}_{\{i\}}$-integrable. This implies $s_i\wt V=\wt V$, and also $s_i Y$ is 212-closed in $\wt V$ (by Observation \ref{note1}). It only remains to show $\lambda\in s_i Y$. If $\langle\lambda,\alpha_i^{\vee}\rangle\geq 2$, then $\lambda-2\alpha_i\in[s_i\lambda, \lambda]\subset \wt V$ by Lemma \ref{L3.1}, and then equation \eqref{E''5.1} once again leads to $\lambda\in Y$ (a contradiction). Next, $\langle\lambda,\alpha_i^{\vee}\rangle=0$ implies $\lambda+\alpha_i=s_i(\lambda-\alpha_i)\in\wt V$ (as $s_i\wt V$ = $\wt V$), which is absurd as $\lambda+\alpha_i\succneqq \lambda$. Therefore $\langle\lambda,\alpha_i^{\vee}\rangle=1$, which implies $s_i\lambda=\lambda-\alpha_i$, and thereby $\lambda \in s_iY$, finishing the proof of base step. \\
		Induction step: Assume that $d(Y)>1$. Fix $\mu\in Y$ such that $\height(\lambda-\mu)=d(Y)$, and also $j\in \mathcal{I}$ such that $\mu+\alpha_j\in\wt V$. For any $\mu'\in Y$, if $\mu'-\alpha_j\in\wt V$, then equation \eqref{E''6.4} (with $\alpha_j$ in place of $\alpha$ in it) implies that $\mu+\alpha_j\in Y$, which contradicts the minimality of $\height(\lambda-\mu)$. So, $y-\alpha_j\notin \wt V$ for any $y\in Y$, and as a result the following points hold true.
		\begin{itemize}
			\item[a)] $f_j V_{y}=\{0\}$ $\forall$ $y\in Y$.
			\item[b)] $j\in J_{\lambda}$, as otherwise the Minkowski difference formula for $\wt V$ in Theorem C in \cite{Teja},\allowdisplaybreaks
			\[\wt V= \big(\wt V \cap \left(\lambda-\mathbb{Z}_{\geq 0}\Pi_{J_{\lambda}}\right)\big)-\mathbb{Z}_{\geq 0}\Delta_{J_{\lambda}^c,1},\]
			results in $\mu-\alpha_j\in\wt V$, which cannot happen as shown above.  
			\item[c)] $\langle y,\alpha_j^{\vee}\rangle\leq 0$ $\forall$ $y\in Y$, as $\langle y',\alpha_j^{\vee}\rangle>0$ for some $y'\in Y$ implies $y'-\alpha_j\in\wt V$ by Lemma \ref{L3.1}.
			\item[d)] $\langle\mu,\alpha_j^{\vee}\rangle<0$, as $\langle\mu,\alpha_j^{\vee}\rangle\geq 0$ implies $\langle\mu+\alpha_j,\alpha_j^{\vee}\rangle\geq 2$, which by Lemma \ref{L3.1} further implies $\mu-\alpha_j\in \wt V$.	 \end{itemize}
		For each $y\in Y$, we define $N_y:= U(\mathfrak{g}_{\alpha_j})V_{y}$. Observe by point a) above that $N_y$ is an integrable $\mathfrak{sl}_{\alpha_j}$-module for each $y\in Y$. Therefore, $\wt N_y$, and more generally $\bigcup\limits_{y\in Y}\wt N_y$ (which is a subset of $\wt V$), is $W_{\{j\}}$-invariant. This shows that $s_jY\subseteq \wt V$, as $Y\subseteq \bigcup\limits_{y\in Y}\wt N_y$. Point d) implies that $s_j\mu\succneqq \mu$, and so $\height(\lambda-s_j\mu)<\height(\lambda-\mu)$. As $s_j\mu\in s_j Y$, we have
		\[d(s_jY)\leq \height(\lambda-s_j\mu)<d(Y)=\height(\lambda-\mu).\] Now, we claim that $s_jY$ is also 212-closed in $\wt V$. Upon showing this claim, observe: i) when $\lambda\in Y$ we are immediately done, or else ii) when $\lambda\notin Y$, by the induction hypothesis applied to $s_jY$ \big(as $d(s_jY)<d(Y)$\big) we will be done. 
		
		To show that $s_jY$ is 212-closed, let $y_1,y_2\in Y$ and $\mu_1,\mu_2\in\wt V$ satisfy the following equation.
		\begin{equation}\label{E4.1}
		(s_jy_1)+(s_jy_2)=(\mu_1)+(\mu_2).
		\end{equation}
		We proceed in several cases below, and prove that both $\mu_1,\mu_2\in s_jY$, which implies the claim.
		\[\text{We define}\quad k:=\langle s_jy_1,\alpha_j^{\vee}\rangle+\langle s_jy_2,\alpha_j^{\vee}\rangle=\langle \mu_1,\alpha_j^{\vee}\rangle+\langle \mu_2,\alpha_j^{\vee}\rangle.\]
		Note, $\langle\mu',\alpha_j^{\vee}\rangle\in\mathbb{Z}$ $\forall$ $\mu'\in \wt V$ (as $j\in J_{\lambda}$). By Point c), $\langle s_jy,\alpha_j^{\vee}\rangle\in\mathbb{Z}_{\geq 0}$ $\forall$ $y\in Y$, and so~$k\in\mathbb{Z}_{\geq 0}$.\\
		Suppose $k=0$, which happens if and only if $\langle s_jy_1,\alpha_j^{\vee}\rangle=\langle s_jy_2,\alpha_j^{\vee}\rangle=0$, if and only if $s_jy_1=y_1$ and $s_jy_2=y_2$. Then equation \eqref{E4.1} implies both $\mu_1,\mu_2\in Y$, as $Y$ is 212-closed. Now, once again by point c), both $\langle\mu_1,\alpha_j^{\vee}\rangle$ and $\langle\mu_2,\alpha_j^{\vee}\rangle$ are non-positive. Finally, 
		\begin{align*}
		\begin{aligned}
		k=\langle\mu_1,\alpha_j^{\vee}\rangle+\langle\mu_2,\alpha_j^{\vee}\rangle=0&\implies \langle\mu_1,\alpha_j^{\vee}\rangle=\langle\mu_2,\alpha_j^{\vee}\rangle=0\\
		&\implies s_j\mu_t=\mu_t\text{ }\forall\text{ }t\in [2]\implies \mu_1,\mu_2\in s_jY.
		\end{aligned}
		\end{align*}
		So, for the rest of the proof, we assume that $k>0$, and also without loss of generality we assume that $\langle \mu_1,\alpha_j^{\vee}\rangle>0$. For $\langle\mu_2,\alpha_j^{\vee}\rangle$ there are two possibilities: either $\langle \mu_2,\alpha_j^{\vee}\rangle\geq 0$, or $\langle \mu_2,\alpha_j^{\vee}\rangle<0$. Accordingly, we proceed in two cases below. Firstly, recall by Lemma \ref{L3.1} that 
		\begin{equation}\label{E4.2} \mu_1-t_1\alpha_j\in\wt V\text{ }\text{ }\forall\text{ }0\leq t_1\leq \langle\mu_1,\alpha_j^{\vee}\rangle,\qquad\mu_2-t_2\alpha_j\in\wt V\text{ }\text{ }\forall\text{ }0\leq t_2\leq \max\{0, \langle\mu_2,\alpha_j^{\vee}\rangle\}.
		\end{equation}
		(1) $\langle\mu_2,\alpha_j^{\vee}\rangle<0$ : Then $\langle\mu_1,\alpha_j^{\vee}\rangle\geq k+1$, and so \eqref{E4.2} implies both $\mu_1-k\alpha_j,\text{ }\mu_1-(k+1)\alpha_j\in\wt V$. Now, by the 212-closedness of $Y$, the equation 
		\[ (y_1)+(y_2)=(\mu_1-k\alpha_j)+(\mu_2),\]
		which is obtained by subtracting $k\alpha_j$ from both sides of equation \eqref{E4.1}, results in $\mu_1-k\alpha_j\in Y$. This means we have $\mu_1-k\alpha_j\in Y$ and $\mu_1-(k+1)\alpha_j\in\wt V$, but this cannot happen as explained in the initial lines of the induction step. So, $\langle\mu_2,\alpha_j^{\vee}\rangle$ cannot be negative.\medskip\\
		(2) $\langle\mu_2,\alpha_j^{\vee}\rangle\geq 0$ : Then $\langle\mu_1,\alpha_j^{\vee}\rangle\leq k$. Recall, we assumed $\langle\mu_1,\alpha_j^{\vee}\rangle>0$. Recall, $s_j\mu_t=\mu_t-\langle\mu_t,\alpha_j^{\vee}\rangle\alpha_j$ $\forall$ $t\in [2]$. By \eqref{E4.2}, both $s_j\mu_1, s_j\mu_2\in\wt V$. Now, the equation
		\[(y_1)+(y_2)=(s_j\mu_1)+(s_j\mu_2),\]
		obtained by applying $s_j$ on both sides of equation \eqref{E4.1}, implies by the 212-closedness of $Y$ that both $s_j\mu_1,s_j\mu_2\in Y$. Therefore, both $\mu_1,\mu_2\in s_j Y$. Hence, the proof of the theorem is complete.
	\end{proof}
	Next, we obtain certain relations between the weights of $L(\lambda)$ and $L(s_i\bullet\lambda)$ for $\lambda\in P^+$. (Below, $\bullet$ denotes the dot-action of $W$ on $\mathfrak{h}^*$.) These relations will be used in the proof of Theorem \ref{thmA}.
	\begin{lemma}\label{L4.2}
		Let $\mathfrak{g}$ be a Kac--Moody algebra, and $\lambda\in P^+$. Fix $i\in\mathcal{I}$, and define $m_i:=\langle\lambda,\alpha_i^{\vee}\rangle+1$, so that $s_i\bullet\lambda=s_i\lambda-\alpha_i=\lambda-\big(\langle\lambda,\alpha_i^{\vee}\rangle+1\big)\alpha_i=\lambda-m_i\alpha_i$. 
		\begin{itemize}
			\item[(a)] If $\mu\in\wt L(\lambda)$ is such that $\height_{\{i\}}(\lambda-\mu)<m_i$, then $\mu-\big(m_i-\height_{\{i\}}(\lambda-\mu)\big)\alpha_i\in\wt L(s_i\bullet\lambda)$.
			\item[(b)] If $\mu\in\wt L(\lambda)$ is such that $\height_{\{i\}}(\lambda-\mu)\geq m_i$, then $\mu\in\wt L(s_i\bullet\lambda)$.
		\end{itemize}
	\end{lemma}
	\begin{proof}[\textnormal{\textbf{Proof}}] Fix $\lambda\in P^+$ and $i\in\mathcal{I}$. Note the following Minkowski difference formula for $\wt L(s_i\bullet \lambda)$.
		\begin{equation}\label{E4.3}
		\begin{aligned}[t]
		\wt L(s_i\bullet\lambda)=\wt M\left(\lambda-m_i\alpha_i,I_{L(\lambda-m_i\alpha_i)}\right)
		\end{aligned}
		\begin{aligned}[t]
		&=\wt M\big(\lambda-m_i\alpha_i, \{i\}^c\big)\\
		&=\wt L_{\{i\}^c}(\lambda-m_i\alpha_i)-\mathbb{Z}_{\geq0}\Delta_{\{i\},1}.
		\end{aligned}
		\end{equation}
		(The above formula holds true by \eqref{E2.3}.) In the above equation(s), $I_{L(\lambda-m_i\alpha_i)}=\{i\}^c$ as $\lambda\in P^+$. Recall, $L_{\{i\}^c}(\lambda-m_i\alpha_i)$ is the $\mathfrak{g}_{\{i\}^c}$-integrable module with highest weight $\lambda-m_i\alpha_i$.
		
		For part (a) of the lemma, we proceed via induction on $\height(\lambda-\mu)\geq 0$. In the base step, $\mu=\lambda$, and so $\lambda-m_i\alpha_i\in\wt L(s_i\bullet\lambda)$ trivially.\\
		Induction step: Let $\mu\neq \lambda$ be as in part (a). Fix $j\in\mathcal{I}$ such that $\mu+\alpha_j\in\wt L(\lambda)$. If $j=i$, then  \[\mu+\alpha_i-\left(m_i-\height_{\{i\}}(\lambda-\mu-\alpha_i)\right)\alpha_i=\mu-\left(m_i-\height_{\{i\}}(\lambda-\mu)\right)\alpha_i\in\wt L(s_i\bullet\lambda)\]   by the induction hypothesis applied to $\mu+\alpha_i$. So, we assume now that $j\neq i$. As $L(\lambda)$ is $\mathfrak{sl}_{\alpha_j}$-integrable and $\mu+\alpha_j\in\wt L(\lambda)$, there must exist a weight $\mu'\in\wt L(\lambda)$ such that 
		\[\mu'\in\mu+\mathbb{Z}_{>0}\alpha_j,\qquad \text{ }\langle\mu',\alpha_j^{\vee}\rangle>0,\qquad \text{ }\mu\in[s_j\mu',\mu']\subset\wt L(\lambda).\] 
		Notice that $\height_{\{i\}}(\lambda-\mu')=\height_{\{i\}}(\lambda-\mu)<m_i$. So, the induction hypothesis applied to $\mu'$ implies $\mu'-\left(m_i-\height_{\{i\}}(\lambda-\mu)\right)\alpha_i\in\wt L(s_i\bullet\lambda)$. Finally by Lemma \ref{L3.1}, \[\left\langle\mu'-\left(m_i-\height_{\{i\}}(\lambda-\mu)\right)\alpha_i,\alpha_j^{\vee}\right\rangle\geq\langle\mu',\alpha_j^{\vee}\rangle>0\quad\text{implies}\]
		\begin{align*}
		 \begin{aligned}\mu-\left(m_i-\height_{\{i\}}(\lambda-\mu)\right)\alpha_i\in&\ \big[s_j\big(\mu'-(m_i-\height_{\{i\}}(\lambda-\mu))\alpha_i\big), \mu'-\left(m_i-\height_{\{i\}}(\lambda-\mu)\right)\alpha_i\big]\\&\ \subset\wt L(s_i\bullet\lambda).
		 \end{aligned}
		 \end{align*}
		This completes the proof of part (a).\smallskip
		
		To show part (b), we prove via induction on $\height(\lambda-\mu)\geq m_i$ that $\mu\in\wt L(\lambda)$ and $\height_{\{i\}}(\lambda-\mu)\geq m_i$ together imply $\mu\in\wt L(s_i\bullet\lambda)$. In the base step, $\mu=\lambda-m_i\alpha_i$, and so $\mu\in\wt L(s_i\bullet\lambda)$ trivially.\\
		Induction step: Let $\mu\in \wt L(\lambda)$ be such that $\height_{\{i\}}(\lambda-\mu)\geq m_i$ and $\height(\lambda-\mu)>m_i$. Fix $j\in\mathcal{I}$ such that $\mu+\alpha_j\in\wt L(\lambda)$. Suppose $j\neq i$. Then clearly $\height_{\{i\}}(\lambda-\mu-\alpha_j)\geq m_i$. As $L(\lambda)$ is $\mathfrak{sl}_{\alpha_j}$-integrable and $\mu+\alpha_j\in\wt L(\lambda)$, there must exist a weight $\mu'\in\wt L(\lambda)$ such that
		\[ \mu'\in\mu+\mathbb{Z}_{>0}\alpha_j,\qquad \langle\mu',\alpha_j^{\vee}\rangle>0,\qquad \mu\in[s_j\mu',\mu']\subset\wt L(\lambda).\] 
		Note, $\height_{\{i\}}(\lambda-\mu')=\height_{\{i\}}(\lambda-\mu)$. The induction hypothesis applied to $\mu'$ implies $\mu'\in\wt L(s_i\bullet\lambda)$. Now by Lemma \ref{L3.1}, $\langle\mu',\alpha_j^{\vee}\rangle>0$ implies $\mu\in [s_j\mu',\mu']\subset\wt L(s_i\bullet\lambda)$.\\
		So, we assume for the rest of proof that $j=i$. If $\height_{\{i\}}(\lambda-\mu)>m_i$, then the induction hypothesis applied to $\mu+\alpha_i$ implies $\mu+\alpha_i\in\wt L(s_i\bullet\lambda)$. Equation \eqref{E4.3} then results in $\mu\in\wt L(s_i\bullet\lambda)$. Else if $\height_{\{i\}}(\lambda-\mu)=m_i$, then by part (a) of the lemma, we get \[\mu+\alpha_i-\left(m_i-\height_{\{i\}}(\lambda-\mu-\alpha_i)\right)\alpha_i=\mu+\alpha_i-(m_i-m_i+1)\alpha_i=\mu\in\wt L(s_i\bullet\lambda)\]
		as required. Hence, finally, the proof of the lemma is complete.
	\end{proof}
	Now we are in a position to prove Theorem \ref{thmA}. 
	\begin{proof}[\textnormal{\textbf{Proof of Theorem \ref{thmA}}}]
		The proof is divided into 11 steps for ease of access. Let $\emptyset\neq Y$ be 212-closed in $\wt V$. In steps 1--4, we show applying Theorem \ref{T4.1} that $\lambda\in \omega Y$ for some $\omega\in W_{I_V}$.\medskip\\
		\textbf{Step 1:} Assume that $\lambda\notin Y$ for this step as well as for the next three steps. Theorem \ref{T4.1} yields a sequence of nodes $i_1,\ldots, i_n\in J_{\lambda}$, $n\in\mathbb{N}$, such that the corresponding sequence of simple reflections $s_{i_1},\ldots,s_{i_n}$ satisfies the conditions in the statement of Theorem \ref{T4.1}. Without loss of generality, we assume that $n$ is the least size for any sequence of simple reflections satisfying the conditions in the statement of Theorem \ref{T4.1}---i.e., if a sequence $s_{l_1},\ldots,s_{l_m}$ satisfies those conditions, then $m\geq n$. If $i_1,\ldots, i_n\in I_V$, then the claim at the beginning of the proof is true. So, we assume that $\{i_1,\ldots,i_n\}\not\subset I_V$ and exhibit a sequence of simple reflections that satisfies the conditions in Theorem~\ref{T4.1} and with the size strictly less than $n$. This will contradict the minimality of $n$.\medskip\\
		\textbf{Step 2:} We define $k$ to be the smallest number in $[n]$ such that $i_k\notin I_V$, which exists by the assumption in the above step that $\{i_1,\ldots ,i_n\}\not\subset I_V$. For convenience, we define 
		\[
		J:=\begin{cases}
		\{i_t\text{ }|\text{ }1\leq t\leq k-1\}\! &\text{if }k>1,\\
		\emptyset\! &\text{if }k=1,
		\end{cases}\quad
		\nu:=
		\begin{cases}
		\prod\limits_{t=1}^{k-1}s_{i_t}\!&\text{if }k>1,\\
		1\in W\!&\text{if }k=1,
		\end{cases}\quad
		Z:=\begin{cases}
		\left(\prod\limits_{t=k+1}^n s_{i_t}\right)Y\!&\text{if }k<n,\\
		Y\!&\text{if }k=n.
		\end{cases}
		\]
		By the definition of $k$, observe that $J\subseteq I_V$ and $i_k\notin J$. Recall by Theorem \ref{T4.1}: (i)~$i_1,\ldots, i_n\in J_{\lambda}$; (ii)~both $Z$ and $s_{i_k}Z$ are 212-closed in $\wt V$; (iii)~when $n>1$, by the minimality of $n$, $\lambda\notin \left(\prod\limits_{t=r}^n s_{i_t}\right)Y$ for any $1<r\leq n$; (iv)~in particular when $k>1$, $\lambda\notin Z\cup s_{i_k} Z$. Note, $\lambda\in \left(\prod\limits_{t=1}^ns_{i_t}\right)Y=\nu s_{i_k}Z$, $\nu^{-1}\lambda\in s_{i_k}Z$ and $s_{i_k}\nu^{-1}\lambda\in Z$. Recall, both $\nu^{-1}\lambda, s_{i_k}\nu^{-1}\lambda\in \wt L(\lambda)$, as $\wt L(\lambda)$ is $W_{J_{\lambda}}$-invariant and $\nu^{-1},s_{i_k}\nu^{-1}\in W_{J_{\lambda}}$. Recall, $W_{J_{\lambda}}\lambda\subseteq \wt L(\lambda)$. With these observations, assumptions, and notation, we now show the claim at the beginning of the proof of the theorem in the next two steps.\medskip\\
	\textbf{Step 3:} Note, $V$ is not $\mathfrak{g}_{\{i_k\}}$-integrable as $i_k\notin I_V$. So, $\lambda-\mathbb{Z}_{\geq 0}\alpha_{i_k}\subseteq \wt V$. More generally, as $V_{s_{i_k}\bullet\lambda}$ is spanned by a maximal vector, $U(\mathfrak{g})V_{s_{i_k}\bullet\lambda}$ is a highest weight $\mathfrak{g}$-module with highest weight $s_{i_k}\bullet\lambda$, and therefore $\wt L(s_{i_k}\bullet \lambda)\subset\wt V$. As $\nu^{-1}\lambda\in \lambda-\mathbb{Z}_{\geq 0}\Pi_{J}$ and $i_k\in J_{\lambda}\setminus J$, 
		Lemma~\ref{L4.2}(a) applied for $\wt L_{J_{\lambda}}(\lambda)$ and $\wt L_{J_{\lambda}}(s_{i_k}\bullet \lambda)$ over $\mathfrak{g}_{J_{\lambda}}$ (with $\nu^{-1}\lambda$, $i_k$ in place of $\mu$, $i$, respectively) implies \[\nu^{-1}\lambda-\big(\langle\lambda,\alpha_{i_k}^{\vee}\rangle+1\big)\alpha_{i_k}\in\wt L_{J_{\lambda}}(s_{i_k}\bullet\lambda)\subseteq \wt L(s_{i_k}\bullet\lambda).\]
		Above, $L_{J_{\lambda}}(\lambda)$ is the integrable highest weight $\mathfrak{g}_{J_{\lambda}}$-module corresponding to $\lambda$; see Definition~\ref{D2.1}(1). Note, $W_{J_{\lambda}}\lambda\subseteq \wt L_{J_{\lambda}}(\lambda)\subseteq \wt L(\lambda)$. Observe that $i_k\notin I_{L\left(s_{i_k}\bullet \lambda\right)}$ \big(the integrability of $L(s_{i_k}\bullet\lambda$)\big), and so the Minkowski difference formula for $\wt L(s_{i_k}\bullet\lambda)$ \big(see \eqref{E4.3}\big) implies
		\begin{equation}\label{E4.5} \nu^{-1}\lambda-\big(\langle\lambda,\alpha_{i_k}^{\vee}\rangle+c\big)\alpha_{i_k}\in\wt L(s_{i_k}\bullet\lambda)\quad\forall\text{ }c\in\mathbb{N}.\end{equation}
		Now, $\nu^{-1}\lambda\in \lambda-\mathbb{Z}_{\geq 0}\Pi_{J}$ and $i_k\in J_{\lambda}\setminus J$ together imply
		\begin{equation}\label{E4.4} \langle \nu^{-1}\lambda,\alpha_{i_k}^{\vee}\rangle\geq \langle\lambda,\alpha_{i_k}^{\vee}\rangle\geq 0,\text{ } \text{ which further implies }\text{ }\height_{\{i_k\}}\left(\lambda-s_{i_k}\nu^{-1}\lambda\right)\geq \langle\lambda,\alpha_{i_k}^{\vee}\rangle. \end{equation}
		Finally, $s_{i_k}\nu^{-1}\lambda=\nu^{-1}\lambda-\langle \nu^{-1}\lambda,\alpha_{i_k}^{\vee}\rangle\alpha_{i_k}$, \eqref{E4.5} and\eqref{E4.4}, and $\wt L(\lambda)\cup\wt L(s_{i_k}\bullet\lambda)\subseteq\wt V$ together yield \begin{equation}\label{E4.6} s_{i_k}\nu^{-1}\lambda-c\alpha_{i_k}\in\wt V\quad\forall\text{ } c\in\mathbb{Z}_{\geq 0}.\end{equation}
		\textbf{Step 4:} Observe by the 212-closedness of $Z$ that \eqref{E4.6} and the following equation imply $\nu^{-1}\lambda\in Z$.
		\[
		2\left(s_{i_k}\nu^{-1}\lambda\right)=(\nu^{-1}\lambda)+\left(s_{i_k}\nu^{-1}\lambda-\langle \nu^{-1}\lambda,\alpha_{i_k}^{\vee}\rangle\alpha_{i_k}\right).
		\]
		Notice, $\nu$ is a product of $k-1$ many simple reflections, and also $\nu Z\subset \wt V$ \big(as $\nu \in W_{J}\subseteq W_{I_V}$\big). This means the sequence $s_{i_1},\ldots,\xcancel{s_{i_k}},\ldots,s_{i_n}$ satisfies the conditions in Theorem \ref{T4.1} with size $n-1$, which contradicts the minimality of $n$. Therefore, our assumption that $i_k\notin I_V$ is invalid. Hence, the proof of the claim at the beginning is complete.\smallskip\smallskip
		
		The rest of the proof is devoted to showing that $\tilde{w}Y=\wt_{\mathcal{I}_{\tilde{w}Y}}V$ for some $\tilde{w}\in W_{I_V}$, which proves Theorem \ref{thmA}. We ignore all the assumptions that were made in the previous steps and start afresh. The standing assumption for the rest of the proof is that $\lambda\in Y$, and this is justified by Steps 1--4 above. Firstly, note that if $\mathcal{J}_Y=\emptyset$, then
		Theorem \ref{thmA} is true by the proof of the implication $(4)\implies(1)$ of Proposition \ref{P3.11}. So, we also assume for the rest of the proof that $\mathcal{J}_Y\neq\emptyset$. With these, we once again proceed in steps below to complete the proof of Theorem \ref{thmA}. Before we start, we point out that the proof involves encountering/observing/listing various elements (or notation) such as $j_l, \gamma, \xi_l$, where $l\in\mathbb{N}$, $j_l\in \mathcal{J}_Y$ and $\gamma,\xi_l\in\Delta_{\mathcal{I}_Y\setminus\mathcal{J}_Y, 1}$, and $\gamma$ is a fixed element for the whole proof. When $Y$ satisfies condition \eqref{*} in Step 5 below and $Y\neq \wt_{\mathcal{I}_Y}V$, working with $Y$ and these elements, we produce a sequence of 212-closed subsets of $\wt V$ which satisfy the same minimality condition \eqref{*} and which do not coincide with any standard face of $\wt V$, and finally we exhibit a contradiction. We stress upon the reader to carefully track these elements in the course of the steps below. Recall, $W_{\mathcal{J}_Y}$ fixes $\lambda$ as $\mathcal{J}_Y\subseteq \mathcal{J}_V$.
		\medskip\\
		\textbf{Step 5:} We begin by recalling the definition of $P(Y)$ and some important properties of it.
		\begin{align*}
		\begin{aligned} P(Y):=\Bigg\{\lambda-\eta\text{ }\in Y\quad\Bigg|
		\end{aligned}\quad
		\begin{aligned}
		&\eta\in\Delta_{\mathcal{I}_Y\setminus\mathcal{J}_Y,1},\text{ }\eta\text{ is a sum of distinct simple roots, and}\\
		&\lambda-\eta'\in Y\text{ whenever }\eta'\in\Delta_{\mathcal{I}_Y\setminus\mathcal{J}_Y,1}\text{ and }\eta'\preceq \eta
		\end{aligned}
		\begin{aligned}\Bigg\}.
		\end{aligned}
		\end{align*}
		\begin{itemize} \item $P(Y)\subseteq P(\wt_{\mathcal{I}_Y}V)$, and $|P(Y)|\leq|P(\wt_{\mathcal{I}_Y}V)|<\infty$.
			\item $\lambda-\eta\in P(Y)$ for some $\eta\in\Delta_{\mathcal{I}_Y\setminus\mathcal{J}_Y,1}$ implies $\wt_{\supp(\eta)}V\subseteq Y$, by Lemma \ref{L3.10}.
			\item In particular, with $\eta$ as in the previous point, if $\eta'\in\Delta_{\mathcal{I}_Y\setminus\mathcal{J}_Y,1}$ and $\eta'\preceq \eta$, then $\lambda-\eta'\in P(Y)$. 
		\end{itemize}
		By Lemma \ref{L3.9}, there exists $v\in W_{\mathcal{J}_Y}$ such that $\lambda-\Pi_{\mathcal{I}_Y\setminus\mathcal{J}_Y}\subset vY$. Note, $\lambda\in vY$ as $W_{\mathcal{J}_Y}$ fixes~$\lambda$. For convenience, we assume that $v=1$; or equivalently, $\lambda-\Pi_{\mathcal{I}_Y\setminus\mathcal{J}_Y}\subset Y$. By the definition of $P(Y)$, this implies $\lambda-\Pi_{\mathcal{I}_Y\setminus\mathcal{J}_Y}\subseteq P(Y)$. 
		\begin{equation}\label{E4.7}
		\text{We define}\quad U_{Y}:=\left\{u\in W_{\mathcal{J}_Y}\text{ }\big|\text{ }\lambda-\Pi_{\mathcal{I}_Y\setminus\mathcal{J}_Y}\subset uY\right\}.
		\end{equation}
		Note that $U_Y\neq\emptyset$ as $1\in U_Y$. Without loss of generality, via replacing $Y$ if necessary by a suitable $U_Y$-conjugate of $Y$, we assume that $Y$ satisfies
		\begin{align*}\label{*}
		\big|P\left(\wt_{\mathcal{I}_Y}V\right)\big|-|P(Y)|\leq \big|P\left(\wt_{\mathcal{I}_{uY}}V\right)\big|-|P(uY)| \quad \forall\text{ } u\in U_Y. 
		\tag{$\ast$}
		\end{align*}
		If $P\left(\wt_{\mathcal{I}_{Y}}V\right)=P(Y)$, then $Y=\wt_{\mathcal{I}_Y}V$ by Proposition \ref{P3.11}, and we are done. Otherwise, we proceed in Step 6--10, and show that $\big|P(\wt_{\mathcal{I}_{Y}}V)\big|>|P(Y)|$ by condition \eqref{*} leads to a contradiction.\medskip\\ 
		\textbf{Step 6:} We fix a root $\gamma\in\Delta_{\mathcal{I}_Y\setminus\mathcal{J}_Y,1}$ with least $\height(\gamma)$ such that $\lambda-\gamma\in P(\wt_{\mathcal{I}_Y}V)\setminus P(Y)$. Note that such a root $\gamma$ may not be unique, and also that $\gamma$ is a sum of distinct simple roots. Now $\lambda-\Pi_{\mathcal{I}_Y\setminus\mathcal{J}_Y}\subset Y$ implies $\height(\gamma)\geq2$. In view of this and also \cite[Lemma 3.4]{Teja}, we fix a node $j_0\in\mathcal{J}_Y$ such that $\gamma-\alpha_{j_0}\in \Delta_{\mathcal{I}_Y\setminus\mathcal{J}_Y,1}$. By Proposition \ref{L3.6}, $\lambda-\gamma+\alpha_{j_0}\in\wt V$. Importantly, $\lambda-\gamma+\alpha_{j_0}\in P(\wt_{\mathcal{I}_Y}V)$, and $\height_{\{j_0\}}(\gamma-\alpha_{j_0})=0$. Now, observe:
		\begin{equation}\label{E5.8}
		\lambda-\gamma+\alpha_{j_0}\in P(Y)\text{ by the minimality of }\height(\gamma),\quad\text{and so by Lemma } \ref{L4.11},\quad \lambda-\gamma\notin Y.\end{equation}
		Next, by Corollary \ref{C3.7}(b), we fix a root $\xi\in\Delta_{\mathcal{I}_Y\setminus\mathcal{J}_Y,1}$ with least $\height(\xi)$ such that $\height_{\{j_0\}}(\xi)>0$ and $\lambda-\xi\in Y$. Once again observe:
		\[ \height(\xi)\geq 2,\quad\text{ and moreover as }\xi\in\Delta_{\mathcal{I}_Y\setminus\mathcal{J}_Y,1},\quad\lambda-\xi+\alpha_p\in\wt V\text{ }\text{ for some }p\in\mathcal{I}\implies p\in\mathcal{J}_Y.\]
		Fix a node $j_1\in\mathcal{J}_Y$ such that $\xi-\alpha_{j_1}\in \Delta_{\mathcal{I}_Y\setminus\mathcal{J}_Y,1}$; this exists by \cite[Lemma 3.4]{Teja}. By Proposition~\ref{L3.6}, $\lambda-\xi+\alpha_{j_1}\in\wt V$.\medskip\\
		\textbf{Step 7:} If $j_1=j_0$, then by the 212-closedness of $Y$, the equation 
		\[ (\lambda-\xi)+(\lambda-\gamma+\alpha_{j_0})=(\lambda-\xi+\alpha_{j_0})+(\lambda-\gamma)\]
		implies $\lambda-\gamma\in Y$, which contradicts $\lambda-\gamma\notin Y$ in \eqref{E5.8}. So, we assume that $j_1\neq j_0$. As $\height_{\{j_0\}}(\xi+\alpha_{j_1})>0$, the minimality of $\height(\xi)$ implies $\lambda-\xi+\alpha_{j_1}\not\in Y$. Now, if there exists $\mu'\in Y$ \big(in particular, if there exists $\mu'\in P(Y)$\big) such that $\mu'-\alpha_{j_1}\in\wt V$, then the 212-closedness of $Y$~and 
		\[(\mu')+(\lambda-\xi)=(\mu'-\alpha_{j_1})+(\lambda-\xi+\alpha_{j_1})\]	 
	together imply $\lambda-\xi+\alpha_{j_1}\in Y$, which contradicts $\lambda-\xi+\alpha_{j_1}\not\in Y$. In view of Lemma \ref{L3.1}, therefore $\langle\mu',\alpha_{j_1}^{\vee}\rangle\leq 0$ $\forall$ $\mu'\in Y$. Similarly, $\langle \lambda-\xi+\alpha_{j_1},\alpha_{j_1}^{\vee}\rangle\geq 2$ implies $ \lambda-\xi-\alpha_{j_1}\in\wt V$, which cannot happen as explained above. This discussion shows: \begin{equation}\label{E'''}
		j_1\neq j_0,\qquad\langle\mu',\alpha_{j_1}^{\vee}\rangle\leq 0\text{ }\text{ }\forall\text{ }\mu'\in Y,\quad\text{ and }\quad-\langle\xi,\alpha_{j_1}^{\vee}\rangle=\langle \lambda-\xi,\alpha_{j_1}^{\vee}\rangle<0.\end{equation}
		Recall, $\mathcal{J}_Y\subseteq \mathcal{J}_V\subseteq I_V$, and so both $\wt V$ and $\wt_{\mathcal{I}_Y}V$ are $W_{\mathcal{J}_Y}$-invariant. So, $s_{j_1}Y\subseteq \wt V$ and it is 212-closed (as is $Y$). Note, $\lambda\in s_{j_1}Y$ as $j_1\in\mathcal{J}_Y$. Importantly, observe that $\langle\mu',\alpha_{j_1}^{\vee}\rangle\leq 0$ implies $\mu'\preceq s_{j_1}\mu'=\lambda-s_{j_1}(\lambda-\mu')$ $\forall$ $\mu'\in Y$. This by the definitions of $P(Y)$ and $P(s_{j_1}Y)$ forces
		\begin{equation}\label{E''''}
		s_{j_1}\big(P(Y)\big)=P(Y)\qquad\text{and}\qquad P(Y)\subseteq P\left(s_{j_1}Y\right).
		\end{equation}
		In particular, $\lambda-\Pi_{\mathcal{I}_Y\setminus\mathcal{J}_Y}\subseteq P(Y)\subseteq s_{j_1}Y$ implies that $s_{j_1}\in U_Y$. \medskip\\
		\textbf{Step 8:} By the analysis in Step 7, observe the following points about the 212-closed subset $s_{j_1}Y$.
		\begin{itemize}
			\item[(a1)] $\mathcal{I}_{s_{j_1}Y}\subseteq \mathcal{I}_Y$, and so $P\left(\wt_{\mathcal{I}_{s_{j_1}Y}}V\right)\subseteq P\left(\wt_{\mathcal{I}_Y}V\right)$. By this and \eqref{E''''}, we have
			\[\Big|P\left(\wt_{\mathcal{I}_{s_{j_1}Y}}V\right)\Big|\leq \big|P\left(\wt_{\mathcal{I}_Y}V\right)\big|\quad\text{and}\quad\big|P\left(s_{j_1}Y\right)\big|\geq \big|P(Y)\big|.\]
			In view of the minimality condition (\ref{*}) satisfied by $Y$, these inequalities force \begin{align*}
			\Big|P\left(\wt_{\mathcal{I}_{s_{j_1}Y}}V\right)\Big|-\big|P\left(s_{j_1}Y\right)\big|=&\big|P\left(\wt_{\mathcal{I}_Y}V\right)\big|-\big|P(Y)\big|,\quad\text{ and thereby}\\
			\Big|P\left(\wt_{\mathcal{I}_{s_{j_1}Y}}V\right)\Big|=\big|P\left(\wt_{\mathcal{I}_Y}V\right)\big|&\quad\text{and}\quad\big|P\left(s_{j_1}Y\right)\big|=\big|P(Y)\big|.\\
			\text{These three equations further force: }\quad P(s_{j_1}Y)&=P(Y),\qquad\mathcal{I}_{s_{j_1}Y}=\mathcal{I}_Y,\qquad\mathcal{J}_{s_{j_1}Y}=\mathcal{J}_Y.
			\end{align*} 
			This implies $s_{j_1}Y$ satisfies the minimality condition \eqref{*} (with $s_{j_1}Y$ in place of $Y$ in it).
			\item[(a2)] $P(s_{j_1}Y)=P(Y)$ implies $\lambda-\gamma+\alpha_{j_0}\in P(s_{j_1}Y)$ and $\lambda-\gamma\notin P(s_{j_1}Y)$. So, Lemma \ref{L4.11} applied to $s_{j_1}Y$ implies $\lambda-\gamma \notin s_{j_1}Y$.
			\item[(a3)] Note that $s_{j_1}\xi\in\Delta_{\mathcal{I}_Y\setminus\mathcal{J}_Y,1}$, $s_{j_1}(\lambda-\xi)=\lambda-s_{j_1}\xi\in s_{j_1}Y$ and $\height_{\{j_0\}}(s_{j_1}\xi)>0$. In view of these, we fix $\xi_1\in\Delta_{\mathcal{I}_Y\setminus\mathcal{J}_Y,1}$---similar to $\xi\in Y$---with least $\height(\xi_1)$ such that $\height_{\{j_0\}}(\xi_1)>0$ and $\lambda-\xi_1\in s_{j_1}Y$. Now, $\height(\xi_1)\leq \height(s_{j_1}\xi)<\height(\xi)$.
		\end{itemize}\medskip
		\textbf{Step 9:} Repeat Steps 6, 7, 8 above with $s_{j_1}Y$ and $\xi_1$ replacing $Y$ and $\xi$, respectively (while keeping $\gamma$ as it is), in these three steps. This yields a node $j_2\in\mathcal{J}_Y$ such that $\xi_1-\alpha_{j_2}\in\Delta_{\mathcal{I}_Y\setminus\mathcal{J}_Y,1}$. If $j_2=j_0$, then by an argument similar to the one in Step 7, we get a contradiction to the minimality of $\height(\xi_1)$. So, we assume that $j_2\neq j_0$. This implies that $\height_{\{j_0\}}(s_{j_2}\xi_1)>0$. As $\lambda- s_{j_2}\xi_1\in s_{j_2}\big(s_{j_1}Y\big)$ (which is also 212-closed in $\wt V$), we will have a root $\xi_2\in\Delta_{\mathcal{I}_Y\setminus\mathcal{J}_Y,1}$ (similar to $\xi_1$ in point (a3)) with least height such that $\lambda-\xi_2\in s_{j_2}s_{j_1}Y$ and $\height_{\{j_0\}}(\xi_2)>0$, and also $\height(\xi_2)<\height(\xi_1)$.  
		
		What we have done so far is: starting with the pair $(Y, \xi)$ we obtained a node $j_1$ and the pair $(s_{j_1}Y, \xi_1)$, and then (in this step) we got another node $j_2$ and the pair $(s_{j_2}s_{j_1}Y, \xi_2)$. So, in this manner, if we repeat Steps 6, 7 and 8 at most ``$\height(\xi)-2$'' times, we will finally have a sequence of nodes $j_1,j_2,\ldots,j_d\in\mathcal{J}_Y$ and also a sequence of roots $\xi_1,\xi_2,\ldots,\xi_d\in\Delta_{\mathcal{I}_Y\setminus\mathcal{J}_Y,1}$ which satisfy:  		\begin{itemize}
			\item For each $1\leq l\leq d$, $j_l\neq j_0$, and $\bigg(\prod\limits_{t=0}^{l-1}s_{j_{l-t}}\bigg)Y$ is 212-closed in $\wt V$ and $\lambda\in\bigg(\prod\limits_{t=0}^{l-1}s_{j_{l-t}}\bigg)Y$.
			\item $\height(\xi_l)$ is least such that $\xi_l\in\Delta_{\mathcal{I}_Y\setminus\mathcal{J}_Y,1}$, $\lambda-\xi_l\in\bigg(\prod\limits_{t=0}^{l-1}s_{j_{l-t}}\bigg)Y$ and $\height_{\{j_0\}}(\xi_l)>0$ $\forall$ $1\leq  l\leq d$. Moreover, $\height(\xi_d)<\cdots<\height(\xi_2)<\height(\xi_1)$.
			\item $\xi_l-\alpha_{j_0}$ is not a root for any $1\leq l\leq d-1$, but $\xi_d-\alpha_{j_0}$ is a root. Note for any $t\in [d]$ that if $\xi_t-\alpha_{j_0}$ is a root, then it belongs to $\Delta_{\mathcal{I}_Y\setminus\mathcal{J}_Y,1}$, as $j_0\in \mathcal{J}_Y$. In the first line of this point, we assumed $\xi_d-\alpha_{j_0}$ to be a root, as there must exist such a stage $d$ given that $\height_{\{j_0\}}(\xi_l)>0$ $\forall$ $l\in [d]$, and also as $\height(\xi_l)$ decreases as $l$ increases.  
			\item The assertions that are analogous to the ones in: points (a1), (a2), (a3), \eqref{E'''}, and $\eqref{E''''}$, also hold for $j_{l+1}$, $\bigg(\prod\limits_{t=0}^{l-1}s_{j_{l-t}}\bigg)Y$ and $\xi_l$ (in place of $j_1$, $Y$ and $\xi$, respectively) $\forall$~$l< d$. 
		\end{itemize}
		\textbf{Step 10:} We define for convenience, $w:=\prod\limits_{t=0}^{d-1}s_{j_{d-t}}$. Observe the following points about the 212-closed subset $wY$ of $\wt V$. 
		\begin{itemize} 
			\item[(b1)]
			$P(wY)=P(Y)$, $\mathcal{I}_{wY}=\mathcal{I}_Y$ and $\mathcal{J}_{wY}=\mathcal{J}_Y$.
			\item[(b2)] $\lambda-\gamma+\alpha_{j_0}\in P(wY)$, and $\lambda-\gamma \notin wY$ by Lemma \ref{L4.11}.
			\item[(b3)] $\lambda-\xi_d\in wY$, and importantly $\xi_d-\alpha_{j_0}\in\Delta_{\mathcal{I}_Y\setminus\mathcal{J}_Y,1}$. So, $\lambda-\xi_d+\alpha_{j_0}\in\wt V$. 
		\end{itemize}
		Finally, by point b3) and by the 212-closedness of $wY$, observe that the equation 
		\[ (\lambda-\xi_d)+(\lambda-\gamma+\alpha_{j_0})=(\lambda-\xi_d+\alpha_{j_0})+(\lambda-\gamma)\] 
		contradicts $\lambda-\gamma\notin wY$ in point (b2). So, our assumption that $|P(\wt_{\mathcal{I}_Y}V)|>|P(Y)|$, which led to the contradiction in the previous line, is invalid. Therefore, $P(\wt_{\mathcal{I}_Y}V)=P(Y)$, and this implies $Y=\wt_{\mathcal{I}_Y}V$ as desired. In the final step below, we show the equivalence of the three notions given at the end of the statement of Theorem \ref{thmA}, which is the only thing that remains to be shown.\medskip\\
		\textbf{Step 11:} Clearly, the standard faces of $\wt V$, and so all of their $W_{I_V}$-conjugates, satisfy the definition of a weak-$\mathbb{A}$-face (for any non-trivial subgroup $\mathbb{A}\subseteq(\mathbb{R},+)$). By Observation \ref{O2.3}, weak-$\mathbb{A}$-faces of $\wt V$ are 212-closed. Now, Steps 1--10 say that 212-closed subsets of $\wt V$ can be conjugated to standard faces by elements of $W_{I_V}$. Hence, finally, the proof of Theorem \ref{thmA} is complete.
	\end{proof}
	\section{Proof of Theorem \ref{thmB}: Weak faces of convex hulls of $\wt V$}
	Throughout this section, we assume that $\mathbb{F}$ is a subfield of $\mathbb{R}$ and $\mathbb{A}$ a non-trivial subgroup of $(\mathbb{F},+)$. The goal of this section is to prove Theorems \ref{thmB} and \ref{T2.3}. We begin by extending Remarks \ref{R3.1} and \ref{R3.2} to arbitrary $\mathbb{F}$-convex sets in the following lemma. 
	\begin{lemma}\label{L5.1}
		Let $\mathbb{F}$ be a subfield of $\mathbb{R}$, and let $X$ be an $\mathbb{F}$-convex subset of an $\mathbb{F}$-vector space. Let $\emptyset\neq Y$ be 212-closed in $X$. Fix two arbitrary points $x, y\in X$, and a number $t\in (0,1)\cap\mathbb{F}$. 
		\begin{itemize}
			\item[(a)] Suppose $x, y\in Y$. Then $rx+(1-r)y\in Y$ $\forall$ $r\in [0,1]\cap\mathbb{F}$. Therefore, 212-closed subsets of convex sets are also convex.
			\item[(b)] Suppose $tx+(1-t)y\in Y$. Then $\{rx+(1-r)y$ $|$ $r\in [0,1]\cap\mathbb{F}\}\subset Y$.
		\end{itemize}
		The analogous assertions about the weak-$\mathbb{A}$-faces (in place of 212-closed subsets) of $X$ for arbitrary non-trivial additive subgroups $\mathbb{A}$ of $(\mathbb{F},+)$ also hold true.
	\end{lemma}
	\begin{proof}
		Let $Y$ and $X$ be as in the statement. The proof of part (a) easily follows by the definition of 212-closed subsets and the following equation.
		\[(x)+(y)=\big(rx+(1-r)y\big)+\big((1-r)x+ry\big).\]
		To show part (b), assume that $tx+(1-t)y\in Y$ for some $t\in (0,1)\cap\mathbb{F}$. We prove that $y\in Y$, and the proof for $x\in Y$ is similar. Write $tx+(1-t)y= y+t(x-y)$. Note that $y+r(x-y)\in X$ $\forall$ $r\in [0,1]\cap\mathbb{F}$, as $X$ is $\mathbb{F}$-convex. When $t\leq \frac{1}{2}$, by the 212-closedness of $Y$, the equation 
		\[2\big(y+t(x-y)\big)=\big(y+2t(x-y)\big)+(y)\]
		 implies $y\in Y$. So, we assume for the rest of the proof that $t\geq \frac{1}{2}$. Fix $m\in\mathbb{N}$ such that $1-\frac{1}{2^m}\leq t< 1-\frac{1}{2^{m+1}}$; such a number $m$ exists as $t\geq \frac{1}{2}$. Observe: \[\frac{1}{2}\leq 2^{i}t-2^{i}+1\leq 1\ \  i\in [m-1],\qquad\text{and}\qquad 0\leq 2^{m}t-2^{m}+1< \frac{1}{2}.
		 \]
		 We show that $y+\left(2^{m}t-2^{m}+1\right)(x-y)\in Y$, so that we will be done as above. For this, consider the following system of equations:
		\[
		2\Big(y+\left(2^{i-1}t-2^{i-1}+1\right)(x-y)\Big)=\Big(y+\left(2^{i}t-2^{i}+1\right)(x-y)\Big)+(x),\quad 1\leq i\leq m.
		\]
		For each $i\in [m]$, inductively, it can be observed that the $i^{\text{th}}$ equation in the above system (as $Y$ is 212-closed) implies that $y+\left(2^{i}t-2^{i}+1\right)(x-y)\in Y$. In particular, $y+\left(2^{m}t-2^{m}+1\right)(x-y)\in Y$ as required. Hence, the proof of the lemma is complete.
	\end{proof}
	\begin{proof}[\textnormal{\textbf{Proof of Theorem \ref{T2.3}}}]
		Let $X$, $Y$, $S$ and $D$ be as in the statement. We first prove part (a). Assume that $X=\conv_{\mathbb{F}}S$. We are required to show $Y=\conv_{\mathbb{F}}\big(Y\cap S\big)$. For this, pick an arbitrary element $y\in Y$. Assume that $y=\sum\limits_{i=1}^nr_is_i$ for some $s_i\in S$ and $r_i\in\mathbb{F}_{>0}$ summing up to 1. We will show that $s_i\in Y$ for all $i\in [n]$, which implies the result. If $n=1$, then $r_1=1$ and $y=s_1\in S$, and so we are done. Else when $n>1$, write:
		\[ y= s_1+(1-r_1)\left[\left(\sum\limits_{i=2}^n\frac{r_i}{1-r_1}s_i\right)-s_1\right].\]
		Since $r_1\in (0,1)\cap\mathbb{F}$, Lemma \ref{L5.1}(b) implies $s_1\in Y$ and also $\sum\limits_{i=2}^n\frac{r_i}{1-r_1}s_i\in Y$. Finally, by invoking an induction argument on $n$, it can be shown that $s_i\in Y$ $\forall$ $i\in [n]$. Furthermore, it follows by the definition of 212-closed subsets that $Y\cap S$ is a (non-empty) 212-closed subset of $S$. This completes the proof of part (a).
		
		For part (b), assume that $X=\conv_{\mathbb{F}}S+\mathbb{F}_{\geq 0}D$. We define for convenience, \[Y':=\conv_{\mathbb{F}}(Y\cap S)+\mathbb{F}_{\geq 0}\big\{d\in D\text{ }|\text{ }(d+Y)\cap Y\neq\emptyset\big\}.\]
		We first show $Y\subseteq Y'$. For this, fix an arbitrary element $y\in Y$. Write $y=y_s+\sum\limits_{i=1}^nr_id_i$ for some $y_s\in\conv_{\mathbb{F}}S$, $d_i\in D$ and $r_i\in\mathbb{F}_{\geq 0}$ $\forall$ $i\in [n]$. By the description of $X$ and the 212-closedness of $Y$,
		\begin{equation}\label{E'5.1}
		2\left(y_s+\sum_{i=1}^nr_id_i\right)=(y_s)+\left(y_s+\sum_{i=1}^n2r_id_i\right)
		\end{equation}
		implies $y_s\in Y$. So, $y_s\in \conv_{\mathbb{F}}(Y\cap S)$ by part (a). When $y=y_s$, we are done by the previous line. So, we assume that not all $r_i$ are 0. Now, fix a number $k\in [n]$ such that $r_k>0$. Pick the least $m\in\mathbb{N}\sqcup\{0\}$ such that $2^mr_k\geq 1$. Note by equation \eqref{E'5.1} that $\left(y_s+\sum_{i=1}^n2r_id_i\right)\in Y$. Via writing a system of equations similar to \eqref{E'5.1}, it can be shown inductively that $y_s+\sum\limits_{i=1}^n2^jr_id_i\in Y$ for each $0\leq j\leq m$. So in particular, $y_s+\sum\limits_{i=1}^n2^mr_id_i\in Y$. By the 212-closedness of $Y$, the equation
		\[
		2\left(y_s+\sum_{i=1}^n2^mr_id_i\right)=\left(y_s+d_k\right)+\bigg(y_s+ (2^mr_k-1)d_k+\sum_{i\in [n]\setminus\{k\}} 2^mr_id_i\bigg)
		\]
	 implies $y_s+d_k\in Y$, which proves $(d_k+Y)\cap Y \neq\emptyset$. This shows $Y\subseteq Y'$ as $k$ is arbitrary.
	 
	 Finally, we show $Y'\subseteq Y$ to complete the proof of part (b). Pick an arbitrary $y'\in Y'$, and write $y'=y'_s+\sum_{i=1}^{n'}r'_id'_i$ for some $y'_s\in \conv_{\mathbb{F}}(Y\cap S)$, $d'_i\in \big\{d\in D$ $|$ $(d+Y)\cap Y\neq\emptyset\big\}$ and $r'_i\in \mathbb{F}_{\geq 0}$ $\forall$ $i\in [n']$. Assume that not all $r'_i$ are 0, as there is nothing to prove otherwise. Fix a number $k\in [n']$ such that $r'_k>0$, and correspondingly an element $y\in Y$ such that $y+d'_k\in Y$; such an element $y$ exists as $d'_k\in \big\{d\in D$ $|$ $(d+Y)\cap Y\neq\emptyset\big\}$. We show $y'_s+r'_kd'_k\in Y$; by this (as $k$ is arbitrary) and by using the 212-closedness of $Y$ one can easily prove that $y'\in Y$, which implies the result. Pick the least $m'\in \mathbb{N}\sqcup\{0\}$ such that $2^{m'}\geq r'_k$. Since $y+d'_k\in Y$, by a similar procedure to the one given below \eqref{E'5.1}, it can be easily shown that $y+2^{m'}d'_k\in Y$. Finally, by the 212-closedness~of~$Y$,
	 \[(y'_s)+(y+2^{m'}d'_k)=(y'_s+r'_kd'_k)+\big(y+(2^{m'}-r'_k)d'_k\big)\]
	 implies $y'_s+r'_kd'_k\in Y$. Hence, the proof of the theorem is complete.
	\end{proof}
	Finally, using Theorems \ref{T2.3} and \ref{thmA}, we show Theorem \ref{thmB}.
	\begin{proof}[\textnormal{\textbf{Proof of Theorem \ref{thmB}}}]
	We first prove the implication (4) $\implies$ (1). Fix $J\subseteq \mathcal{I}$ and $w\in W_{I_V}$. Observe that if $J=\mathcal{I}$, then $\conv_{\mathbb{R}}\big(w[(\lambda-\mathbb{Z}_{\geq 0}\Pi_J)\cap\wt V]\big)=\conv_{\mathbb{R}}\wt V$ (which itself equals the maximizer subset for $0\in\mathfrak{h}^*$ in it), and so there is nothing to prove in this case. So we assume that $J\subsetneqq\mathcal{I}$. By the linear independence of $\Pi\subset \mathfrak{h}^*$, we fix elements $\omega^{\vee}_i\in\mathfrak{h}$ $\forall$ $i\in\mathcal{I}$ satisfying $\langle\alpha_j,\omega^{\vee}_i\rangle=\delta_{i,j}$ $\forall$ $i,j\in\mathcal{I}$. We define a linear functional $\psi$ on $\mathfrak{h}^{*}$ by $\psi(\alpha):=\Big\langle\alpha,\sum\limits_{i\in J^c}w\omega_i^{\vee}\Big\rangle$, the evaluation of $\alpha$ against $\sum\limits_{i\in J^c}w\omega_i^{\vee}$.  Observe:	\[
		\conv_{\mathbb{R}}\big(w[(\lambda-\mathbb{Z}_{\geq 0}\Pi_K)\cap\wt V]\big)=w\big[(\lambda-\mathbb{R}_{\geq 0}\Pi_K)\cap(\conv_{\mathbb{R}}\wt V)\big]\quad\forall\text{ }w\in W_{I_V},\text{ }K\subseteq \mathcal{I}.
		\]
		Now, it can be easily checked that $\conv_{\mathbb{R}}\big(w[(\lambda-\mathbb{Z}_{\geq 0}\Pi_J)\cap\wt V]\big)$ maximizes $\psi$ inside $\conv_{\mathbb{R}}\wt V$.\\
		Next, (1) $\implies$ (2) $\implies$ (3) follows from Observation \ref{O2.3}. As mentioned before, (3) $\implies$ (4) follows by Theorem \ref{T2.3} in view of Theorem \ref{thmA}. Hence, the proof of the theorem is complete.
	\end{proof}
	\section{Proof of Theorem \ref{thmC}: Weak faces of root systems}
	Let $\mathfrak{g}$ be a Kac--Moody algebra with root system $\Delta$, and fix a non-trivial additive subgroup $\mathbb{A}\subseteq(\mathbb{R},+)$. The goal of this section is to prove Theorem \ref{thmC}, which: 1)~finds all the 212-closed subsets of $\Delta$ and of $\Delta\sqcup\{0\}$; 2)~shows the equivalence of the notions the 212-closed subsets of $\Delta\sqcup\{0\}$, weak-$\mathbb{A}$-faces of $\Delta$ and weak-$\mathbb{A}$-faces of $\Delta\sqcup\{0\}$; 3)~shows the equivalence of the 212-closed subsets of $\Delta$ and all the notions in the previous point except for just two cases $\mathfrak{g}=\mathfrak{sl}_3(\mathbb{C})\text{ and } \widehat{\mathfrak{sl}_3(\mathbb{C})}$. Recall: i)~the notation $\Delta_s,\Delta_l,\theta$ for $\Delta$ of finite type, and ii)~for $\Delta$ of affine type and $Y\subseteq\Delta$, the definitions of $\mathring{\Delta}$, $\mathring{W}$, $Y_s$ and $Y_l$ from the definitions in \eqref{E2.4} and the paragraph above it.
	\begin{equation}\label{E5.1}
	Y_s:=\begin{cases}
	(Y\cap\mathring{\Delta}_s)+\mathbb{Z}\delta &\text{if } Y\cap \mathring{\Delta}_s\neq\emptyset,\\
	\emptyset &\text{if }Y\cap\mathring{\Delta}_s=\emptyset,
	\end{cases}\qquad\quad  Y_l:=\begin{cases}
	(Y\cap\mathring{\Delta}_l)+r\mathbb{Z}\delta &\text{if } Y\cap \mathring{\Delta}_l\neq\emptyset,\\
	\emptyset &\text{if }Y\cap\mathring{\Delta}_l=\emptyset.
	\end{cases}
	\end{equation}
	
	We begin this section with the following remark on 212-closed subsets of $\Delta$ and of $\Delta\sqcup\{0\}$, which will be referred to later at many places. 
	\begin{rem}\label{R5.2} 
		Let $\Delta$ be an indecomposable root system, and $Y\subseteq \Delta$ be 212-closed in $\Delta$.\\
		(1) Suppose $Y$ contains an imaginary root $\eta$. Then $Y=\Delta$ by the following implications: \begin{align}
		&2(\eta)=(3\eta)+(-\eta)\text{ }\text{ and }\text{ }(\mathbb{Z}\setminus\{0\})\eta\subset\Delta\text{ }\text{ together}\text{ imply } -\eta\in Y.\label{E5.2}\\
		&(\eta)+(-\eta)=(\xi)+(-\xi)\text{ }\text{for any }\xi\in\Delta \text{ }\implies\Delta\subseteq Y\implies \Delta=Y.\label{E5.3}
		\end{align}
		Similarly, if $Y$ is 212-closed in $\Delta\sqcup\{0\}$ and $Y$ contains 0 or an imaginary root, then $Y=\Delta\sqcup\{0\}$. Thus, if $Y$ is proper 212-closed in $\Delta$, or in $\Delta\sqcup\{0\}$, then $Y\subset \Delta^{re}$ the set of real roots of $\Delta$.\\
		(2) Similar to point (1), if there is a root $\gamma$ such that $\pm\gamma\in Y$, then $Y=\Delta$ by \eqref{E5.3}. \\
		(3) Let $Y'$ be 212-closed in $\Delta\sqcup\{0\}$. In view of point (1), $Y'=\Delta\sqcup\{0\}$ if and only if $\Delta\subset Y'$. So, if $Y'\subsetneqq \Delta\sqcup\{0\}$, then by point (1) and the definition of 212-closed subsets, $Y'$ is proper 212-closed in $\Delta$. Similar assertions hold true for the weak-$\mathbb{A}$-faces of $\Delta$ and $\Delta\sqcup\{0\}$. 
	\end{rem}
	Before we prove Theorem \ref{thmC}, we make some observations concerning the 212-closed subsets and weak-$\mathbb{A}$-faces of $\Delta\sqcup\{0\}$ that are common to the proofs of parts (\ref{thmC}1) and (\ref{thmC}2).
	\begin{rem}\label{R7.2}
		Let $\mathfrak{g}$ be of finite type. Recall, $\mathfrak{g}\simeq L(\theta)$, and so $\Delta\sqcup\{0\}=\wt L(\theta)$. Therefore by Theorem \ref{thmA}, the proper 212-closed subsets of $\Delta\sqcup\{0\}$ and the weak-$\mathbb{A}$-faces of $\Delta\sqcup\{0\}$ are the same, and they both coincide with the $W$-conjugates of all the proper standard faces of $\wt L(\theta)=\Delta\sqcup\{0\}$. Recall, the proper standard faces of $\wt L(\theta)$ are the subsets of the form $(\theta-\mathbb{Z}_{\geq 0}\Pi_J)\cap \wt L(\theta)$ for $J\subsetneqq \mathcal{I}$. These sets and their $W$-conjugates also satisfy the definition of a weak-$\mathbb{A}$-face of $\Delta$.
	\end{rem}
	We now show Theorem \ref{thmC}, by proving each of its parts (\ref{thmC}1)--(\ref{thmC}4) separately.
	\begin{proof}[\textnormal{\textbf{Proof for part (\ref{thmC}1):}}]
		Assume that $\mathfrak{g}$ is of finite type and $\mathfrak{g}\neq \mathfrak{sl}_3(\mathbb{C})$. To show the equivalence of the four notions in the statement of part (\ref{thmC}1), we prove the following claim.
		\begin{claim}
			The 212-closed subsets of $\Delta$ and of $\Delta\sqcup\{0\}$ are the same.
		\end{claim}
		Assuming the truthfulness of the above claim, one can easily observe: i)~the proper 212-closed subsets of $\Delta$ are the same as the $W$-conjugates of the proper standard faces of $\Delta\sqcup\{0\}$; ii)~as the weak-$\mathbb{A}$-faces of $\Delta$ are also 212-closed in $\Delta$, point i) and the last sentence in Remark \ref{R7.2} imply that the proper 212-closed subsets of $\Delta$, the proper weak-$\mathbb{A}$-faces of $\Delta$ and the $W$-conjugates of proper standard faces of $\Delta\sqcup\{0\}$ are all the same. Finally, once again by Remark \ref{R7.2}, all the four notions are same. Therefore, all that remains to prove is the above claim. For this, first of all notice that if $Y'\subsetneqq \Delta\sqcup\{0\}$ is 212-closed in $\Delta\sqcup\{0\}$, then $Y'$ is proper 212-closed in $\Delta$ by Remark \ref{R5.2}(3). We prove the converse result in four layers: 1)~when the rank of $\Delta$ is $\geq 3$; 2)~when $\Delta$ is of type $A_1$; 3)~when $\Delta$ is of type $B_2$; 4)~when $\Delta$ is of type $G_2$. Observe that these four cases exhaust all $\mathfrak{g}$ of finite type, $\mathfrak{g}\neq \mathfrak{sl}_3(\mathbb{C})$. In the proofs for these four cases, we fix $\emptyset\neq Y\subsetneqq\Delta$ to be 212-closed in $\Delta$.  \medskip\\
		\textbf{For $\Delta$ of rank $\geq 3$:} Assume that the rank of $\mathfrak{g}$ is $\geq 3$. To show $Y$ is proper 212-closed in $\Delta\sqcup\{0\}$, consider the equation:
		\[(\gamma_1)+(\gamma_2)=(x)+(y)\quad \text{ for some }\gamma_1,\gamma_2\in Y,\text{ }x,y\in\Delta\sqcup\{0\}.\]
		Proving $x,y\in Y$ will imply the result. In the above equation, if both $x\text{ and }y$ are roots, then they belong to $Y$ as $Y$ is 212-closed in $\Delta$. Else if $x=y=0$, then $\gamma_2=-\gamma_1$, and this means $\pm\gamma_1\in Y$. Now, \eqref{E5.3} (with $\gamma_1$ in place of $\eta$) implies $Y=\Delta$, which contradicts $Y\subsetneqq\Delta$. So, we assume $x\in \Delta$ and $y=0$ for the rest of the proof, and once again show that this leads to $Y=\Delta$ contradicting $Y\subsetneqq\Delta$.  
		
		Pick a node $i\in\mathcal{I}$ and $w\in W$ such that $i$ is a leaf (a degree one node) in the Dynkin diagram and $wx=\alpha_i$. The previous line can be verified by recalling: (i)~any two roots of the same length are $W$-conjugate, (ii)~there can be at most two lengths in $\Delta$, (iii)~there always exist two leaves in the Dynkin diagram out of which one corresponds to a short simple root and the other to a long simple root. For convenience, call $wY,w\gamma_1\text{ and }w\gamma_2$ as $Z,\beta_1\text{ and }\beta_2$, respectively. Note, $Z$ is proper 212-closed in $\Delta$ (as is $Y$), and $\beta_1,\beta_2\in Z$. As the Dynkin diagram is connected and $|\mathcal{I}|\geq 3$, choose a node $j\in\mathcal{I}\setminus\{i\}$ that is adjacent to $i$ (or equivalently, $\langle\alpha_i,\alpha_j^{\vee}\rangle<0$). As $\alpha_i=\beta_1+\beta_2$, the previous line implies that either $\langle\beta_1,\alpha_j^{\vee}\rangle<0$ or $\langle\beta_2,\alpha_j^{\vee}\rangle<0$. Without loss of generality assume that $\langle\beta_1,\alpha_j^{\vee}\rangle<0$. Pick $k\in\mathcal{I}\setminus\{i,j\}$ such that there is at least one edge joining the node $k$ to a node in $\{i,j\}$, such a node $k$ exists as $|\mathcal{I}|\geq 3$ and as the Dynkin diagram is connected. The two assumptions that $i$ is a leaf, and $i$ and $j$ are adjacent together force $\langle\alpha_{i},\alpha_k^{\vee}\rangle=0$, and so $\langle\alpha_j,\alpha_k^{\vee}\rangle<0$. In particular, this implies $\alpha_j+\alpha_k$ is a root. Now, we define a root $\xi$ as follows: 
		\[ 
		\xi:=\begin{cases}
		\alpha_j\quad &\text{if } \beta_1\neq-\alpha_j,\\
		\alpha_j+\alpha_k\quad&\text{if }\beta_1=-\alpha_j.
		\end{cases}
		\]
		Recall, we assumed that $\langle\beta_1,\alpha_j^{\vee}\rangle<0$. So, $\beta_1+\alpha_j\in\Delta$ when $\beta_1\neq -\alpha_j$. By the way $\xi$ is defined, it can easily seen that both $\alpha_i+\xi\text{ and }\beta_1+\xi$ are roots. Finally, the following implications (by the 212-closedness of $Z$) lead to the desired contradiction, finishing the proof of the claim in this case.
		\begin{align*}
		(\beta_1)+(\beta_2)=\alpha_i=(\alpha_i+\xi)+(-\xi)&\implies \alpha_i+\xi,-\xi\in Z.\\
		(\beta_2)+(-\xi)=(\beta_2+\beta_1)+(-\beta_1-\xi)&\implies \beta_2+\beta_1,-\beta_1-\xi\in Z.\\
		(\beta_1)+(-\beta_1-\xi)=(\alpha_i)+(-\alpha_i-\xi)&\implies \alpha_i,-\alpha_i-\xi\in Z.\\
		\text{By Remark }\ref{R7.2},\quad\pm (\alpha_i+\xi)\in Z\implies Z=&\Delta\Rightarrow\!\Leftarrow Z\subsetneqq\Delta\text{ }(\text{as }Y\subsetneqq \Delta).	                  \end{align*}
		
		To look into the rest of the aforementioned three cases where the rank of $\mathfrak{g}$ is $\leq 2$, we record some necessary notation. Let $\mathcal{I}$ be equal to $\{1\}$, and $\{1,2\}$, in the cases when $\mathfrak{g}$ is of type $A_1$, and respectively when $\mathfrak{g}$ is not of type $A_1$. When $\mathfrak{g}$ is of type either $B_2\text{ or }G_2$, let $2\in \mathcal{I}$ correspond to the long simple root in $\Delta$. We use without further mention the following basic facts: i)~every long root of $\Delta$ belongs to $W\alpha_2$, and ii)~every short root of $\Delta$ belongs to $W\alpha_1$. \medskip\\
		\textbf{For type $A_1$:} In this case, it can be easily verified that $\{\alpha_1\}$ and $\{-\alpha_1\}$ are the only proper 212-closed subsets of $\Delta$, and also of $\Delta\sqcup\{0\}$.\medskip\\
		\textbf{For type $B_2$:} Assume that $\Delta$ is of type $B_2$. Recall:
		\[
		\Delta=\big\{\pm(\alpha_1),\pm(\alpha_2),\pm (\alpha_2+\alpha_1),\pm(\alpha_2+2\alpha_1)\big\}\quad\text{ and }\quad\theta=\alpha_2+2\alpha_1.
		\]
		We proceed in two steps below, and prove that $Y=w[(\theta-\mathbb{Z}_{\geq 0}\Pi_J)\cap\Delta]$ for some $w\in W$ and $J\subseteq \mathcal{I}$, which implies the claim at the beginning of the proof of part (\ref{thmC}1).\smallskip\\
		(1) Assume that $Y$ contains a short root, say $u^{-1}\alpha_1$ for $u\in W$. This implies $\alpha_1\in uY$, and so 
		\[
		2(\alpha_1)=(\alpha_2+2\alpha_1)+(-\alpha_2)\implies \alpha_2+2\alpha_1,-\alpha_2\in uY,
		\]
		since $uY$ is 212-closed in $\Delta$. So, $\{-\alpha_2,\alpha_1,\alpha_2+2\alpha_1\}\subseteq uY$. Now, in view of Remark \ref{R5.2}(2), observe that none of the roots $-\alpha_2-2\alpha_1,-\alpha_1\text{ and }\alpha_2$ fall in $uY$, as $uY$ is proper 212-closed in $\Delta$. Observe once again via the same reasoning and the following implications that $\pm(\alpha_2+\alpha_1)\not\in uY$.
		\[\alpha_2+\alpha_1\in uY\text{ }\text{ and }\text{ } 2(\alpha_2+\alpha_1)=(\alpha_2+2\alpha_1)+(\alpha_2)\implies \alpha_2\in uY\implies\pm\alpha_2\in uY. \text{ }\text{ Similarly,}\]
		\[-\alpha_2-\alpha_1\in uY\text{ }\text{and } 2(-\alpha_2-\alpha_1)=(-\alpha_2-2\alpha_1)+(-\alpha_2)\implies -\alpha_2-2\alpha_1\in uY\implies\pm(\alpha_2+2\alpha_1)\in uY. \]
		Therefore, $uY=\{-\alpha_2,\alpha_1,\alpha_2+2\alpha_1\}$. Finally, notice that we are done by the following equation. 
		\[\{-\alpha_2,\alpha_1,\alpha_2+2\alpha_1\}= s_2\{\alpha_2,\alpha_2+\alpha_1,\alpha_2+2\alpha_1\}=s_2[(\theta-\mathbb{Z}_{\geq 0}\Pi_{\{1\}})\cap\Delta].
		\]
		(2) Assume that $Y$ consists only of some long roots. Recall that $\pm\alpha_2,\pm(\alpha_2+2\alpha_1)$ are the four long roots in $\Delta$, so $Y\subset\{\pm\alpha_2,\pm(\alpha_2+2\alpha_1)\}$. By conjugating $Y$ with a suitable Weyl group element, we may assume that $\theta=\alpha_2+2\alpha_1\in Y$. Observe by Remark \ref{R5.2}(2) that we must have $|Y|\leq 2$, and also $-\alpha_2-2\alpha_1\notin Y$. If $|Y|=1$, then $Y=\{\theta\}$, and we are done. Otherwise we have two choices for $Y$: i)~$Y=\{\alpha_2,\alpha_2+2\alpha_1\}$, or ii)~$Y=\{-\alpha_2,\alpha_2+2\alpha_1\}$. The following implications show that both the cases i) and ii) lead to a short root belonging to $Y$, and hence a contradiction. 
		\begin{align*}
		\begin{aligned}
		(\alpha_2+2\alpha_1)+(\alpha_2)=2(\alpha_2+\alpha_1)&\ \implies \alpha_2+\alpha_1\in Y,\\
		(\alpha_2+2\alpha_1)+(-\alpha_2)=2(\alpha_1)&\ \implies \alpha_1\in Y.
		\end{aligned}
		\end{align*}
		Therefore, if $|Y|>1$, then $Y$ must contain a short root, in which case we are done by Step (1).\medskip\\
		\textbf{For type $G_2$:} Assume that $\Delta$ is of type $G_2$. Recall:
		\[
		\Delta=\big\{\pm(\alpha_1),\pm(\alpha_2),\pm(\alpha_2+\alpha_1),\pm(\alpha_2+2\alpha_1),\pm(\alpha_2+3\alpha_1),\pm(2\alpha_2+3\alpha_1)\big\}\text{ }\text{ and }\text{ }\theta=2\alpha_2+3\alpha_1.
		\] We proceed in several cases below to show that $Y$ is equal to $w[(\theta-\mathbb{Z}_{\geq 0}\Pi_J)\cap\Delta]$ for some $w\in W$ and $J\subseteq \mathcal{I}$, which implies the claim at the beginning. Suppose $Y$ contains a short root, say $u^{-1}\alpha_1$ for $u\in W$. Then by the 212-closedness of $uY$, the implications below lead in view of Remark \ref{R5.2}(2) to $uY=\Delta$, which contradicts $Y\subsetneqq\Delta$.
		\begin{align*}
		2(\alpha_1&)=(\alpha_2+2\alpha_1)+(-\alpha_2)=(\alpha_2+3\alpha_1)+(-\alpha_2-\alpha_1)\\
		&\implies -\alpha_2,-\alpha_2-\alpha_1,\alpha_2+2\alpha_1,\alpha_2+3\alpha_1\in uY. \\
		2(-\alpha_2-\alpha_1)=(-&\alpha_2-2\alpha_1)+(-\alpha_2)\implies -\alpha_2-2\alpha_1\in uY\implies \pm(\alpha_2+2\alpha_1)\in uY.
		\end{align*}
		So, we assume now that $Y$ consists only of long roots, which implies 
		\[
		Y\subset\big\{\pm(\alpha_2),\pm(\alpha_2+3\alpha_1),\pm(2\alpha_2+3\alpha_1)\big\}.\]
		By conjugating $Y$ with a suitable Weyl group element, we may assume that $\theta=2\alpha_2+3\alpha_1\in Y$. In view of Remark \ref{R5.2}(2), observe that $|Y|\leq 3$, and also $-2\alpha_2-3\alpha_1\notin Y$. Since $Y$ is 212-closed and $\alpha_1$ is short, the following show that both $-\alpha_2,-\alpha_2-3\alpha_1\not\in Y$:
		\begin{align*}
		(2\alpha_2+3\alpha_1)+(-\alpha_2)=(\alpha_2+2\alpha_1)+(\alpha_1)&\implies \alpha_2+2\alpha_1,\alpha_1\in Y.\\
		(2\alpha_2+3\alpha_1)+(-\alpha_2-3\alpha_1)=(\alpha_2+\alpha_1)+(-&\alpha_1)\implies \alpha_2+\alpha_1,-\alpha_1\in Y.   \end{align*}
		Similarly, the following implication shows (as $\alpha_2+\alpha_1$ is short) that $\{\alpha_2,\alpha_2+3\alpha_1\}\not\subset Y$.
		\[(\alpha_2)+(\alpha_2+3\alpha_1)=(\alpha_2+2\alpha_1)+(\alpha_2+\alpha_1)\implies \alpha_2+2\alpha_1,\alpha_2+\alpha_1\in Y.\]
		Therefore, $Y$ is equal to either $\{\alpha_2+3\alpha_1,2\alpha_2+3\alpha_1\}$ or $\{\alpha_2,2\alpha_2+3\alpha_1\}$. Now,
		\[
		\{\alpha_2+3\alpha_1,2\alpha_2+3\alpha_1\}=\big(\theta-\mathbb{Z}_{\geq 0}\Pi_{\{2\}}\big)\cap\Delta\quad\text{and} \quad \{\alpha_2,2\alpha_2+3\alpha_1\}=s_1\big[(\theta-\mathbb{Z}_{\geq 0}\Pi_{\{2\}})\cap\Delta\big].
		\]
		By the above two equations, notice that we are done. Hence, the proof of part (\ref{thmC}1) is complete.
	\end{proof}
	\begin{proof}[\textnormal{\textbf{Proof for part (\ref{thmC}2):}}] Assume that $\mathfrak{g}=\mathfrak{sl}_3(\mathbb{C})$. Fix a non-trivial subgroup $\mathbb{A}\subseteq(\mathbb{R},+)$. In this case, we show that there are more 212-closed subsets for $\Delta$ beyond the $W$-conjugates of the proper standard faces of $\Delta\sqcup\{0\}$. We also give a list of these additional 212-closed subsets of $\Delta$. We show that none of these additional 212-closed subsets are weak-$\mathbb{A}$-faces of $\Delta$ over any $\{0\}\subsetneqq \mathbb{A}\subseteq (\mathbb{R},+)$. So, the last sentence in Remark \ref{R7.2} and the previous line together imply that the weak-$\mathbb{A}$-faces of $\Delta$ are the same as the $W$-conjugates of the standard faces of $\Delta\sqcup\{0\}$. This shows that the three notions other than the proper 212-closed subsets of $\Delta$ are all the same as the $W$-conjugates of the standard faces of $\Delta\sqcup\{0\}$. To show the above claims, we begin by recalling
		\[
		\Delta=\big\{\pm\alpha_1,\pm\alpha_2,\pm(\alpha_2+\alpha_1)\big\}\text{ and }\theta=\alpha_1+\alpha_2. 
		\]
		Recall, all the $W$-conjugates of $\{\theta\}$---which are precisely all the singleton subsets of $\Delta$---are 212-closed in $\Delta\sqcup\{0\}$ as well as in $\Delta$. The following is a list of all the proper 212-closed subsets of $\Delta\sqcup\{0\}$---which are the same as all the $W$-conjugates of proper standard faces of $\Delta\sqcup\{0\}$---other than the singleton subsets.
		\begin{equation}\label{list}
		\{\alpha_2,\alpha_2+\alpha_1\}\quad\{-\alpha_2,\alpha_1\}\quad \{-\alpha_1,-\alpha_2-\alpha_1\}\quad\{\alpha_1,\alpha_2+\alpha_1\}\quad \{-\alpha_1,\alpha_2\}\quad \{-\alpha_2,-\alpha_2-\alpha_1\}.
		\end{equation}
		Let $\emptyset\neq Y$ be proper 212-closed in $\Delta$. We assume that $|Y|\geq 2$, as we already discussed the case of singleton subsets. By conjugating $Y$ by a suitable Weyl group element, we may assume that $Y$ contains $\theta=\alpha_2+\alpha_1$. As $|\Delta^+|=|\Delta^-|\leq 3$, Remark \ref{R5.2}(2) implies that $|Y|\leq 3$ and $-\alpha_2-\alpha_1\notin Y$. We now proceed in two cases below.\smallskip\\
		(1) Assume that $|Y|=2$. If either $\alpha_1$ or $\alpha_2$ fall in $Y$, then $Y$ is already in the list in \eqref{list}. Otherwise, $Y$ is either $\{-\alpha_1,\alpha_2+\alpha_1\}$ or $\{-\alpha_2,\alpha_2+\alpha_1\}$. Notice, these subsets are not in the list in \eqref{list}. The following implications show that $\{-\alpha_1,\alpha_2+\alpha_1\}$ is 212-closed in $\Delta$; similarly, it can be shown that $\{-\alpha_2,\alpha_2+\alpha_1\}$ is also 212-closed in $\Delta$. For any $x,y\in\Delta$: \begin{align}\label{E6.3}
		\begin{aligned}
		 &2(\alpha_2+\alpha_1)=(x)+(y)\implies x=y=\alpha_2+\alpha_1.\\
		&2(-\alpha_1)=(x)+(y)\implies x=y=-\alpha_1.\\
		(\alpha_2+\alpha_1)&+(-\alpha_1)=(x)+(y)\implies (x,y)=(\alpha_2+\alpha_1,-\alpha_1)\text{ or }(-\alpha_1,\alpha_2+\alpha_1).
		\end{aligned}
		\end{align}
		Note that the first two implications in \eqref{E6.3} above may be omitted as both the subsets $\{\alpha_2+\alpha_1\}$ and $\{-\alpha_1\}$ (being singleton) are 212-closed in $\Delta$.  
		Further, notice that \[s_1\{-\alpha_1,\alpha_2+\alpha_1\}=\Pi=s_2\{-\alpha_2,\alpha_2+\alpha_1\}.\] This means we have obtained six more 212-closed subsets of $\Delta$, which are precisely the six Weyl group translates of $\Pi$, outside the list in \eqref{list}.  Finally, the equation
		\begin{equation}\label{E6.4}
		2a(\alpha_1)+2a(\alpha_2)=2a(\alpha_2+\alpha_1)+a(\alpha_1)+a(-\alpha_1),\qquad\text{for }a\in \mathbb{A}_{>0}
		\end{equation}
		implies that $\Pi$ is not a weak-$\mathbb{A}$-face of $\Delta$, as both $\alpha_2+\alpha_1,-\alpha_1\notin\Pi$. Also, it can be easily seen that the weak-$\mathbb{A}$-faces of $\Delta$ of size two are precisely all the subsets in the list in \eqref{list}.\smallskip\\
		(2) Assume that $|Y|=3$. Recall, $\alpha_2+\alpha_1\in Y$, and so $-\alpha_2-\alpha_1\notin Y$. The previous line, Remark~\ref{R5.2}(2), and $Y\subsetneqq\Delta$ together imply that we have four choices for $Y$:
		\begin{equation}\label{list2}
		\Delta^+,\quad \{-\alpha_1,\alpha_2,\alpha_2+\alpha_1\}=s_1\Delta^+,\quad \{-\alpha_2,\alpha_1,\alpha_2+\alpha_1\}=s_2\Delta^+,\quad \{-\alpha_1,-\alpha_2,\alpha_2+\alpha_1\}.
		\end{equation}
		Observe that every size two subset of each set in the list in \eqref{list2} is 212-closed in $\Delta$. Thus by the definition of a 212-closed subset, every set in the list in \eqref{list2} is 212-closed in $\Delta$. Check:
		\[u\Delta^+\neq \{-\alpha_1,-\alpha_2,\alpha_2+\alpha_1\} \text{ for any}\text{ }u\in W,\quad\text{as } \sum_{x\in \{-\alpha_1,-\alpha_2,\alpha_2+\alpha_1\}}x=0.\]
		Therefore, all the $W$-conjugates of the two sets $\Delta^+\text{ and } \{-\alpha_1,-\alpha_2,\alpha_2+\alpha_1\}$ are precisely all the size three 212-closed subsets of $\Delta$. Finally, once again by equation \eqref{E6.4}, neither of $\Delta^+$ and $\{-\alpha_1,-\alpha_2,\alpha_2+\alpha_1\}$, and so none of their $W$-conjugates, can be weak-$\mathbb{A}$-faces of $\Delta$ for any $\{0\}\subsetneqq \mathbb{A}\subseteq (\mathbb{R},+)$. Hence, the proof of part (\ref{thmC}2) is complete. 
	\end{proof}	
	\begin{proof}[\textnormal{\textbf{Proof for part (\ref{thmC}3):}}]
		Let $\Delta$ be an affine root system, and $\{0\}\subsetneqq \mathbb{A}\subseteq (\mathbb{R},+)$. We prove this part in four steps. Steps 1 and 2 show the equivalence of statements a) and b) in part (\ref{thmC}3) for the 212-closed subsets of $\Delta\sqcup\{0\}$. Replacing $\Delta\sqcup\{0\}$ and $\mathring{\Delta}\sqcup\{0\}$ by $\Delta$ and $\mathring{\Delta}$, respectively, everywhere in Steps 1 and 2 shows the same result for the 212-closed subsets of $\Delta$. Step 3 shows that the 212-closed subsets and weak-$\mathbb{A}$-faces of $\Delta\sqcup\{0\}$ are the same. Step 4 discusses the equivalence of all the four notions.
		\medskip\\
		\textbf{Step 1: a) $\implies$ b)}. Let $\Delta$ be of type $X_{\ell}^{(r)}$, $r\in [3]$ and $\ell\in\mathbb{N}$---see \cite[Tables Aff 1--Aff 3]{Kac}. Fix $\emptyset\neq Y$ to be proper 212-closed in $\Delta\sqcup\{0\}$. Recall the definitions of $Y_s$ and $Y_l$ from \eqref{E5.1}. By the definition of 212-closed subsets, it can be easily seen that $Y\cap(\mathring{\Delta}\sqcup\{0\})$ is proper 212-closed in $\mathring{\Delta}\sqcup\{0\}$. Therefore, it only remains to show that $Y= Y_s\cup Y_l$. Fix a root $\beta\in Y$. Note in view of Remark \ref{R5.2} that $Y\subset\Delta^{re}$, so by \cite[Proposition 6.3]{Kac} $\beta$ has one of the following forms: 	  \begin{itemize}
			\item[(i)] $\beta=q\delta+\eta_1$ for some $\eta_1\in\mathring{\Delta}_s$ and $q\in\mathbb{Z}$.
			\item[(ii)] $\beta=rq\delta+\eta_2$ for some $\eta_2\in\mathring{\Delta}_l$ and $q\in\mathbb{Z}$.
			\item[(iii)] $\beta=\frac{1}{2}[(2q-1)\delta+\eta_2]$ for some $\eta_2\in\mathring{\Delta}_l$ and $q\in\mathbb{Z}$ (only when $\Delta$ is of type $A_{2\ell}^{(2)}$).
		\end{itemize} 
		Observe that (iii) cannot hold; since by the 212-closedness of $Y$, \[2\left(\frac{1}{2}\Big[(2q-1)\delta+\eta_2\Big]\right)=(\eta_2)+\big((2q-1)\delta\big)\quad\implies\quad (2q-1)\delta\in Y\quad\Rightarrow\!\Leftarrow\quad Y\subset\Delta^{re}.\]
		Now if (i) holds, then the equation $2(q\delta+\eta_1)=(\eta_1)+(2q\delta+\eta_1)$ results in $\eta_1\in Y\cap \mathring{\Delta}_s$. Similarly, if (ii) holds, then the equation $2(rq\delta+\eta_2)=(\eta_2)+(2rq\delta+\eta_2)$ results in $\eta_2\in Y\cap\mathring{\Delta}_l$. These two lines immediately imply that $Y\subseteq Y_s\cup Y_l$, as $\beta\in Y$ is arbitrary. Now we prove the reverse inclusion in the previous line. Note by the definitions of $Y_s$ and $Y_l$ that $Y\neq\emptyset$ implies $Y_s\cup Y_l\neq\emptyset$, which in particular implies $Y\cap\mathring{\Delta}\neq \emptyset$. Assume that $\eta\in Y\cap \mathring{\Delta}$. Then for any $m\in\mathbb{Z}$, the equations \[2(\eta)=(m\delta+\eta)+(-m\delta+\eta)\quad\text{ or resp. }\quad 2(\eta)=(rm\delta+\eta)+(-rm\delta+\eta),\]
		depending on whether $\eta$ is short or respectively long in $\mathring{\Delta}$, imply that $Y_s\cup Y_l\subseteq Y$. So, $Y= Y_s\cup Y_l$.\medskip\\
		\textbf{Step 2: b) $\implies$ a)}. In this step, we show that every 212-closed subset of $\mathring{\Delta}\sqcup\{0\}$ gives rise to a unique 212-closed subset of $\Delta\sqcup\{0\}$. Let $Z$ be 212-closed in $\mathring{\Delta}\sqcup\{0\}$. The result is trivial when $Z=\emptyset$ or $\mathring{\Delta}\sqcup\{0\}$. So we assume for the rest of this step that $\emptyset\neq Z\subsetneqq \mathring{\Delta}\sqcup\{0\}$. We define $Z_s\text{ and }Z_l$ similar to $Y_s\text{ and }Y_l$, respectively, as in \eqref{E5.1}. To show $Z_s\cup Z_l$ is 212-closed in $\Delta\sqcup\{0\}$,~we~prove:
		\begin{equation}\label{E5.4} (\gamma_1)+(\gamma_2)=(\gamma_3)+(\gamma_4)\text{ }\text{ for some }\gamma_1,\gamma_2\in Z_s\cup Z_l\text{ and }\gamma_3,\gamma_4\in \Delta\sqcup\{0\}\text{ }\implies\text{ } \gamma_3,\gamma_4\in Z_s\cup Z_l.
		\end{equation}
		By the choice of $\gamma_1\text{ and }\gamma_2$, and also as each of $\gamma_3$ and $\gamma_4$ either belongs to $\mathbb{Z}\delta$ or is of one of the forms in (i)--(iii) in Step 1, we can uniformly express $\gamma_1,\ldots,\gamma_4$ as follows:
		\begin{equation}\label{E7.9} \gamma_i=m_i\delta+\xi_i
		\quad
		\text{where } \xi_1,\xi_2\in Z,\text{ }\text{ } \xi_3,\xi_4\in\mathring{\Delta}\sqcup\{0\}\sqcup\frac{1}{2}\mathring{\Delta}_l,\text{ }\text{ } m_1,m_2\in \mathbb{Z},\text{ and } m_3,m_4\in\frac{1}{2}\mathbb{Z}.\end{equation}
		Now, observe:
		\begin{itemize}
			\item For $t\in\{ 3,4\}$, $\xi_t=0$ if and only if $ \gamma_t\in\mathbb{Z}\delta$. Similarly, $\xi_t\in\frac{1}{2}\mathring{\Delta}_l$ if and only if $ m_t=\frac{2q-1}{2}\text{ for some }q\in\mathbb{Z}$, if and only if $ \gamma_t$ is of the form in (iii) in Step 1.
			\item $m_1+m_2=m_3+m_4$, as $\height_{\{0\}}(\gamma_1+\gamma_2)=\height_{\{0\}}(\gamma_3+\gamma_4)$. Recall, 0 is a node in the Dynkin diagram, and also $\mathcal{I}\setminus\{0\}$ is the set of vertices in the Dynkin subdiagram for $\mathring{\Delta}$. \end{itemize}
		So, on plugging-in $m_i\delta+\xi_i$ in place of $\gamma_i$ $\forall$ $i\in[4]$ in the equation in \eqref{E5.4}, and then subtracting $(m_1+m_2)\delta$ from both sides of the resulting equation, we get
		\begin{equation}\label{E5.5}
		(\xi_1)+(\xi_2)=(\xi_3)+(\xi_4). \end{equation}
		We now proceed in two cases below.\smallskip\\ 
		(1) $\Delta$ is not of type $A_{2\ell}^{(2)}$: Every real root in $\Delta$ must be of the form either in (i) or (ii) in Step~1. In particular, $\xi_3,\xi_4\in \mathring{\Delta}\sqcup\{0\}$. As $Z$ is proper 212-closed in $\mathring{\Delta}\sqcup\{0\}$, equation \eqref{E5.5} implies $\xi_3,\xi_4\in Z\cap\mathring{\Delta}$. Hence, $\gamma_3,\gamma_4\in Z_s\cup Z_l$. This proves $Z_s\cup Z_l$ is 212-closed in $\Delta\sqcup\{0\}$ in this~case.\smallskip\\
		(2) $\Delta$ is of type $A_{2\ell}^{(2)}$: If $\xi_3,\xi_4\in\mathring{\Delta}\sqcup\{0\}$, then we are done by the previous case. So, we assume without loss of generality that $\xi_3\in\frac{1}{2}\mathring{\Delta}_l$, and show that this assumption contradicts $Z\subsetneqq \mathring{\Delta}\sqcup\{0\}$. Recall, $\mathring{\Delta}$ is of type either $C_{\ell}$ when $\ell\geq 2$, or $A_{1}$ when $\ell=1$. Let the set of nodes in the Dynkin diagram for $\mathring{\Delta}$ be $\{1,\ldots,\ell\}$. When $\mathring{\Delta}$ is of type $C_{\ell}$, let $\ell$ correspond to the unique long simple root of $\mathring{\Delta}$. When $\mathring{\Delta}$ is of type $A_1$, $\alpha_{1}$ must be treated as the long (as well as the short) simple root, so that we get the roots of the form in (iii) in Step 1 inside $\Delta$---see \cite[Proposition 6.3]{Kac}. Observe that the long roots of $\mathring{\Delta}$ can be expressed as $R_{\pm k}$ for some $k\in [\ell]$, where
		\[
		R_{\pm k}:=\pm\left(-\alpha_{\ell}+\sum\limits_{j=k}^{\ell}2\alpha_j\right)\qquad \forall\text{ }k\in[\ell].\]
		In the above notation, let $\xi_3=\frac{1}{2}R_{k_1}$ for some $k_1\in [-\ell,\ell]\setminus\{0\}$. Observe that \[\height_{\{\ell\}}(\xi_3)=\text{sign}(k_1)\frac{1}{2},\quad\text{ and }\quad\height_{\{\ell\}}(\xi_3+\xi_4)=\height_{\{\ell\}}(\xi_1+\xi_2)\in\mathbb{Z}.\] This implies $\xi_4\in\frac{1}{2}\mathring{\Delta}_l$. Let $\xi_4=\frac{1}{2}R_{k_2}$ for some $k_2\in [-\ell,\ell]\setminus\{0\}$. Now, one easily observes that $\xi_3+\xi_4$ is either a root or 0, by checking that $\xi_3+\xi_4$ equals:
		\begin{align*}
		1)\text{ }&0\text{ }(\text{when }k_1=-k_2);\quad & 3)\text{ } R_{k_2}+\text{sign}(k_2)\sum_{j=|k_1|}^{|k_2|-1}\alpha_j\text{ }\big(\text{when }\text{sign}(k_1)=\text{sign}(k_2)\text{ and }|k_1|<|k_2|\big);\\
		2)\text{ }&R_{ k_1}\text{ }(\text{when }k_1=k_2);\quad
		& 4)\text{ }\text{sign}(k_1)\sum_{j=|k_1|}^{|k_2|-1}\alpha_j\text{ }\big(\text{when }\text{sign}(k_1)\neq \text{sign}(k_2)\text{ and }|k_1|<|k_2|\big).\hspace{1cm}    
		\end{align*}
		Notice when $\mathring{\Delta}$ is of type $A_{1}$ that only 1) is possible as $\mathring{\Delta}=\{\pm\alpha_{\ell}\}$. When $\mathring{\Delta}$ is of type $C_{\ell}$, note that the roots of the form in 2) are long (as stated earlier), and those of the forms in 3) and 4) are short. If $\xi_3+\xi_4=0$, then \eqref{E5.3} and equation \eqref{E5.5} lead to $Z=\mathring{\Delta}\sqcup\{0\}$, contradicting $Z\subsetneqq \mathring{\Delta}\sqcup\{0\}$. This finishes the proof of this step when $\mathring{\Delta}$ is of type $A_1$. So, we assume for the rest of the proof that $\ell>1$ and $\mathring{\Delta}$ is of type $C_{\ell}$. Suppose $\xi_3+\xi_4$ is a long root. Pick a Weyl group element $u_1\in \mathring{W}$ such that $u_1(\xi_3+\xi_4)=\alpha_{\ell}$. Note, $u_1Z$ is proper 212-closed in $\mathring{\Delta}\sqcup\{0\}$. Applying $u_1$ to the terms on either side of equation \eqref{E5.5} results in $(u_1\xi_1)+(u_1\xi_2)=\alpha_{\ell}$. In the Dynkin diagram, assume that the node $\ell-1$ is adjacent to $\ell$ (recall $\ell>1$). Observe that the following implications lead by the 212-closedness of $u_1Z$ to the desired contradiction.
		\begin{align*}
		(u_1\xi_1)+(u_1\xi_2)=\alpha_{\ell}=(\alpha_{\ell}+\alpha_{\ell-1})+(-\alpha_{\ell-1})&\implies \alpha_{\ell}+\alpha_{\ell-1},-\alpha_{\ell-1}\in u_1Z.\\
		2(\alpha_{\ell}+\alpha_{\ell-1})=(\alpha_{\ell}+2\alpha_{\ell-1})+(\alpha_{\ell})&\implies \alpha_{\ell}+2\alpha_{\ell-1},\alpha_{\ell}\in u_1Z.\\
		2(-\alpha_{\ell-1})=(-\alpha_{\ell}-2\alpha_{\ell-1})+(\alpha_{\ell})&\implies -\alpha_{\ell}-2\alpha_{\ell-1}\in u_1Z.\\
		\pm(\alpha_{\ell}+2\alpha_{\ell-1})\in u_1Z\implies u_1Z=&\mathring{\Delta}\sqcup\{0\}\Rightarrow\!\Leftarrow Z\subsetneqq \mathring{\Delta}\sqcup\{0\}.
		\end{align*}
		Finally, if $\xi_3+\xi_4$ is a short root in $\mathring{\Delta}$, similar to above, pick a Weyl group element $u_2\in \mathring{W}$ such that $u_2(\xi_3+\xi_4)=\alpha_{\ell-1}$, so that $(u_2\xi_1)+(u_2\xi_2)=\alpha_{\ell-1}$. Observe that the following implications lead by the 212-closedness of $u_2 Z$ to the desired contradiction.
		\begin{align*}            (u_2\xi_1)+(u_2\xi_2)=\alpha_{\ell-1}=(\alpha_{\ell}+\alpha_{\ell-1})+(-\alpha_{\ell})&\implies \alpha_{\ell}+\alpha_{\ell-1},-\alpha_{\ell}\in u_2Z.\\  2(\alpha_{\ell}+\alpha_{\ell-1})=(\alpha_{\ell}+2\alpha_{\ell-1})+(\alpha_{\ell})&\implies \alpha_{\ell}+2\alpha_{\ell-1},\alpha_{\ell}\in u_2Z.\\
		\pm \alpha_{\ell}\in u_2Z\implies u_2Z=\mathring{\Delta}\sqcup\{0\}&\Rightarrow\!\Leftarrow Z\subsetneqq \mathring{\Delta}\sqcup\{0\}.
		\end{align*}
		\textbf{Step 3}. In this step, we show that the 212-closed subsets and weak-$\mathbb{A}$-faces of $\Delta\sqcup\{0\}$ are the same. Let $\emptyset\neq Y$ be proper 212-closed in $\Delta\sqcup\{0\}$. By the implication a) $\implies$ b) in (\ref{thmC}3), $Y=Y_s\cup Y_l$, and $Y\cap \mathring{\Delta}$ is proper 212-closed in $\mathring{\Delta}\sqcup\{0\}$. So, by parts (\ref{thmC}1) and (\ref{thmC}2), $Y\cap\mathring{\Delta}$ is a proper weak-$G$-face of $\mathring{\Delta}\sqcup\{0\}$ for any subgroup $\{0\}\subsetneqq G\subseteq(\mathbb{R},+)$; in particular for $G=\mathbb{A},\frac{1}{2}\mathbb{A}\text{ and }\frac{1}{4}\mathbb{A}$. Here, for $r\in\mathbb{N}$, $\frac{1}{r}\mathbb{A}$ denotes the additive subgroup $\left\{\frac{a}{r}\text{ }\big|\text{ }a\in \mathbb{A}\right\}$ of $\mathbb{R}$. To show that $Y$ is a weak-$\mathbb{A}$-face of $\Delta\sqcup\{0\}$, assume that the following hold.
		\begin{align*}
		\begin{aligned}[c]
		\sum_{p=1}^n a_p y_p=\sum_{q=1}^m b_q x_q\text{ }\text{ and }\text{ }\sum_{p=1}^n a_p =\sum_{q=1}^m b_q>0 
		\end{aligned}\quad\quad
		\begin{aligned}[c]
		&\text{where }y_p\in Y,\text{ }x_q\in\Delta\sqcup\{0\},\text{ }a_p,b_q\in\mathbb{A}_{> 0},\\\
		& n,m\in\mathbb{N}\text{ }\forall\text{ }1\leq p\leq n\text{ and }1\leq q\leq m.
		\end{aligned}
		\end{align*}
		We prove that $x_q\in Y$ $\forall$ $q\in [m]$, which implies the result. Recall firstly that $x_q$ can be either of any of the forms in (i)--(iii) in Step 1, or it belongs to $\mathbb{Z}\delta$; whereas $y_p$ can be of the forms only in (i) and (ii) as $Y\subsetneqq\Delta\sqcup\{0\}$. Recall the procedure in Step 2 by which we got rid of $\delta$ from the terms in equation \eqref{E5.4} \big(upon re-writing as in \eqref{E7.9}\big) to get equation \eqref{E5.5}, which consists of terms from $\mathring{\Delta}\sqcup\{0\}\sqcup\frac{1}{2}\mathring{\Delta}$. By the same procedure, we can clear off $\delta$ from both sides of the equation $\sum_{p=1}^n a_p y_p=\sum_{q=1}^m b_q x_q$ \big(upon re-writing $y_p$ and $x_q$ as in \eqref{E7.9}\big), to obtain the following equation:
		\[\sum_{p=1}^n a_p y'_p=\sum_{q=1}^m b_q x'_q\quad\text{where }y'_p\in Y\cap\mathring{\Delta},\text{ }x'_q\in \mathring{\Delta}\sqcup\{0\}\sqcup\frac{1}{2}\mathring{\Delta}\text{ are such that }y'_p\in \mathbb{Z}\delta+y_p,\text{ }x'_q\in \frac{1}{2}\mathbb{Z}\delta+x_q\]
		for all $p\in [n]$ and $q\in [m]$. For each $q\in [m]$, to work simultaneously with the two cases where $x'_q\in\frac{1}{2}\mathring{\Delta}$ and $x'_q\notin\frac{1}{2}\mathring{\Delta}$, we define for convenience,
		\[\widehat{b}_q:=
		\begin{cases}
		b_q \text{ }&\text{ if }x'_q\notin\frac{1}{2}\mathring{\Delta},\\
		\frac{b_q}{2}\text{ }&\text{ if }x'_q\in\frac{1}{2}\mathring{\Delta}, 
		\end{cases}\qquad\text{and}\qquad
		\widehat{x}_q:=
		\begin{cases}
		x'_q \text{ }&\text{ if }x'_q\notin\frac{1}{2}\mathring{\Delta},\\
		2x'_q\text{ }&\text{ if }x'_q\in\frac{1}{2}\mathring{\Delta}. 
		\end{cases}
		\]
		Notice, $\widehat{b}_q\in\frac{1}{2}\mathbb{A}$, $\widehat{b}_q\leq b_q$, $\widehat{x}_q\in\mathring{\Delta}\sqcup\{0\}$, and $b_qx'_q=\widehat{b}_q\widehat{x}_q$, $\forall$ $q$. Fix a root $\xi\in\mathring{\Delta}$, and observe that
		\begin{equation}\label{E6.1}
		\sum_{p=1}^n a_p y'_p=\sum_{q=1}^m \widehat{b}_q \widehat{x}_q+\frac{1}{2}\left(\sum_{q=1}^m b_q-\sum_{q=1}^m \widehat{b}_q\right)\big(\xi\big)+\frac{1}{2}\left(\sum_{q=1}^{m}b_q-\sum_{q=1}^m \widehat{b}_q\right)\big(-\xi\big)\end{equation}
		implies $\widehat{x}_q\in Y\cap \mathring{\Delta}$ $\forall$ $q\in [m]$, as $\widehat{b}_q>0$, and as $Y\cap \mathring{\Delta}$ is a proper weak-$\mathbb{A}$-face of $\mathring{\Delta}\sqcup\{0\}$. In particular, $\widehat{x}_q\neq 0$ $\forall$ $q$. Now we closely look at $\widehat{x}_q$. If $\widehat{x}_q=x'_q$ (which happens when $x_q\in\mathring{\Delta}$, and which implies $\widehat{b}_q=b_q$) $\forall$ $q\in [m]$, then the definitions of $x'_q$, $Y_s$, $Y_l$ imply $x_q\in Y_s\cup Y_l=Y$ $\forall$ $q\in [m]$ as desired,  and we are done. Else if there exists some $r\in [m]$ such that $x'_r\in \frac{1}{2}\mathring{\Delta}$, then firstly $\left(\sum_{q=1}^m b_q-\sum_{q=1}^m \widehat{b}_q\right)>0$. As $Y\cap\mathring{\Delta}$ is also a weak-$\frac{1}{4}\mathbb{A}$-face of $\mathring{\Delta}\sqcup\{0\}$, equation \eqref{E6.1} implies that $\pm \xi\in Y\cap\mathring{\Delta}$. But by Remark \ref{R5.2}(2) this leads to $\mathring{\Delta}\subset Y$, which by the 212-closedness of $Y$ further leads to $Y=\Delta\sqcup\{0\}$, contradicting $Y\subsetneqq \Delta\sqcup\{0\}$. This completes the proof of the claim at the beginning of this step. \medskip\\
		\textbf{Step 4}. By Step 3, the proper 212-closed subsets of $\Delta\sqcup\{0\}$ and proper weak-$\mathbb{A}$-faces of $\Delta\sqcup\{0\}$ are the same. Suppose $Y$ is a proper weak-$\mathbb{A}$-face of $\Delta$. Then $Y\cap\mathring{\Delta}$ is a proper weak-$\mathbb{A}$-face of $\mathring{\Delta}$, and since $\mathring{\Delta}$ is of finite type, $Y\cap\mathring{\Delta}$ is proper 212-closed in $\mathring{\Delta}\sqcup\{0\}$ by parts (\ref{thmC}1) and (\ref{thmC}2). Now, as $Y$ is also 212-closed in $\Delta$---recall, weak-$\mathbb{A}$-faces of $\Delta$ are always 212-closed---the implication a) $\implies$ b) results in $Y=Y_s\cup Y_l$. The previous two sentences in view of b) $\implies$ a) prove that $Y$ is proper 212-closed in $\Delta\sqcup\{0\}$. Recall, the proper weak-$\mathbb{A}$-faces of $\Delta\sqcup\{0\}$ are also proper weak-$\mathbb{A}$-faces of $\Delta$. By Step 3, the weak-$\mathbb{A}$-faces of $\Delta\sqcup\{0\}$ are the same as the 212-closed subsets of $\Delta\sqcup\{0\}$. The previous three sentences together imply that the following three notions are all the same: proper weak-$\mathbb{A}$-faces of $\Delta$, proper weak-$\mathbb{A}$-faces of $\Delta\sqcup\{0\}$, and proper 212-closed subsets of $\Delta\sqcup\{0\}$. This also proves the last assertion in the statement of part (\ref{thmC}3) that each of these three sets is 212-closed in $\Delta$ (as weak-$\mathbb{A}$-faces of $\Delta$ are 212-closed in $\Delta$). Finally, we are left to show when $\mathfrak{g}\neq\widehat{\mathfrak{sl}_3(\mathbb{C})}$ that all the four notions are the same. For this let $\emptyset\neq Z$ be proper 212-closed in $\Delta$. Now a) $\implies$ b) results in $Z=Z_s\cup Z_l$ and $Z\cap\mathring{\Delta}$ is proper 212-closed in $\mathring{\Delta}$. By part (\ref{thmC}1) \big(since $\mathfrak{g}_{\{0\}^c}\neq \mathfrak{sl}_3(\mathbb{C})$, see \cite[Tables Aff 1--Aff 3]{Kac}\big), $Z\cap\mathring{\Delta}$ is also proper 212-closed in $\mathring{\Delta}\sqcup\{0\}$. By the implication b)~$\implies$~a), $Z$ is proper 212-closed in $\Delta\sqcup\{0\}$. Hence, the proof of part (\ref{thmC}3) is complete.       \end{proof}
	\begin{proof}[\textnormal{\textbf{Proof for part (\ref{thmC}4):}}]
		Assume that $\Delta$ is of indefinite type. We show that $\Delta$ is the only 212-closed subset of $\Delta$; the proof for $\Delta\sqcup\{0\}$ is the only 212-closed subset of $\Delta\sqcup\{0\}$ is very similar. From this, it also follows that $\Delta$ (respectively, $\Delta\sqcup\{0\}$) is the only weak-$\mathbb{A}$-face of $\Delta$ (respectively, of $\Delta\sqcup\{0\}$), and hence all the four notions are the same. Let $\emptyset\neq Y$ be 212-closed in $\Delta$, and fix a root $\beta\in Y$. Note by Remark \ref{R5.2} that $\beta$ must be real. Using $\beta$, we show that $Y$ contains an imaginary root, so that $Y=\Delta$ once again by Remark \ref{R5.2}. Fix an imaginary root $\gamma\in\Delta$ with the property that $\langle\gamma,\alpha_t^{\vee}\rangle<0$ $\forall$ $t\in\mathcal{I}$; such a root exists by \cite[Theorem 5.6 part C]{Kac}. Pick $w\in W$ such that $w\beta\in\Pi$, and correspondingly pick a simple root $\alpha$ such that $\langle w\beta,\alpha^{\vee}\rangle<0$. Such a simple root $\alpha$ exists as $|\mathcal{I}|\geq 2$ and the Dynkin diagram is connected. Note that $\langle\gamma,\alpha^{\vee}\rangle<0$ and $s_{\alpha}(\gamma)$ is imaginary. Now,  \begin{align*}\begin{aligned}\langle\gamma,w\beta^{\vee}\rangle<0\text{ and }\langle\alpha,w\beta^{\vee}\rangle<0&\implies \langle s_{\alpha}(\gamma), w\beta^{\vee}\rangle\leq -2\\
		&\implies s_{\alpha}(\gamma),\text{ }s_{\alpha}(\gamma)+w\beta,\text{ }s_{\alpha}(\gamma)+2w\beta\in \Delta.\end{aligned}\end{align*} 
		\[\text{(By the 212-closedness of } wY)\qquad 2(w\beta)=\big(s_{\alpha}(\gamma)+2w\beta\big)+\big(-s_{\alpha}(\gamma)\big)\implies -s_{\alpha}(\gamma)\in wY.\] 
		Therefore, $Y$ contains the imaginary root $-w^{-1}s_{\alpha}(\gamma)$, as required.
	\end{proof}
	Finally, we provide an answer to Problem \ref{prob1} mentioned at the end of Section 2.
	\subsection{Answer to Problem \ref{prob1}}
	Observe that it suffices to address the problem when $\Delta$ is of finite type, say with highest root $\theta$. Assuming this, given $Z\subseteq \Delta$ and $\alpha\in \Delta$, first define for convenience
	\[W_Z^+:=\{w\in W\text{ }|\text{ }wZ\subseteq \Delta^+\}\quad \text{ and }\quad W_{\alpha}^+:=\{w\in W\text{ }|\text{ } w\alpha\in\Delta^+\}.\]
	Fix $J\subseteq \mathcal{I}$ and $Y=(\theta-\mathbb{Z}_{\geq 0}\Pi_J)\cap\Delta$. Notice that $Y=\wt L_J(\theta)$; $L_J(\theta)$ is the simple highest weight $\mathfrak{g}_J$-module with highest weight $\theta$. Our aim is to prove that  
	\begin{equation}\label{E1}
	W_Y^+=\bigcap\limits_{w\in W_J} W_{\theta}^+w.
	\tag{$\star$}
	\end{equation}
	This shows that the subsets of $\Delta$ of the form $\omega\big[(\theta-\mathbb{Z}_{\geq 0}\Pi_I)\cap\Delta\big]$, for all $I\subseteq\mathcal{I}$ and $\omega\in W_{(\theta-\mathbb{Z}_{\geq 0}\Pi_I)\cap\Delta}$, are precisely all the 212-closed subsets of $\Delta$ that are contained in $\Delta^+$. 
	Before we proceed, observe:
	\begin{itemize}
		\item $\{\theta\}$ is 212-closed in $\Delta$. So, one has to compute $W_{\theta}^+$ in the simplest case $Y=\{\theta\}$.
		\item $W_{\theta}^+$ is precisely the set of minimal length coset representatives of $W/\{1,s_{\theta}\}$, where $1$ is the identity element in $W$ and $s_{\theta}$ is the reflection about the hyperplane perpendicular to $\theta$. 
		\item $W_JY=Y$. So, $W_J\subseteq W_Y^+$. Note when $J=\mathcal{I}$ that $Y=\Delta$, and so $W_Y^+=\emptyset$.\allowdisplaybreaks
		\item Assume that $\emptyset\neq J\subsetneqq \mathcal{I}$. $W_Y^+=W_J\iff |J^c|=1$ and $\height_{J^c}(\theta)=1$. Recall for $\mathfrak{g}$ of types $G_2, F_4, E_6, E_7\text{ and }E_8$ that $\height_{\{i\}}(\theta)>1$ $\forall$ $i\in\mathcal{I}$. So for $\mathfrak{g}$ of these types, $W_J \subsetneqq W_Y^+$.
	\end{itemize}
	To prove \eqref{E1}, recall the standard result $\conv_{\mathbb{R}}\wt L_J(\theta)=\conv_{\mathbb{R}} (W_J \theta)$, and so $\conv_{\mathbb{R}}Y=\conv_{\mathbb{R}}(W_J \theta)$. Therefore: 
	\[W_Y^+=W_{W_J\theta}^+=\bigcap\limits_{w\in W_J} W_{w\theta}^+=\bigcap\limits_{w\in W_J} W_{\theta}^+w\qquad \text{as }W_{w\theta}^+=W_{\theta}^+w^{-1}\text{ }\forall
	\text{ }w\in W_J. \]
	This finishes the proof of \eqref{E1}.	More generally, it can be shown that 
	\begin{itemize} 
		\item $Y= W_J\theta$ if $Y$ contains no short root. 
		\item If $Y$ contains a short root, then $Y=W_J\theta \sqcup W_J\theta_s$, where $\theta_s$ is the highest short root in $\Delta$.
	\end{itemize}
	With all of the above, we conclude the paper here.
	\subsection*{Acknowledgements} The author extends his hearty thanks to his guide, Apoorva Khare, for his invaluable guidance, for introducing the author to the various combinatorial subsets studied in this article, for stimulating and fruitful discussions, and for feedback on preliminary drafts that helped improve the exposition. This work is supported by a scholarship from the National Board for Higher Mathematics (Ref. No. 2/39(2)/2016/NBHM/R{\&}D-II/11431).
	
	\address{\textsc{R-21, Department of Mathematics, Indian Institute of Science, Bangalore 560012, India}}
	
	\textit{E-mail address}: \email{\texttt{tejag@iisc.ac.in}}
	
\end{document}